\documentclass[11pt,reqno]{amsart}
\usepackage[T1]{fontenc}
\usepackage{graphicx, amsbsy, amssymb, amsmath, stmaryrd, mathrsfs}
\usepackage{cite}
\usepackage{color}
\definecolor{MyLinkColor}{rgb}{0,0,0.4}

\newcommand{\R}{\mathbb{R}}

\newcommand{\bD}{\mathbb{D}}
\newcommand{\bB}{\mathbb{B}}
\newcommand{\bT}{\mathbb{T}}

\newcommand{\sfe}{\mathsf{e}}

\newcommand{\Z}{{\mathbb Z}}
\newcommand{\N}{{\mathbb N}}

\newcommand{\sfB}{\mathsf{B}}

\newcommand{\sfT}{\mathsf{T}}
\newcommand{\sfA}{\mathsf{A}}
\newcommand{\sfH}{\mathsf{H}}

\newcommand{\kH}{\mathcal{H}}
\newcommand{\cO}{\mathcal{O}}

\newcommand{\kL}{\mathcal{L}}

\newcommand{\cV}{\mathcal{V}}
\newcommand{\cI}{\mathcal{I}}

\newcommand{\wh}{\widehat}
\newcommand{\wt}{\widetilde}

\newcommand{\re}{\mathop{\rm Re}\nolimits}

\newcommand{\PV}{\mathop{\rm PV}\nolimits}
\newcommand{\ov}{\overline}
\newcommand{\p}{\partial}

\newcommand{\e}{\varepsilon}
\newcommand{\0}{\Omega}
\newcommand{\G}{\Gamma}

\newcommand{\supp}{\mathop{\rm supp}\nolimits}
\newcommand{\mn}{{\mathrm n}}
\newcommand{\mt}{{\mathrm t}}

\newtheorem{thm}{Theorem}[section]
\newtheorem{prop}[thm]{Proposition}
\newtheorem{lemma}[thm]{Lemma}

\theoremstyle{remark} 
\newtheorem{rem}[thm]{Remark}

\numberwithin{equation}{section} 
\usepackage[colorlinks=true,linkcolor=MyLinkColor,citecolor=MyLinkColor]{hyperref}

\begin{document}

\title[The Hele-Shaw Problem via potential theory]{A potential theory approach to the capillarity-driven Hele-Shaw problem}

\author{Bogdan--Vasile Matioc}
\address{Fakult\"at f\"ur Mathematik, Universit\"at Regensburg,   93053 Regensburg, Deutschland.}
\email{bogdan.matioc@ur.de}

\author{Christoph Walker}
\address{Leibniz Universit\"at Hannover\\
Institut f\"ur Angewandte Mathematik\\
Welfengarten 1\\
30167 Hannover\\
Germany}
\email{walker@ifam.uni-hannover.de}

\thanks{Partially supported by the RTG 2339
 ``Interfaces, Complex Structures, and Singular Limits'' of the German Science Foundation (DFG)}

\begin{abstract}
  In this paper, we demonstrate that potential theory  provides a powerful  framework for analyzing quasistationary fluid flows in bounded geometries, 
  where the bulk dynamics are governed by elliptic equations with constant coefficients. 
  This approach is illustrated by the two-dimensional Hele-Shaw problem with surface tension, 
  for which we derive local well-posedness and parabolic smoothing in (almost) optimal function spaces. 
 In addition, we establish a generalized principle of linearized stability for a particular class of abstract quasilinear parabolic problems, which  enables us to show that the stationary solutions
   to the Hele-Shaw problem are exponentially stable.
\end{abstract}

\subjclass[2020]{35B35; 35K59; 35R37; 76D27}
\keywords{Hele-Shaw Problem; Surface tension; Singular integrals; Subcritical spaces; Stability}

\maketitle

\pagestyle{myheadings}
\markboth{\sc{B.-V.~Matioc \& Ch. Walker }}{}

 \section{Introduction  and  main  results}\label{SEC:1}
 The surface tension-driven Hele-Shaw problem~\cite{HS98} is a classical model in fluid mechanics that describes the motion of an incompressible fluid confined within a narrow channel between two transparent flat plates. Due to the small gap separating the plates, the flow is effectively uniform in the transverse direction, and the problem reduces to a two-dimensional model. 
 Let~${\Omega(t)\subset\mathbb{R}^2,~t\geq 0}$,  represent the bounded domain occupied by the incompressible fluid at time $t\geq 0$, with corresponding velocity field $v(t)$ and pressure $ u(t)$. 
The dynamics in the bulk  is governed by Darcy's empirical law~\cite{Da56}; that is,
\[
\mathrm{div}\,v(t) = 0 \qquad \text{and} \qquad v(t) = -\nabla u(t) 
\]
for $t>0$.
Moreover, the free boundary $\Gamma(t):=\p\Omega(t)$ of the fluid evolves with normal velocity~$V(t)$  that coincides with the normal component of the fluid velocity~$v(t)$.
Taking into account surface tension effects, the pressure~$u(t)$ on~$\Gamma(t)$ is assumed to be proportional to the curvature $\kappa_{\Gamma(t)}$ of~$\Gamma(t)$, taken as positive for convex shapes.
Altogether, the surface tension-driven Hele-Shaw problem is described by the following system of equations:
\begin{subequations}\label{PB-}
\begin{equation}\label{PB1-}
\left.
\arraycolsep=1.4pt
\begin{array}{rcllll}
\Delta u(t)&=&0&\quad\text{in $\Omega(t)$,}\\
u(t)&=& \kappa_{\Gamma(t)}&\quad\text{on $\Gamma(t)$,}\\
V(t)&=&-\p_{{\rm n}_{\G(t)}}  u(t)&\quad\text{on $\Gamma(t)$}
\end{array}
\right\}
\end{equation}
for $t>0$, where, for simplicity, all material parameters are normalized to $1$, and ${\rm n}_{\G(t)}$ denotes the outer unit normal vector to~$\Gamma(t)$.
System~\eqref{PB1-} is supplemented by the initial condition
 \begin{equation}\label{IC-}
\arraycolsep=1.4pt
\begin{array}{rcllll}
\Omega(0)&=&\Omega_0,
\end{array}
\end{equation}
where $\Omega_0\subset\R^2$ is a given bounded domain.
\end{subequations}

The Hele-Shaw problem \eqref{PB-} has received considerable attention in the mathematical community; see, e.g.,~\cite{EEM11, ES97a, CP93, ES96c, MP12, Pr98, GO03, DR84, AK23, AP25x}, though this list is by no means exhaustive.
 A powerful approach to studying~\eqref{PB-} and related problems --~such as the Muskat problem, the Mullins-Sekerka flow, or the quasistationary Stokes flow~-- in bounded domains, where the bulk unknowns are determined by solving elliptic boundary value problems depending on the geometry (see the monograph~\cite{PS16}), is to apply the Hanzawa transformation~\cite{Ha81} and reformulate the problem on a fixed reference domain.
This approach, however, imposes certain regularity requirements on the geometry. 
 Moreover, the transformed equations are often quite involved and depend nonlinearly on the geometry  through the Hanzawa transform. 
 A further drawback of this method is that, after solving the transformed boundary value problems in the bulk, additional nonlinearities and nonlocalities are  introduced.
 
 An alternative approach that has been used recently in the context of the problems mentioned above  --~mainly in unbounded settings with equations defined in the entire space (see, e.g.,~\cite{AM22, MP2021, BM25x, EMM24, Ngu20, CCG11, GGS25, GMS23, CN23, FN21, LS25x, AN2021, GGPS25, GGPS23, GGPS19, DGN23, GGHP24})~-- is to employ potential theory in order to derive an explicit integral representation for the unknowns in the bulk.
This method requires less regularity  of the geometry  than the classical Hanzawa approach and yields results for initial data in (nearly) optimal function spaces within the setting of classical solutions. Moreover, it may even allow for the treatment of critical regularity cases for initial data within the framework of strong or viscosity solutions.

In this paper, we demonstrate that the potential-theoretic approach,  combined with abstract parabolic  theory, can be applied to the Hele-Shaw problem driven by surface tension, 
even in the case of a bounded domain. 
For simplicity, we restrict our analysis to a two-dimensional, star-shaped geometry $\Omega(t)$. 
However, the method is expected to extend to more general geometries and higher dimensions, albeit with increased technical complexity.
Our choice of star-shaped domains is motivated by the fact that, in such geometries, the Rellich identities (see Lemma~\ref{L:REl})~--~which are central to our analysis~--~admit an explicit and compact form that can be readily applied to study the invertibility of the double-layer potential operator.\medskip

\subsection*{Main results} 

In the following, we denote by $\bT$ the boundary of the unit disc and by $\mt$ and $\mn$ the unit tangent and the outward unit normal vectors to $\bT$, respectively. Functions defined on $\mathbb{T}$ are throughout identified with $2\pi$-periodic functions on $\mathbb{R}$.
We denote, for a given function $\rho \in \mathrm{C}^1(\mathbb{T})$ with~$\rho > 0$, by 
\[
\Gamma_\rho := \{ \rho(\tau)\mn(\tau) \,:\, \tau \in \mathbb{R} \}
\]
the ${\rm C}^1$-boundary of the star-shaped domain $\Omega_\rho \subset \mathbb{R}^2$, and note that the map $\Xi_\rho: \mathbb{T} \to \Gamma_\rho$, given by
\begin{equation} \label{xirho}
\Xi_\rho(\tau) = \rho(\tau)\mn(\tau), \qquad \tau \in \mathbb{R},
\end{equation}
is a ${\rm C}^1$-diffeomorphism.

In order to  tackle problem~\eqref{PB-} analytically, we assume that, at each time instant $t>0$, the evolving boundary $\Gamma(t)$ in ~\eqref{PB-} takes the form $\Gamma(t) = \Gamma_{\rho(t)}$. In this geometric setting, the Hele-Shaw problem~\eqref{PB-}  can be written as
\begin{subequations}\label{PB}
\begin{equation}\label{PB1}
\left.
\arraycolsep=1.4pt
\begin{array}{rcllll}
\Delta u(t)&=&0&\quad\text{in $\Omega_{\rho(t)}$,}\\[1ex]
u(t)&=& \kappa_{\Gamma_{\rho(t)}}&\quad\text{on $\Gamma_{\rho(t)}$,}\\[1ex]
V(t)&=&-\p_{{\rm n}_\rho(t)}  u(t)&\quad\text{on $\Gamma_{\rho(t)}$}
\end{array}
\right\}
\end{equation}
for $t>0$, where ${\rm n}_\rho(t) := n_{\Gamma_{\rho(t)}}$. Assuming that $\p\0_0 = \Gamma_{\rho_0}$ for a positive function $\rho_0$, the initial condition~\eqref{IC-} becomes
\begin{equation}\label{IC}
\arraycolsep=1.4pt
\begin{array}{rcllll}
\rho(0)&=&\rho_0.
\end{array}
\end{equation}
\end{subequations}
The unknowns of problem~\eqref{PB} are the function $\rho$ and the pressure $u$.
 However, since at each time instant $t > 0$ the pressure $u(t)$ is uniquely determined by the geometry $\rho(t)$  we will henceforth refer to $\rho$ alone as the solution to~\eqref{PB}. In fact, we formulate in Section~\ref{SEC:2} and Section~\ref{SEC:3}  the system~\eqref{PB} as a quasilinear parabolic problem  for~$\rho$ of the form
\begin{equation}\label{intrQPP}
\cfrac{{\rm d}\rho}{{\rm d}t}(t) = \Phi(\rho(t))[\rho(t)], \quad t > 0, \qquad \rho(0) = \rho_0,
\end{equation}
where $\Phi : \cV_r \to \kH(H^{r+1}(\bT), H^{r-2}(\bT))$ is a smooth mapping  for each~$r\in (3/2, 2]$ and 
\begin{equation}\label{cvr}
 \cV_r:=\{\rho\in H^{ r}(\bT)\,:\,\rho>0\}, \qquad r>3/2.
\end{equation}
That is, $\Phi(\rho)[\cdot]$ is the generator of an analytic semigroup on the Bessel potential space $H^{r-2}(\bT)$ with domain $H^{r+1}(\bT)$ for each $\rho\in  \cV_r $.
This feature together with  the quasilinear parabolic theory from \cite{Am93,MW20} enables us to  prove that the Hele-Shaw problem \eqref{PB} is locally well-posed in~$\cV_{\bar r}$ for any~$\bar r\in (3/2, 2)$ as stated in the following result:

\begin{thm}\label{MT1}
Let $\bar r\in (3/2,2)$ and chose an arbitrary $r\in (3/2,\bar r)$.
Then, given  $\rho_0\in \cV_{\bar r}$, there exists a unique maximal classical solution $\rho=\rho(\cdot;\rho_0) $ to \eqref{PB}
such that 
\begin{equation}\label{prop1}
\rho\in {\rm C}([0,T^+),\cV_{\bar r})\cap {\rm C}((0,T^+), H^{r+1}(\bT)) \cap {\rm C}^1((0,T^+), H^{r-2}(\bT))
\end{equation}
and, for some $\eta\in(0,(\bar r-r)/3]$,
\begin{equation}\label{prop1'}
\rho\in{\rm C}^{\eta}([0,T^+), H^r(\bT)),
\end{equation}
with  
\begin{equation}\label{prop2}
\rho(t)\in H^{r+2}(\bT)\quad\text{and}\quad {u(t)\in{\rm C}^2(\0_{\rho(t)})\cap {\rm C}^1(\ov{\0_{\rho(t)}})}\qquad\text{for $t\in(0,T^+)$,}
\end{equation}
where $T^+=T^+(\rho_0)>0$ is the maximal existence time of the solution.

Additionally, $[(t,\rho_0)\mapsto\rho(t; \rho_0)]$ defines a semiflow on $\cV_{\bar r}$ that is smooth in
\begin{equation}\label{openset}
\{(t,\rho_0)\,:\, \text{$\rho_0\in\cV_{\bar r}$ and $0<t<T^+(\rho_0)$}\}
\end{equation}
 and, moreover,   
 \begin{equation}\label{prop3}
 [(t,\tau)\mapsto\rho(t)(\tau)]\in {\rm C}^\infty((0,T^+)\times\R).
 \end{equation}
\end{thm}

We add the following observations derived from the parabolic smoothing property \eqref{prop3}:

\begin{rem}\label{Rem:1}\,
\begin{itemize}
\item[(a)] According to \eqref{prop3}, solutions to \eqref{PB} become instantaneously smooth, even though the curvature of the boundary of the initial geometry is not a function, but merely a distribution.

 \item[(b)] It follows from \eqref{prop3}  that the maximal existence time $T^+(\rho_0)$ of the maximal solution  corresponding to  $\rho_0\in \cV_{\bar r}$  is independent of the choice of $r\in(3/2,\bar r)$.
\end{itemize}
\end{rem}

Theorem~\ref{MT1} reveals  the strength of the potential-theoretic approach:

 \begin{rem}\label{Rem:2x}
In  Theorem~\ref{MT1} we may choose $\bar r$ arbitrarily  close to the critical threshold $3/2$ in the Bessel potential scale.
 Indeed,  as noted also in the context of the Muskat problem and the Mullins--Sekerka flow \cite{Ngu20, EMM24} (that are two-phase analogues of the Hele-Shaw flow), the exponent~$3/2$ is critical
 since $H^{3/2}(\bT)$ is invariant under the natural scaling of the problem.
 We point out that previous local well-posedness results in bounded geometries \cite{ES96c, ES97a},  based on a Hanzawa transformation approach,
 require ${\rm C}^{2+\alpha}$-regularity of the initial geometry  with~${\alpha > 0}$.
\end{rem}

{ 

To prove Theorem~\ref{MT1}, we first show in Proposition~\ref{Prop:1} that the pressure (and  the velocity) in the bulk  is determined by the geometry of the interface via an explicit integral formula.
As in \cite{ES96c, ES97a}, the curvature operator is decomposed into a quasilinear part containing the highest  order derivatives and a nonlinear lower order part. 
The handling of the nonlinear lower order term requires special care due to the lower regularity setting considered herein compared to \cite{ES96c, ES97a}.
 By differentiating this term and treating it as part of a quasilinear structure, the corresponding term in the final formulation~\eqref{intrQPP} of the Hele-Shaw problem can be regarded as lower order.
This analysis is made possible by the potential-theoretic framework, particularly through Lemmas~\ref{L:1} and~\ref{L:2}, 
where we show that the derivative of certain singular integral operators, evaluated at a density function~$\beta$, coincides with the inverse of the adjoint operators applied to~$-\beta'$.
Finally, in Section~\ref{SEC:2}, we reformulate the problem as a quasilinear evolution equation,  that is shown in Section~\ref{SEC:3} to be of parabolic type. 
This allows us to apply abstract parabolic theory from~\cite{Am93, MW20} to establish Theorem~\ref{MT1}.
The arguments in Sections~\ref{SEC:2} and~\ref{SEC:3} rely on technical results developed in Appendices~\ref{SEC:A}--\ref{SEC:B}, 
where we establish mapping properties, commutator estimates, and localization results for a particular class of (singular) integral operators that may be of independent interest.\\

Concerning the long-time behavior of solutions, we point out that the set of equilibrium solutions to \eqref{PB} forms a 3-dimensional manifold consisting exclusively of circles.
Moreover, the flow \eqref{PB} preserves both the area and the center of mass of $\Omega_{\rho_0}$, since Reynolds' transport theorem and Stokes' theorem yield for $t > 0$ that
\begin{align}
\frac{{\rm d}}{{\rm d} t}|\Omega_{\rho(t)}| &= -\int_{\Gamma_{\rho(t)}} \partial_{{\rm n}_{\rho(t)}} u(t)\, |{\rm d}\xi| = 0, \label{lol1}\\
-\frac{{\rm d}}{{\rm d} t}\int_{\Omega_{\rho(t)}} z\,{\rm d}z &=\int_{\Gamma_{\rho(t)}} z\partial_{{\rm n}_{\rho(t)}} u(t)\, |{\rm d}\xi|= \int_{\Gamma_{\rho(t)}} u(t)\, \mn_{\rho(t)}\, |{\rm d}\xi| =\int_{\Gamma_{\rho(t)}} \kappa_{\Gamma_{\rho(t)}} \mn_{\rho(t)}\, |{\rm d}\xi| = 0. \label{lol2}
\end{align}
We establish in Theorem~\ref{MT2} the exponential stability of the unit circle with center of mass located at $(0,0)$, which corresponds to the stationary solution $\rho=1$ to \eqref{PB}.
Since system~\eqref{PB-} is invariant under rotations and translations,  the exponential stability result in Theorem~\ref{MT2}  is actually valid for any circle 
(with arbitrary area and center of  mass) provided the  perturbations in the phase space preserve both the area and the center of mass of the circle.

\begin{thm}\label{MT2}
Let $\bar r\in (3/2,2)$.
Then, given $\xi\in (0,6)$, there exist  constants $\e>0$  and $M\geq 1$ such that for all $\rho_0\in\cV_{\bar r}$ 
 satisfying $\|\rho_0-1\|_{H^{\bar r}}<\e$, 
\begin{equation}\label{inidatacoo}
|\Omega_{\rho_0}|=\pi,  \qquad\text{and}\qquad \int_{\Omega_{\rho_0}} z\,{\rm d}z =(0,0),
\end{equation}
 the maximal solution $\rho=\rho(\cdot;\rho_0)$  to \eqref{PB} is globally  defined and
\begin{equation}\label{conv}
\|\rho(t)-1\|_{H^{{\bar r}}}\leq Me^{-\xi t} \|\rho_0-1\|_{H^{{\bar r}}}\qquad\text{for all $t\in[0,\infty)$.}
\end{equation}
\end{thm}

 We emphasize that, due to the invariants of the problem; see \eqref{lol1}--\eqref{lol2},   
$0$ is an isolated semi-simple eigenvalue of the operator $\Phi(1)$ with multiplicity $3$, which makes the stability analysis delicate.
In particular,	 the  principle of linearized stability for quasilinear parabolic problems in interpolation spaces~\cite[Theorem 1.3]{MW20} cannot be directly applied.
Instead, we develop and prove in Theorem~\ref{MT3} an abstract generalized principle of linearized stability for quasilinear parabolic problems in interpolation spaces, which plays a crucial role in the proof of  Theorem~\ref{MT2}.
This result fits within the framework of parabolic theory developed in~\cite{Am93} and accommodates cases where the linearized operator includes~$0$ in its spectrum.
We note that more general versions of such generalized stability principles have been established in~\cite{PS16, PSZ9a, PSZ09c} within the context of continuous, H\"older, or $L_p$-maximal regularity.
 
 \begin{rem}\label{R:PS} 
Our main results, Theorem~\ref{MT1} and Theorem~\ref{MT2}, can alternatively be established  in the setting of strong solutions in the Besov space 
\[
B_{2,p}^{s}(\mathbb{T})\qquad\text{for any $p > 3$ and $s\in\Big(\frac52-\frac3p,3-\frac3p\Big]$,}
\]
by using the abstract parabolic theory developed in~\cite{PS16, PSZ9a, PSZ09c} in the context of  $L_p$-maximal regularity. 
This is due to the fact, using the real interpolation functor $ (\cdot,\cdot)_{1-1/p,p}$, there is~$r=r(s)\in(3/2,2]$  such that 
  \[
  B_{2,p}^{s}(\mathbb{T})=(H^{r-2}(\bT),H^{r+1}(\bT))_{1-1/p,p}\hookrightarrow H^r(\bT),
  \]
with~$B_{2,p}^{3/2}(\mathbb{T})$ being a scaling invariant space for  \eqref{PB}  for each $1\leq p\leq\infty$.

Related to our results, we also refer to the recent papers \cite{GGPS25, GGPS23}, 
where the stability of two-dimensional Muskat bubbles in (critical) 
Wiener spaces is investigated through a combination of potential theory and subtle energy estimates.
 \end{rem}
 
 \subsection*{Outline} After setting up the notation in Section~\ref{SEC:NOT}, we show in Section~\ref{SEC:2} that the Hele-Shaw problem~\eqref{PB} can be reformulated as the evolution problem~\eqref{intrQPP} for $\rho$ alone. Section~\ref{SEC:3} is then devoted to the proof of Theorem~\ref{MT1}, while the stability result, Theorem~\ref{MT2}, is  proved in Section~\ref{SEC:4}. Moreover, in Appendix~\ref{SEC:A-1}, we establish the generalized principle of linearized stability stated in Theorem~\ref{MT3}. Finally, Appendices~\ref{SEC:A}--\ref{SEC:B} collect mapping properties, commutator estimates, and localization results for the family of (singular) integral operators introduced in~\eqref{Bnmp}.

\section{Notation and Conventions}  \label{SEC:NOT}

Given  $z=(x,y)\in\R^2$, we set $z^\top=(y,-x) =-iz$
and note that
\[
\mn^\top=-\mt,\qquad \mt^\top=\mn,\qquad \mn'=\mt,\qquad \mt'=-\mn.
\]
Moreover, we compute for the mapping  $\Xi_\rho: \mathbb{T} \to \Gamma_\rho$, $\tau\mapsto  \rho(\tau)\mn(\tau)$, defined in \eqref{xirho} that  
\[
\Xi_\rho' =\rho \mt +\rho' \mn\qquad\text{and}\qquad\Xi_\rho'^\top =\rho \mn -\rho' \mt,
\]
with
\[
\omega_\rho:=|\Xi_\rho'|=\big(\rho^2 +\rho'^2\big)^{1/2}.
\]
Hence, the unit tangent vector $\mt_\rho$ and the unit outward normal  vector $\mn_\rho$  at  $\Gamma_\rho=\p\0_\rho$ are given by
\begin{equation}\label{mtmn}
\mt_\rho\circ\Xi_\rho = \frac{ \Xi_\rho'}{\omega_\rho}\qquad\text{and}\qquad
\mn_\rho\circ\Xi_\rho =  \frac{ \Xi_\rho'^\top }{ \omega_\rho}.
\end{equation}
 
If  $\rho\in{\rm C}^2(\bT)$, then the curvature $\kappa_{\Gamma_{\rho}}$ of $\Gamma_\rho$ can be expressed as
\begin{equation}\label{darcru}
\kappa_{\Gamma_{\rho}}\circ\Xi_\rho=\kappa(\rho)[\rho]+f(\rho),
\end{equation}
with leading order quasilinear part $\kappa(\cdot)[\cdot]$ and lower order nonlinear part $f(\cdot)$ defined by
\begin{align}\label{krhof}
\kappa(\rho)[h]:=-\frac{\rho}{\omega_\rho^3}h''\qquad\text{and}\qquad 
f(\rho):=  \frac{\rho^2+2\rho'^2}{\omega_\rho^3}.
\end{align}

 Given an integrable function $g:\Gamma_\rho\rightarrow \R$ we write
$$
\int_{\Gamma_\rho} g(\xi)\,\vert {\rm d}\xi\vert:=\int_{-\pi}^\pi g\big(\Xi_\rho(s)\big) \vert \Xi_\rho'(s)\vert\, {\rm d}s
$$
for the line integral (and analogously for principal values).

Some of our arguments rely on  the well-known interpolation property
  \begin{align}\label{IP}
[H^{r_0}(\mathbb{\bT}),H^{r_1}(\mathbb{\bT})]_\theta=H^{(1-\theta)r_0+\theta r_1}(\mathbb{\bT}),\qquad\theta\in(0,1),\,  -\infty\leq r_0\leq r_1<\infty,
\end{align}
where $[\cdot,\cdot]_\theta$ denotes the complex interpolation functor of exponent $\theta$ and $H^{r}(\mathbb{\bT})$, $r\in\R$, are the standard, $L_2$-based, Bessel potential spaces.
Furthermore, we will  also use the fact that, given~${r\in(0,1)}$, there is a constant $C>1$ such that, for all $h\in H^{r}(\bT)$,
\begin{align}\label{equiv}
C^{-1}\|h\|_{H^{r}}&\leq \|h\|_2+[h\big]_{H^{r}}\leq C\|h\|_{H^{r}},
\end{align}
 where the seminorm $[\cdot]_{H^{r}}$ is given by 
\begin{equation}\label{equivsem}
[h]_{H^{r}}^2:=\int_{-\pi}^\pi\frac{\|\sfT_sh-h\|_2^2}{|s|^{1+2 r}}\,{\rm d}s
\end{equation}
 with $\sfT_s h:=h(\cdot+s)$ denoting the  right-translation operator.
 We write  $ L_{2,0}(\bT)$ for the space of functions $u\in L_{2}(\bT)$  with $\langle u\rangle=0$, where
\[
\langle u\rangle:=\frac{1}{2\pi}\int_{-\pi}^\pi u(\tau)\, {\rm d}\tau,
\]
and  set $H^r_0(\bT):=H^r(\bT)\cap L_{2,0}(\bT)$ for $r\geq 0$.
 We also denote by $\langle\cdot ,\cdot \rangle $  the canonical duality pairing between $\mathcal{D}'(\bT)$ and $\mathcal{D}(\bT)={\rm C}^\infty(\bT)$.

We also point out  the estimate
		\begin{equation}\label{algebra}
			\|ab\|_{H^s}\leq C\big(\|a\|_\infty\|b\|_{H^s}+\|a\|_{H^s}\|b\|_\infty\big),
		\end{equation}
		which holds for all $a,\, b\in H^s(\bT)$ with  a  constant $C=C(s)>0$, provided $s\in(1/2,1]$.
If  $A$ is an operator and $\varphi$ is a function,  we write $\llbracket \varphi,A\rrbracket  $ for the commutator
\begin{align}\label{comm}
    \llbracket \varphi,A\rrbracket [h]&:=\varphi A[h]-A[\varphi h].
\end{align}

 Given  Banach spaces $E_0$ and $E_1$ with continuous and dense embedding~$E_1\hookrightarrow E_0$, we denote by~$\mathcal{H}(E_1,E_0)$ 
the open subset of the bounded operators $\mathcal{L}(E_1,E_0)$ consisting of generators of strongly continuous, analytic semigroups.

Finally, if  $E_1, \ldots, E_n, E, F$ are Banach spaces, $n \in \mathbb{N}$, 
we write $\mathcal{L}^n(E_1\times \ldots\times E_n, F)$ for the Banach space of bounded $n$-linear maps from $\prod_{i=1}^n E_i$ into $F$. 
When~${E_i = E}$ for all $1 \leq i \leq n$, we use the abbreviation $\mathcal{L}^n(E, F)$ and denote by $\mathcal{L}_{\mathrm{sym}}^n(E, F)$ its subspace of symmetric operators.
If~$\mathcal{U} $ is  open subset of $E$, we write~${\rm C}^{1-}(\mathcal{U}, F)$ for the space of locally Lipschitz continuous mappings from~$\mathcal{U}$ to $F$, 
and~${\rm C}^{\infty}(\mathcal{U}, F)$ is its subspace consisting of smooth mappings.

 \section{An equivalent formulation of \eqref{PB} using singular integrals}\label{SEC:2}
  
The main goal of this section is to show that the evolution problem~\eqref{PB} can be formulated as a quasilinear evolution equation for $\rho$  exclusively, with nonlinearities expressed as singular integrals. 
To achieve this, we first establish in Proposition~\ref{Prop:1} the unique solvability of an elliptic problem related to~\eqref{PB}, see \eqref{FTP}, 
which   implies in particular that the function $\rho$ determines at each time instant the   pressure $u$.
Moreover, we determine $u$ explicitly as an integral  involving a density function $\beta=\beta(\rho)$ that
 solves a linear equation associated with the double layer potential of the Laplace operator corresponding to the graph~$\G_\rho$; see~\eqref{reseq}.
The unique solvability of \eqref{reseq}  is a key  ingredient in the analysis, and is based on mapping properties for the family of (singular) integral operators introduced in \eqref{Bnmp} below
and investigated in Appendix~\ref{SEC:A}.  

Throughout this section, we fix an arbitrary~$r\in(3/2,2]$ and recall the definition  of $\cV_r$ in~\eqref{cvr}, noticing that $H^r(\mathbb{T})\hookrightarrow {\rm C}^{r-1/2}(\mathbb{T})$.

\subsection*{The fixed time problem} We prove that  the pressure $u$ is uniquely determined by the geometry, i.e. that the elliptic Dirichlet problem~\eqref{FTP}  below has a unique solution which is given as an explicit integral over the boundary $\G_\rho$.
\begin{prop}\label{Prop:1}
Given $\rho\in\cV_r$ and $\varphi\in H^r(\bT)$,  
 the Dirichlet problem 
\begin{equation}\label{FTP}
\left.
\arraycolsep=1.4pt
\begin{array}{rcllll}
\Delta u&=&0&\quad\text{in $\Omega_{\rho}$,}\\[1ex]
u&=& \varphi\circ \Xi_\rho^{-1}&\quad\text{on $\Gamma_{\rho}$}
\end{array}
\right\}
\end{equation}
 has a unique solution~${u\in{\rm C}^2(\0_\rho)\cap {\rm C}^1(\ov{\0_\rho})}$ which is given by
\begin{align}\label{solfor}
u(z):=\frac{1}{\pi }\int_{\Gamma_\rho} \frac{(\xi-z)\cdot \mn_\rho(\xi)}{|\xi-z|^2}\beta\circ\Xi_\rho^{-1}(\xi)\, |{\rm d}\xi|
=\frac{1}{\pi }\int_{-\pi}^\pi  \frac{(\Xi_\rho(s)-z)\cdot \Xi_\rho'(s)^\top}{|\Xi_\rho(s)-z|^2}\beta(s)\, {\rm d}s
\end{align}
for $z\in\0_\rho$,
with $\beta\in H^r(\bT)$ denoting the unique solution to the equation
\begin{equation}\label{reseq}
(1 +\bD(\rho))[\beta]=\varphi,
\end{equation}
 where $\bD(\rho)$ is the double layer potential for the Laplace operator associated with the curve~$\Gamma_\rho$; see~\eqref{DLP}.
\end{prop}
\begin{proof}
The uniqueness of the solution to \eqref{FTP} follows via the weak  maximum  principle for elliptic problems.
In order to establish the existence claim we define, for  a given density function~$\beta\in H^r(\bT),$ the function $u:=u(\rho)[\beta]:\0_\rho\to\R$  by~\eqref{solfor}.
Clearly, $u\in{\rm C}^\infty(\0_\rho)$, and, since
\[
\Delta_z\frac{(\xi-z)_i}{|\xi-z|^2}=0\qquad\text{in~$\0_\rho$, \quad $i=1,2$.}
\]
it follows that $u$ satisfies~\eqref{FTP}$_1$. 
 Recall from Plemelj's theorem ~\cite{JKL93} that for a H\"older continuous function $\varphi:\Gamma_\rho\rightarrow\R$ and
$$
\phi(z):=\frac{1}{2\pi i}\int_{\Gamma_\rho}\frac{\varphi(\xi)}{\xi-z}\,{\rm d}\xi, \qquad z\in  \mathbb{R}^2\setminus\G_\rho,
$$
one has that  $\phi\in {\rm C}(\overline{\Omega_\rho^\pm})$, where $\Omega^+_\rho:=\Omega_\rho$ and $\Omega^-_\rho:=\mathbb{R}^2\setminus\overline{\Omega^+_\rho}$, and, given $z_0\in\Gamma_\rho$,
\begin{equation}\label{plemelj}
\lim_{\Omega^\pm_\rho\ni z\to z_0}\phi(z)=\pm\frac{\varphi(z_0)}{2}+\frac{1}{2\pi i} \PV\int_{\Gamma_\rho}\frac{\varphi(\xi)}{\xi-z_0}\,{\rm d}\xi,
\end{equation}
 where the symbol $\PV$ stands for the principal value.
Thus, since
$$
\mathrm{Re}\,\left(\frac{1}{i}\frac{1}{\xi-z}\mt_\rho(\xi)\right)=\frac{(\xi-z)\cdot\mn_\rho(\xi)}{\vert\xi-z\vert^2},\qquad  \xi\in \G_\rho,\, z\in\mathbb{R}^2\setminus\G_\rho,
$$
it follows that
\begin{equation}\label{i8}
\mathrm{Re}\,\phi(z)= \frac{1}{ 2\pi}\int_{\Gamma_\rho}\frac{(\xi-z_0)\cdot\mn_\rho(\xi)}{\vert\xi-z_0\vert^2}\varphi(\xi)\,\vert{\rm d}\xi\vert, \qquad  z\in\mathbb{R}^2\setminus\G_\rho.
\end{equation}
Consequently, \eqref{solfor}, \eqref{plemelj}, and \eqref{i8} entail that  $u\in{\rm C}(\overline{\Omega_\rho})$ and
 \[
u\circ\Xi_\rho=(1+\bD(\rho))[\beta]\qquad\text{on $\bT$,}
\]
with $\bD(\rho)$ being defined in~\eqref{DLP}.
Hence, if  $\beta$ solves the equation~\eqref{reseq},
then $u$ solves also~\eqref{FTP}$_2$.
The existence of a solution to~\eqref{reseq} is established in Proposition~\ref{Prop:3}~(ii) below.

Finally, in view of the formula
\begin{equation*}
\nabla_z\frac{(\Xi_\rho(s)-z)\cdot \Xi_\rho'(s)^\top}{|\Xi_\rho(s)-z|^2}=\p_s\frac{(\Xi_\rho(s)-z)^\top}{|\Xi_\rho(s)-z|^2},
\qquad s\in\R,\, z \in\Omega_\rho,
\end{equation*}
integration by parts yields for $z\in\Omega_\rho$
\begin{align*}
\nabla u(z)&=\frac{1}{\pi } \int_{-\pi}^\pi  \p_s\frac{(\Xi_\rho(s)-z)^\top}{|\Xi_\rho(s)-z|^2}\beta(s)\, {\rm d}s
=-\frac{1}{\pi }\int_{-\pi}^\pi  \frac{(\Xi_\rho(s)-z)^\top}{|\Xi_\rho(s)-z|^2}\beta'(s)\, {\rm d}s\\
& =\frac{i}{\pi }\int_{-\pi}^\pi  \frac{\Xi_\rho(s)-z}{|\Xi_\rho(s)-z|^2}\beta'(s)\, {\rm d}s
=-\frac{1}{\pi i}\int_{-\pi}^\pi  \frac{1}{\overline{\Xi_\rho(s)-z}} \beta'(s) \, {\rm d}s\\
&=\overline{\frac{1}{\pi i}\int_{-\pi}^\pi  \frac{1}{\Xi_\rho(s)-z}\frac{\beta'(s)}{\Xi_\rho'(s) }\Xi_\rho'(s) \, {\rm d}s}
=\overline{\frac{1}{\pi i}\int_{\G_\rho}  \frac{1}{\xi-z}\frac{\beta'}{\Xi_\rho' }\circ \Xi_\rho^{-1}(\xi)\, {\rm d}\xi},
\end{align*}
and Plemelj's theorem ensures that~${u\in{\rm C}^1 (\overline{\0_\rho})}$ with
\begin{equation}\label{trv}
(\nabla u)\circ\Xi_\rho(\tau)=\frac{\beta'\Xi_\rho'}{\omega_\rho^2}(\tau)-\frac{1}{\pi }\PV\int_{-\pi}^\pi  \frac{(\Xi_\rho(s)-(\Xi_\rho(\tau))^\top}{|\Xi_\rho(s)-\Xi_\rho(\tau)|^2}\beta'(s)\, {\rm d}s,\qquad \tau\in\R.  
\end{equation}
\end{proof}

 The solvability of~\eqref{reseq} requires some preparation, which is the context of the subsequent considerations.

 \subsection*{A family of (singular) integral operators} 
 To establish the unique solvability of equation~\eqref{reseq}, 
 we introduce a family of (singular) integral operators that play a crucial role in our approach.
Given   $m,\,n,\,p\in\N$  with $0\leq p\leq n+1$   and $\varrho:=(\varrho_1,\ldots,\varrho_m)\in\cV_r^m$ we define the (singular) integral operator 
\begin{equation}\label{Bnmp}
B_{n,m}^p(\varrho )[h,\beta](\tau):=\displaystyle\frac{1}{\pi}\PV\int_{-\pi}^\pi\frac{t_{[s]}^p\prod\limits_{i=1}^n\frac{\delta_{[\tau,s]}h_i}{t_{[s]}}}
{\prod\limits_{i=1}^m\Big[(\varrho_i(\tau)+\varrho_i(\tau-s))^2+\Big(\frac{\delta_{[\tau,s]}\varrho_i}{t_{[s]}}\Big)^2\Big]}
\frac{\beta(\tau-s)}{t_{[s]}}\,{\rm d}s
\end{equation}
for $h=(h_1,\ldots, h_n)\in W^1_\infty(\bT)^n$, $\beta\in L_2(\bT)$, and $\tau\in\R$, where, for any function $u:\R\to\R$, 
 \[
t_{[s]}:=\tan(s/2)\qquad\text{and}\qquad\delta_{[\tau,s]}u:=u(\tau)-u(\tau-s),\qquad s\in(-\pi,\pi),\, \tau\in\R.
\]
The  principal value is needed only when $p=0$, the kernel of~$B_{n,m}^p$ being bounded when $1\leq p\leq n+1$ since $h_i\in W^1_\infty(\bT)$, $1\leq i\leq n$.
When the components of $\varrho$ and $h$ are equal to $\rho\in\cV_r$, we set
\begin{equation}\label{sfB}
 {\sf B}_{n,m}^p(\rho):=B_{n,m}^p(\rho,\ldots,\rho )[\rho,\ldots,\rho,\cdot].
\end{equation} 
We recall that the periodic Hilbert transform~$H$ is given by~$H = 2B^{0}_{0,1}(1)$  and is  a Fourier multiplier with symbol $ (-i \,\mathrm{sign}(k))_{k\in\mathbb{Z}}$. 

As a straightforward consequence of Lemma~\ref{L:B1} and Lemma~\ref{L:B1p}  from the  Appendix~\ref{SEC:A}, the mappings 
\begin{equation}\label{basreg1}
\begin{aligned}
&[\rho\mapsto \sfB_{n,m}^0(\rho)]:\cV_r\to\kL(L_2(\bT)),\\
&[\rho\mapsto \sfB_{n,m}^p(\rho)]:\cV_r\to\kL(L_1(\bT), {\rm C}(\bT)),\qquad 1\leq p\leq n+1,
\end{aligned}
\end{equation}
are locally Lipschitz continuous.
Moreover, as  shown in Lemma~\ref{L:B5},  it holds that 
\begin{equation}\label{basreg2}
\begin{aligned}
&[\rho\mapsto \sfB_{n,m}^0(\rho)]\in{\rm C}^\infty(\cV_r,\kL(H^{r-1}(\bT))),\\
&[\rho\mapsto \sfB_{n,m}^p(\rho)]\in{\rm C}^\infty(\cV_r,\kL(L_2(\bT),H^1(\bT))),\qquad 1\leq p\leq n+1.
\end{aligned}
\end{equation}

\subsection*{The double layer potential and its dual}
Given $\rho\in\cV_r$, we  introduce the  double layer potential~$\bD(\rho)$ for the Laplace operator associated  with the graph~$\G_\rho$   and its $L_2$-adjoint $\bD(\rho)^*$ by
\begin{align}
\bD(\rho)[\beta](\tau)&:=-\frac{1}{\pi}\PV\int_{-\pi}^\pi \frac{(\delta_{[\tau,s]} \Xi_\rho) \cdot \Xi_\rho'(\tau-s)^\top}{|\delta_{[\tau,s]} \Xi_\rho|^2}\beta(\tau-s)\,  {\rm d}s,\label{DLP}\\
\bD(\rho)^*[\beta](\tau)&:=\frac{1}{\pi}\PV\int_{-\pi}^\pi \frac{(\delta_{[\tau,s]} \Xi_\rho)\cdot \Xi_\rho'(\tau)^\top}{|\delta_{[\tau,s]} \Xi_\rho|^2}\beta(\tau-s)\, {\rm d}s\label{AdjDLP}
\end{align}
for $\beta\in L_2(\bT)$ and $\tau\in\R$.
The $L_2$-boundedness of $\bD(\rho)$ and of  $\bD(\rho)^*$ follow immediately from  \eqref{basreg1} in view of the identities 
\begin{equation}\label{forD}
\begin{aligned}
\bD(\rho)[\beta]&=-\sfB^2_{1,1}(\rho)[\rho\beta]- \sfB^0_{1,1}(\rho)[\rho\beta]+2\rho\sfB^1_{0,1}(\rho)[\rho\beta ]+2\rho\sfB^0_{0,1}(\rho)[ \rho' \beta]
\end{aligned}
\end{equation}
and
\begin{equation}\label{forD*}
\begin{aligned}
\bD(\rho)^*[\beta]&=\rho\sfB^2_{1,1}(\rho)[\beta]+\rho \sfB^0_{1,1}(\rho)[\beta]+2\rho\sfB^1_{0,1}(\rho)[\rho\beta ]-2\rho'\sfB^0_{0,1}(\rho)[\rho\beta],
\end{aligned}
\end{equation}
 which may be verified using the tangent half-angle formulas for sine and cosine and expressing the denominator as
$$
|\delta_{[\tau,s]} \Xi_\rho|^2 =\frac{t_{[s]}^2}{1+t_{[s]}^2}\left(
\big(\rho(\tau)+\rho(\tau-s)\big)^2+\Big(\frac{\delta_{[\tau,s]}\rho}{t_{[s]}}\Big)^2\right).
$$

Moreover, these formulas together with  \eqref{basreg2}  show that
\begin{equation}\label{reg:Ds}
\bD,\, \bD(\cdot)^*\in {\rm C}^\infty(\cV_r,\kL(H^{r-1}(\bT))).
\end{equation}

	An important property relating $\bD(\rho)$ and $\bD(\rho)^*$ is provided by the following lemma.
\begin{lemma}\label{L:1}
Given $\rho\in\cV_r$ and $\beta\in H^1(\bT)$, it holds that  $\bD(\rho)[\beta]\in H^1(\bT)$ with
\begin{equation}\label{eqcom}
(\bD(\rho)[\beta])'=-\bD(\rho)^*[\beta'].
\end{equation}
Moreover, we have
\begin{equation}\label{reg:Dr}
\bD\in {\rm C}^\infty(\cV_r,\kL(H^{r}(\bT))).
\end{equation}
\end{lemma}
\begin{proof}
In view of \eqref{basreg1}, \eqref{forD}, and \eqref{forD*} it suffices to establish \eqref{eqcom} for $\rho,\,\beta\in{\rm C}^\infty(\bT)$ with~${\rho>0}$.
In this case,  Lemma~\ref{L:B3} (in particular \eqref{forder}) and Lemma~\ref{L:B3p} (in particular~\eqref{derH1})
 enable us to exchange differentiation and integration when differentiating the function~$\sfB_{n,m}^p(\rho)[\beta]\in H^1(\bT)$ and, together 
with \eqref{forD}, we deduce for $\tau\in\R$ that
\begin{align*}
\big(\bD(\rho)[\beta]\big)'(\tau)&
=\bD(\rho)[\beta'](\tau)-\frac{1}{\pi}\PV\int_{-\pi}^\pi  \p_\tau\bigg[\frac{(\delta_{[\tau,s]} \Xi_\rho) \cdot \Xi_\rho'(\tau-s)^\top}{|\delta_{[\tau,s]} \Xi_\rho|^2}\bigg]\beta(\tau-s)\,{\rm d}s.
\end{align*}
Since for $0\neq s\in(-\pi,\pi)$ and $\tau\in\R$ we have
\[
\p_\tau\frac{(\delta_{[\tau,s]} \Xi_\rho) \cdot \Xi_\rho'(\tau-s)^\top}{|\delta_{[\tau,s]} \Xi_\rho|^2}
 =\p_s\frac{(\delta_{[\tau,s]} \Xi_\rho) \cdot (\delta_{[\tau,s]} \Xi_\rho')^\top}{|\delta_{[\tau,s]} \Xi_\rho|^2},
\]
 integration  by parts leads us to \eqref{eqcom}.
 The remaining  mapping property~\eqref{reg:Dr} now follows by combining \eqref{reg:Ds} and \eqref{eqcom}. 
\end{proof}

\subsection*{Two further singular integral operators}
We  define two additional singular integral operators that are used in the analysis. 
Specifically, given $\rho\in\cV_r$, we set
\begin{align}
\bB(\rho)[\beta](\tau)&:=-\frac{1}{\pi}\PV\int_{-\pi}^\pi \frac{(\delta_{[\tau,s]} \Xi_\rho) \cdot \Xi_\rho'(\tau-s)}{|\delta_{[\tau,s]} \Xi_\rho|^2}\beta(\tau-s)\, {\rm d}s,\label{B}\\
\bB(\rho)^*[\beta](\tau)&:=\frac{1}{\pi}\PV\int_{-\pi}^\pi \frac{(\delta_{[\tau,s]} \Xi_\rho) \cdot \Xi_\rho'(\tau)}{|\delta_{[\tau,s]} \Xi_\rho|^2}\beta(\tau-s)\,  {\rm d}s,\label{AdjB}
\end{align}
for $\beta\in L_2(\bT)$ and  $\tau\in\R$, with $\bB(\rho)^*$  being the $L_2$-adjoint of $\bB(\rho)$.
Indeed, both operators belong to $\kL(L_2(\bT))$  in view of \eqref{basreg1}, since  
\begin{equation}\label{forB}
\bB(\rho)[\beta]=-\sfB^2_{1,1}(\rho)[\rho'\beta]- \sfB^0_{1,1}(\rho)[\rho'\beta]-2\rho\sfB^0_{0,1}(\rho)[\rho\beta ]+2\rho\sfB^1_{0,1}(\rho)[\rho'\beta]
\end{equation}
and
\begin{equation}\label{forB*}
\bB(\rho)^*[\beta]=\rho' \sfB^2_{1,1}(\rho)[\beta]+\rho'\sfB^0_{1,1}(\rho)[\beta]+2\rho\sfB^0_{0,1}(\rho)[\rho\beta ]+2\rho'\sf B^1_{0,1}(\rho)[\rho\beta].
\end{equation}
Moreover, recalling~\eqref{basreg2}, we infer from~\eqref{forB}-\eqref{forB*}  that
\begin{equation}\label{reg:Bs}
\bB,\, \bB(\cdot)^*\in {\rm C}^\infty(\cV_r,\kL(H^{r-1}(\bT))).
\end{equation}

We now prove for \(\bB(\rho)\) and \(\bB(\rho)^*\) an analogue of Lemma~\ref{L:1}.
\begin{lemma}\label{L:2}
Given $\rho\in\cV_r$ and $\beta\in H^1(\bT)$, it holds that  $\bB(\rho)[\beta]\in H^1(\bT)$ with
\begin{equation}\label{eqcom2}
(\bB(\rho)[\beta])'=-\bB(\rho)^*[\beta'].
\end{equation}
\end{lemma}
\begin{proof}
Arguing as in the proof of Lemma~\ref{L:1}, for $\rho,\,\beta\in{\rm C}^\infty(\bT)$ with~${\rho>0}$ and~$\tau\in\R$ we have
\begin{align*}
\big(\bB(\rho)[\beta]\big)'(\tau)&
=\bB(\rho)[\beta'](\tau)-\frac{1}{\pi}\PV\int_{-\pi}^\pi  \p_\tau\bigg[\frac{(\delta_{[\tau,s]} \Xi_\rho) \cdot \Xi_\rho'(\tau-s)}{|\delta_{[\tau,s]} \Xi_\rho|^2}\bigg]\beta(\tau-s)\,{\rm d}s,
\end{align*}
and, for $0\neq s\in(-\pi,\pi)$ and $\tau\in\R$,  we compute
\[
\p_\tau\frac{(\delta_{[\tau,s]} \Xi_\rho) \cdot \Xi_\rho'(\tau-s)}{|\delta_{[\tau,s]} \Xi_\rho|^2}
= \p_s\frac{(\delta_{[\tau,s]} \Xi_\rho) \cdot(\delta_{[\tau,s]} \Xi_\rho')}{|\delta_{[\tau,s]} \Xi_\rho|^2}.
\]
Integration by parts now yields~\eqref{eqcom2}.
\end{proof}

In view of Lemma~\ref{L:1} and Lemma~\ref{L:2} it immediately follows that 
\begin{equation}\label{invdua}
\bB(\rho)^*,\,\bD(\rho)^*\in\kL(L_{2,0}(\bT)).
\end{equation}
Indeed,  given $\beta\in L_{2,0}(\bT)$, there exists $\varphi \in H^1(\bT)$ with $\varphi'=\beta$, and together with~\eqref{eqcom2} we get
\begin{align*}
(2\pi)\langle\bB(\rho)^*[\beta]\rangle=\langle \bB(\rho)^*[\beta]|1\rangle_{L_2} =\langle \bB(\rho)^*[\varphi']|1\rangle_{L_2}=-\langle (\bB(\rho)[\varphi])'|1\rangle_{L_2}=0,
\end{align*}
which proves \eqref{invdua} for $\bB(\rho)^*$ (the corresponding property for $\bD(\rho)^*$ following similarly).

\subsection*{Invertibility of layer potentials in $L_2(\bT)$}
The invertibility of layer potentials is a fundamental issue in potential theory; see~\cite{Ve84},  
where the invertibility of $\lambda+\bD(\rho)$ and $\lambda+\bD(\rho)^*$ is established in an $L_2$-setting for Lipschitz domains and $\lambda=\pm 1$.  
For star-shaped domains in~$\mathbb{R}^2$, we provide herein a short and direct approach, based on the Rellich identities~\eqref{Rel12222},  
 that, on the one hand, allows us to consider a larger set of values for $\lambda$  than in \cite{Ve84}, and, on the other hand, 
 permits us to establish the invertibility of these operators in Sobolev spaces of higher order.

\begin{lemma}[Rellich identities]\label{L:REl}
Given   $\rho  \in \cV_r$ and $\beta\in L_2(\bT)$,  it holds that
\begin{equation}\label{Rel12222}
\begin{aligned}
&\int_{-\pi}^\pi\frac{\rho^2}{\omega_\rho^2}\big(\big|(\pm1-\bD(\rho)^*)[\beta]\big|^2-\big|\bB(\rho)^*[\beta]\big|^2\big)\,{\rm ds}\\[1ex]
&=\int_{-\pi}^\pi\frac{(\rho^2)'(\pm1-\bD(\rho)^*)[\beta]\bB(\rho)^*[\beta]}{\omega_\rho^2}  \,{\rm ds}-4(\pm 1-1)\pi|\langle\beta\rangle|^2.
\end{aligned}
\end{equation}
\end{lemma}
\begin{proof} Recalling~\eqref{basreg1}, \eqref{forD*}, and \eqref{forB*}, it suffices to prove~\eqref{Rel12222} for
 $\rho,\, \beta \in {\rm C}^\infty(\bT)$ 
with~$\rho>0$, that we now fix. 
Similarly as in the proof of Proposition~\ref{Prop:1}, we define~${v:=v(\rho)[\beta]:\R^2\setminus\G_\rho\to\R^2}$  by
\begin{align}\label{vau}
v(z):=-\frac{1}{\pi }\int_{\Gamma_\rho} \frac{(\xi-z)^\top}{|\xi-z|^2}\frac{\beta}{\omega_\rho}\circ\Xi_\rho^{-1}(\xi)\, |{\rm d}\xi|,\qquad z\in\R^2\setminus\G_\rho.
\end{align}
Then, $v\in{\rm C}^\infty(\R^2\setminus\G_\rho)$ satisfies
${\rm div\, } v={\rm rot\, } v=0$ in $\R^2\setminus\G_\rho.$
Let $v^\pm:=v|_{\0_\rho^\pm}$, where  again we set~$\0_\rho^+:=\0_\rho$ and~${\0_\rho^-:=\R^2\setminus\overline{\0_\rho^+}}$.
Plemelj's theorem (see, e.g.~\cite{JKL93}) ensures  that  $v^\pm\in{\rm C} (\overline{\0_\rho^\pm})$  with 
\begin{equation}\label{v1}
v^\pm\circ\Xi_\rho(\tau)=\pm\frac{\beta\Xi_\rho'}{\omega_\rho^2}(\tau)
-\frac{1}{\pi}\PV\int_{-\pi}^\pi \frac{(\Xi_\rho(s)-\Xi_\rho(\tau))^\top}{|\Xi_\rho(s)-\Xi_\rho(\tau)|^2}\beta(s)\,  {\rm d}s,
\qquad \tau\in\bT.
\end{equation}
Let now $W:\R^2\setminus\G_\rho\to\R^2$ be given by
\[
W(z)=z|v|^2(z)-2v(z)z\cdot v(z)
\]
and set $W^\pm:=W|_{\0_\rho^\pm}$.
Then, ${\rm div\,}W^\pm=0$ in $\0^\pm_\rho$. 
Since $W^+\in {\rm C}^\infty(\0_\rho^+)\cap {\rm C}\big(\overline{\0_\rho^+}\big)$, Stokes' theorem  thus yields 
\[
\int_{-\pi}^\pi(W^+\circ\Xi_\rho(\tau))\cdot\Xi_\rho'(\tau)^\top\,{\rm d\tau}=\int_{\G_\rho} W^+\cdot\mn_\rho\,{\rm |d\xi|}=0.
\]
Hence, splitting $v^\pm\vert_{\Gamma_\rho}$ into normal and tangential components and noticing from~\eqref{v1} that
$$
(v^\pm\cdot \mn_\rho)\circ\Xi_\rho=\frac{\bB(\rho)^*[\beta]}{\omega_\rho}, \qquad (v^\pm\cdot \mt_\rho)\circ\Xi_\rho=\frac{(\pm 1-\bD(\rho)^*)[\beta]}{\omega_\rho}
$$
leads to the identity in \eqref{Rel12222} with $+$.

To establish the second identity \eqref{Rel12222} (with $-$), we first apply Lebesgue's dominated convergence theorem to deduce that
\begin{align*}
\ov z v(z)&=\overline{\frac{1}{\pi i}\int_{\G_\rho}  \frac{z}{ \xi-z}\frac{\beta}{\omega_\rho }\circ \Xi_\rho^{-1}(\xi)\, |{\rm d}\xi}|
\underset{|z|\to\infty}\longrightarrow\overline{\frac{1}{\pi i}\int_{\G_\rho} \frac{\beta}{\omega_\rho }\circ \Xi_\rho^{-1}(\xi)\, |{\rm d}\xi}|=\overline{-\frac{1}{\pi i}\int_{-\pi}^\pi \beta\, {\rm ds}},
\end{align*}
hence $\ov z v(z)\to -2i\overline{\langle\beta\rangle}$  as $|z|\to\infty$.
 Since $W^-\in {\rm C}^\infty(\0_\rho^-)\cap {\rm C}\big(\overline{\0_\rho^-}\big)$, we integrate~${\rm div\,}W^-$ over  the annular domain
$\0_\rho^{-,R}=\{z\in\0_\rho^-\,:\, |z|<R\},$ $ R>2,$
and obtain via Stokes' theorem that
\[
\int_{\G_\rho} W^-\cdot\mn_\rho\,{\rm |d\xi|}=\int_{R\bT} W^-(\xi)\cdot\frac{\xi}{|\xi|}\,{\rm |d\xi|}\underset{R\to\infty}\to8\pi|\langle\beta\rangle|^2.
\] 
 This yields the second Rellich identity and completes the proof.
\end{proof}

We are now in a position  to address the invertibility of  $\lambda+\bD(\rho) $ and~$\lambda+\bD(\rho)^*$.
\begin{prop}\label{Prop:2}
 Let $\rho\in \cV_r$.
\begin{itemize}
\item[(i)]  The operator  $\lambda+\bD(\rho)^*\in \kL(L_{2,0}(\bT))$ is invertible for all $\lambda\in\R\setminus(-1,1)$.
\item[(ii)] The operator  $\lambda+\bD(\rho)\in \kL(L_{2}(\bT))$ is invertible for all $\lambda\geq1$.
\end{itemize}
\end{prop}
\begin{proof} Let   $\rho  \in \cV_r$ be fixed.
Concerning~(i), we prove that there exists a constant~${C\geq1}$ such that for  all $\beta\in L_{2,0}(\bT)$ and $\lambda\in\R\setminus(-1,1)$ we have 
\begin{equation}\label{INV}
C\|(\lambda+\bD(\rho)^*)[\beta]\|_2\geq \|\beta\|_2.
\end{equation}
The claim then follows directly from~\eqref{INV} and the method of continuity; see, e.g., \cite[Proposition~I.1.1.1]{LQPP}.
To  this end, Young's inequality and~\eqref{Rel12222} imply there is a constant~$C \geq 1$ such that for all $\beta\in L_{2,0}(\bT)$
\begin{equation}\label{est1}
\|(\pm1-\bD(\rho)^*)[\beta]\|_2\leq C\|\bB(\rho)^*[\beta]\|_2,
\end{equation}
 and, as a direct consequence of \eqref{est1},
\begin{equation}\label{est2}
\|\beta\|_2\leq C\|\bB(\rho)^*[\beta]\|_2,\qquad \beta\in L_{2,0}(\bT).
\end{equation}
Substituting for $\lambda\in\R\setminus(-1,1)$ 
\begin{equation}\label{substitution}
(\pm1-\bD(\rho)^*)[\beta]=(\lambda\pm1)\beta-(\lambda+\bD(\rho)^*)[\beta]
\end{equation}
in \eqref{Rel12222}, yields 
\begin{align}
&\int_{-\pi}^\pi\frac{\rho^2}{\omega_\rho^2}\big(|(\lambda\pm1)\beta|^2-2(\lambda\pm1)\beta(\lambda+\bD(\rho)^*)[\beta]+\big|(\lambda+\bD(\rho)^*)[\beta]\big|^2-\big|\bB(\rho)^*[\beta]\big|^2\big)\,{\rm ds}\nonumber\\
&=\int_{-\pi}^\pi\frac{(\rho^2)'}{\omega_\rho^2} \big((\lambda\pm1)\beta\bB(\rho)^*[\beta]-(\lambda+\bD(\rho)^*)[\beta]\bB(\rho)^*[\beta]\big) \,{\rm ds}.\label{Rel12trans}
\end{align}
We next multiply the equation~\eqref{Rel12trans} with $+$ by $\lambda-1$ and the equation~\eqref{Rel12trans} with $-$ by~$-(\lambda+1)$ to obtain, after building the sum of the resulting identities,
\begin{equation*} 
\int_{-\pi}^\pi\frac{\rho^2}{\omega_\rho^2}\big((\lambda^2-1)|\beta|^2 -\big|(\lambda+\bD(\rho)^*)[\beta]\big|^2+\big|\bB(\rho)^*[\beta]\big|^2\big)\,{\rm ds}=\int_{-\pi}^\pi\frac{(\rho^2)'}{\omega_\rho^2}  (\lambda+\bD(\rho)^*)[\beta]\bB(\rho)^*[\beta] \,{\rm ds}.
\end{equation*}
Using  again Young's inequality, we find a constant $C\geq 1$ with the property  that for all~${\beta\in L_{2,0}(\bT)}$ and $\lambda\in\R\setminus(-1,1)$ we have
\[
(\lambda^2-1)\|\beta\|_2^2+\|\bB(\rho)^*[\beta]\|_2^2\leq C\|(\lambda+\bD(\rho)^*)[\beta]\|_2^2.
\]
This relation together with \eqref{est2}  immediately implies \eqref{INV},
and the proof of~(i) is complete.\medskip

Similarly, for~(ii) is suffices to  show that   there exists a constant $C\geq1$ such that for  all~${\beta\in L_{2}(\bT)}$
  and~$\lambda\geq1$  it holds
\begin{equation}\label{INV222}
C\|(\lambda+\bD(\rho))[\beta]\|_2\geq \|\beta\|_2.
\end{equation}
Arguing as  above, we find from~\eqref{Rel12222}  a constant $C \geq1$ such that for all $\beta\in L_{2}(\bT)$ we have
\begin{equation*} 
\|(\pm1-\bD(\rho)^*)[\beta]\|_2\leq C(\|\bB(\rho)^*[\beta]\|_2+|\langle \beta\rangle|) 
\end{equation*}
and  
\begin{equation}\label{est2222}
\|\beta\|_2\leq C(\|\bB(\rho)^*[\beta]\|_2+|\langle \beta\rangle|).
\end{equation}
Using the same substitution~\eqref{substitution} in \eqref{Rel12222} yields 
\begin{align}
&\int_{-\pi}^\pi\frac{\rho^2}{\omega_\rho^2}\big(|(\lambda\pm1)\beta|^2-2(\lambda\pm1)\beta(\lambda+\bD(\rho)^*)[\beta]+\big|(\lambda+\bD(\rho)^*)[\beta]\big|^2-\big|\bB(\rho)^*[\beta]\big|^2\big)\,{\rm ds}\nonumber\\
&=\int_{-\pi}^\pi\frac{(\rho^2)'}{\omega_\rho^2} \big((\lambda\pm1)\beta\bB(\rho)^*[\beta]-(\lambda+\bD(\rho)^*)[\beta]\bB(\rho)^*[\beta]\big) \,{\rm ds}-4(\pm 1-1)\pi|\langle\beta\rangle|^2.\label{Rel12trans222}
\end{align}
We next multiply the equation~\eqref{Rel12trans222} with $+$ by $\lambda-1$ and the equation~\eqref{Rel12trans222} with $-$ by~$-(\lambda+1)$ to obtain, after building the sum of the resulting identities,
\begin{equation*}
\begin{aligned}
&\int_{-\pi}^\pi\frac{\rho^2}{\omega_\rho^2}\big((\lambda^2-1)|\beta|^2 -\big|(\lambda+\bD(\rho)^*)[\beta]\big|^2+\big|\bB(\rho)^*[\beta]\big|^2\big)\,{\rm ds}\\
&=\int_{-\pi}^\pi\frac{(\rho^2)'}{\omega_\rho^2}  (\lambda+\bD(\rho)^*)[\beta]\bB(\rho)^*[\beta] \,{\rm ds}-8(\lambda+1)\pi|\langle\beta\rangle|^2.
\end{aligned}
\end{equation*}
Hence,  there is a constant $C\geq 1$ such that for all $\beta\in L_{2}(\bT)$ and $\lambda\in\R\setminus(-1,1)$ we have
\[
 (\lambda+1) |\langle\beta\rangle|^2+(\lambda^2-1)\|\beta\|_2^2+\|\bB(\rho)^*[\beta]\|_2^2\leq C\|(\lambda+\bD(\rho)^*)[\beta]\|_2^2.
\]
Combining the latter relation with \eqref{est2}, we obtain~\eqref{INV222}, which proves (ii).
\end{proof}

\subsection*{Invertibility of layer potentials in Sobolev spaces}
We now address the invertibility of $\lambda+\bD(\rho)$ and~$\lambda+\bD(\rho)^*$ in $H^r(\bT)$ and ${H}_0^{r-1}(\bT)$, respectively, for the same range of $\lambda$ as in Proposition~\ref{Prop:2}
 (recalling that $r\in(3/2,2]$ is arbitrary).

\begin{prop}\label{Prop:3}
 Let $\rho\in \cV_r$.
\begin{itemize}
\item[(i)]  The operator  $\lambda+\bD(\rho)^*\in \kL( H^{r-1}_0(\bT))$ is invertible for all~$\lambda\in\R\setminus(-1,1)$;
\item[(ii)] The operators $\lambda+\bD(\rho)\in \kL(H^{r-1}(\bT))$ and  $\lambda+\bD(\rho)\in \kL(H^r(\bT))$ are invertible for all~$\lambda\geq1$.
\end{itemize}
\end{prop}
\begin{proof} Let $\rho\in\cV_r$ be fixed.  Concerning~(i), it suffices to prove that there is a constant~$C\geq 1$ such that for all~${\beta\in  H^{r-1}_0(\bT)}$
  and $\lambda\in\R\setminus(-1,1)$ we have 
\begin{equation}\label{INV26i}
C\|(\lambda+\bD(\rho)^*)[\beta]\|_{H^{r-1}}\geq \|\beta\|_{H^{r-1}}.
\end{equation}
To this end we compute, in the particular case when $r\in(3/2,2)$, using \eqref{equivsem} and~\eqref{INV},
\begin{align*}
[\beta]_{H^{r-1}}^2&=\int_{-\pi}^\pi\frac{\|\sfT_s\beta-\beta\|_2^2}{|s|^{1+2 (r-1)}}\,{\rm d}s\leq C\int_{-\pi}^\pi\frac{\|(\lambda+\bD(\rho)^*)[\sfT_s\beta-\beta]\|_2^2}{|s|^{1+2 (r-1)}}\,{\rm d}s\\
&\leq C\int_{-\pi}^\pi\frac{\|\sfT_s\big((\lambda+\bD(\rho)^*)[\beta]\big)-(\lambda+\bD(\rho)^*)[\beta]\|_2^2}{|s|^{1+2 (r-1)}}\,{\rm d}s \\
&\quad+C\int_{-\pi}^\pi\frac{\|\sfT_s\big(\bD(\rho)^*[\beta]\big)-\bD(\rho)^*[\sfT_s\beta]\|_2^2}{|s|^{1+2 (r-1)}}\,{\rm d}s\\
&\leq C[(\lambda+\bD(\rho)^*)[\beta]]_{H^{r-1}}^2 +C\int_{-\pi}^\pi\frac{\|\bD(\sfT_s\rho)^*[\sfT_s\beta]\big)-\bD(\rho)^*[\sfT_s\beta]\|_2^2}{|s|^{1+2 (r-1)}}\,{\rm d}s.
\end{align*}
Let $r'\in(3/2,r)$ be fixed. 
Then, recalling the representation~\eqref{forD*} of $\bD(\rho)^* $ and \eqref{difference}, we deduce from Lemma~\ref{L:B1} (with $r=r'$ therein), Lemma~\ref{L:B1p}, 
and Lemma~\ref{L:B4} (with $r=r'$ therein) 
 that there is a constant~$C>0$ such that for all~${\beta\in  H^{r-1}_0(\bT)}$ and $s\in(-\pi,\pi)$ we have
\begin{equation*} 
 \|\bD(\sfT_s\rho)^*[\beta]-\bD(\rho)^*[\beta]\|_2 \leq C\|\sfT_s\rho-\rho\|_{H^{1}}\|\beta\|_{H^{r'-1}},
\end{equation*}
and together with the previous estimate we get
\[
[\beta]_{H^{r-1}}\leq C\big([(\lambda+\bD(\rho)^*)[\beta]]_{H^{r-1}}+ \|\beta\|_{H^{r'-1}}\big).
\]
  This estimate, together with \eqref{equiv} and  \eqref{INV}, implies  there exists a constant $C\geq1$ such that for all~$\lambda\in\R\setminus(-1,1)$ 
  and~${\beta\in  H^{r-1}_0(\bT)}$  we have 
\begin{equation}\label{zwischen}
\|\beta\|_{H^{r-1}}\leq C\big(\|(\lambda+\bD(\rho)^*)[\beta]\|_{H^{r-1}}+ \|\beta\|_{H^{r'-1}}\big). 
\end{equation}
The desired estimate~\eqref{INV26i} for $r\in(3/2,2)$ follows now from~\eqref{zwischen} and~\eqref{IP} since $r'\in(3/2,r)$.

If $r=2$,  we infer from \eqref{INV} that there exists a constant $C\geq 1$ such that  for all $\lambda\geq 1$ and~$\beta\in H^1(\bT)$ we have
\begin{align*}
C\|(\lambda+\bD(\rho)^*)[\beta]\|_{H^{1}}&\geq \|(\lambda+\bD(\rho)^*)[\beta]\|_{2}+\|((\lambda+\bD(\rho)^*)[\beta])'\|_{2}\\
&\ge C^{-1}\|\beta\|_{2}+\|(\lambda+\bD(\rho)^*)[\beta']\|_{2}-\|(\bD(\rho)^*[\beta])'-\bD(\rho)^*)[\beta']\|_2\\
&\geq C^{-1}(\|\beta\|_{H^1}-\|\beta \|_{H^{r'-1}}),
\end{align*}
where in the last line we used \eqref{forD*}, \eqref{forder}, Lemma~\ref{L:B1}, and Lemma~\ref{L:B4} (with $r=r'$  therein) to estimate
 \[
\|(\bD(\rho)^*[\beta])'-\bD(\rho)^*)[\beta']\|_2\leq C\|\beta\|_{H^{r'-1}},\qquad \beta\in H^1(\bT). 
 \]
 Thus, \eqref{zwischen} holds  also for $r=2$, and the estimate~\eqref{INV26i} follows similarly as for~$r\in(3/2,2)$.

To establish (ii), we note that   $[u\mapsto \|u\|_2+\|u'\|_{H^{r-1}}]$ is an  equivalent norm on $H^r(\bT)$. 
This property together with the estimates  \eqref{INV222}, \eqref{INV26i}, and Lemma~\ref{L:1} implies that 
there exists a constant~$C\geq 1$ such that for all $\lambda\geq 1$ and $\beta\in H^r(\bT)$ we have
\begin{align*}
C\|(\lambda+\bD(\rho))[\beta]\|_{H^{r}}&\geq \|(\lambda+\bD(\rho))[\beta]\|_{2}+\|((\lambda+\bD(\rho))[\beta])'\|_{H^{r-1}}\\
&= \|(\lambda+\bD(\rho))[\beta]\|_{2}+\|(\lambda-\bD(\rho)^*)[\beta']\|_{H^{r-1}}\\
&\geq C^{-1}(\|\beta\|_2+\|\beta'\|_{H^{r-1}})\geq C^{-2}\|\beta\|_{H^r},
\end{align*}
and the invertibility of $\lambda+\bD(\rho)\in \kL(H^r(\bT))$ follows from the method of continuity.
This property together with Proposition~\ref{Prop:2}~(ii) also ensures the invertibility of~${\lambda+\bD(\rho)\in \kL(H^{r-1}(\bT))}$  by interpolation (see~\eqref{IP}). 
\end{proof}

\subsection*{An equivalent formulation of \eqref{PB}} We may now  reformulate the Hele-Shaw problem~\eqref{PB} as an evolution problem for $\rho$, 
with nonlinearities expressed by singular integrals, see \eqref{QPP}.
In doing so, special attention is  required when solving the equation~\eqref{betat},  particularly in identifying the leading-order terms with respect to~$\rho$.

Assume  that $\rho:[0,T)\to\cV_r$ is a  solution to solution~\eqref{PB} enjoying the  regularity properties~\eqref{prop1} and~\eqref{prop2}.
Noticing that the functions introduced in \eqref{krhof} satisfy
 $f(\rho),\, \kappa(\rho)[\rho]\in H^{ r}(\bT)$ for~${\rho\in H^{r+2}(\bT)}$,  we have
 $\kappa_{\Gamma_{\rho(t)}}= \kappa(\rho(t))[\rho(t)]+f(\rho(t))\in H^{ r}(\bT)$ for all~$t\in(0,T)$.
Proposition~\ref{Prop:1} then ensures  that the boundary value problem~\eqref{PB1}$_{1}$-\eqref{PB1}$_2$ has for each $t\in(0,T)$ a unique solution $u(t)$ satisfying~\eqref{prop2}. 
Moreover,  recalling  \eqref{mtmn}, \eqref{trv}, and \eqref{AdjB}  we have
\begin{align}\label{u}
\p_{\mn_{\rho(t)}}u(t)(\Xi_{\rho(t)}(\tau))=\frac{1}{\omega_{\rho(t)}(\tau)} \bB(\rho(t))^*[\beta(t)'](\tau),\qquad\tau\in\R,
\end{align}
where $\beta(t)\in H^{r}(\bT)$ is the unique solution to 
\begin{equation} \label{betat}
(1+\bD(\rho(t)))[\beta(t)]= \kappa(\rho(t))[\rho(t)]+f(\rho(t)).
\end{equation}
We may decompose  $\beta(t)=\beta_1(t)+\beta_2(t)$, where $\beta_i(t)\in H^{r}(\bT)$, $i=1,\, 2$, are defined as the unique solutions to the equations
\begin{equation} \label{betat1}
(1+\bD(\rho(t)))[\beta_1(t)]= \kappa(\rho(t))[\rho(t)] 
\end{equation}
and
\begin{equation} \label{betat2}
(1+\bD(\rho(t)))[\beta_2(t)]= f(\rho(t)).
\end{equation}
Moreover, Lemma~\ref{L:1} and Proposition~\ref{Prop:3} ensure that $\beta_2(t)'\in  H^{ r-1}_0(\bT)$ solves the equation 
\begin{equation} \label{betat2'}
(1-\bD(\rho(t))^* )[\beta_2(t)']= (f(\rho(t)))',
\end{equation}
where, for $\rho\in \cV_{r+1}$,
\[
(f(\rho))'=\frac{\rho^2\rho'\rho''-\rho^3\rho'-4\rho\rho'^3-2\rho'^3\rho'' }{\omega_{\rho}^5}.
\]
Since  the normal velocity is for $t\in (0,T)$ given by
\[
V(t, \Xi_{\rho(t)}(\tau))=\Big(\frac{{\rm d}\rho}{dt}(t)\frac{\rho(t)}{\omega_{\rho(t)}}\Big)(\tau),\qquad \tau\in\R,
\] 
we infer from the  kinematic boundary condition \eqref{PB1}$_1$  together with \eqref{u} and \eqref{eqcom2} that we may recast \eqref{PB1} as  
\begin{equation*} 
\cfrac{{\rm d}\rho}{{\rm d}t}(t)=\Big(\frac{1}{\rho(t)} \bB(\rho(t))[\beta_1(t)]\Big)'+\frac{\rho(t)'}{\rho^2(t)}\bB(\rho(t))[\beta_1(t)]-\frac{1}{\rho(t)} \bB(\rho(t))^*[\beta_2(t)']
\end{equation*}
for $t\in (0,T)$.
We now define the mapping $\lambda:\cV_r\to\kL(H^{r+1}(\bT),   H^{ r-1}_0(\bT))$ by setting
\begin{equation}\label{defla}
\lambda(\rho)[h]:=\frac{\rho^2\rho'h''-\rho^3h'-4\rho\rho'^2h'-2\rho'^3h'' }{\omega_{\rho}^5}-\Big\langle\frac{\rho^2\rho'h''-\rho^3h'-4\rho\rho'^2h'-2\rho'^3h'' }{\omega_{\rho}^5}\Big\rangle,
\end{equation}
and note that $\lambda(\rho)[\rho]=(f(\rho))'$ for $\rho\in \cV_{r+1}$.
 Arguing as in \cite[Appendix~C]{MP2021}, it is not difficult to  prove that 
\begin{equation}\label{reg:fk}
\lambda\in{\rm C}^\infty(\cV_r, \kL(H^{r+1}(\bT),  H^{ r-1}_0(\bT))) \qquad\text{and}\qquad \kappa\in{\rm C}^\infty(\cV_r, \kL(H^{r+1}(\bT),H^{ r-1}(\bT))).
\end{equation}
Associated   with \eqref{betat1}-\eqref{betat2'}, we define for $\rho\in\cV_r$ and $h\in H^{ r+1}(\bT)$  
\begin{align}
\alpha_1(\rho)[h]&:=(1+\bD(\rho))^{-1}[\kappa(\rho)[h]],\label{alpha1}\\
\alpha_2(\rho)[h]&:=(1-\bD(\rho)^*)^{-1}[\lambda(\rho)[h]],\label{alpha2}
\end{align}
and infer from \eqref{reg:Ds}, \eqref{reg:fk}, and Proposition~\ref{Prop:3} that  
\begin{equation}\label{reg:alphas}
\alpha_1\in{\rm C}^\infty(\cV_r, \kL(H^{ r+1}(\bT),H^{r-1}(\bT))) \quad\text{and}\quad \alpha_2\in{\rm C}^\infty(\cV_r, \kL(H^{ r+1}(\bT),H^{ r-1}_0(\bT))).
\end{equation}
Moreover, if $\rho=h=\rho(t)\in \cV_{r+2}$ for some $t\in(0,T)$, then
\begin{equation}\label{partint}
\beta_1(t)=\alpha_1(\rho(t))[\rho(t)] \quad\text{and}\quad \beta_2(t)'=\alpha_2(\rho(t))[\rho(t)]. 
\end{equation}
Consequently, the Hele-Shaw problem \eqref{PB} can be recast as a quasilinear  evolution problem
\begin{equation}\label{QPP}
\cfrac{{\rm d}\rho}{{\rm d}t}(t)=\Phi(\rho(t))[\rho(t)],\quad t>0,\qquad \rho(0)=\rho_0,
\end{equation}
where the operator $\Phi:=\Phi_1+\Phi_2:\cV_r\to \kL(H^{ r+1}(\bT),H^{ r-2}(\bT))$ is defined by
\begin{align}
\Phi_1(\rho)[h]&:=\Big(\frac{1}{\rho} \bB(\rho)[\alpha_1(\rho)[h]]\Big)',\label{Phi1}\\
\Phi_2(\rho)[h]&:=\frac{\rho'}{\rho^2}\bB(\rho)[\alpha_1(\rho)[h]]-\frac{1}{\rho} \bB(\rho)^*[\alpha_2(\rho)[h]].\label{Phi2}
\end{align}
Recalling~\eqref{reg:Bs}, we infer from \eqref{reg:alphas} that
\begin{equation}\label{reg:Phis}
\Phi_1\in{\rm C}^\infty(\cV_r, \kL(H^{ r+1}(\bT),H^{ r-2}(\bT))) \quad\text{and}\quad \Phi_2\in{\rm C}^\infty(\cV_r, \kL(H^{ r+1}(\bT),  H^{ r-1}(\bT))),
\end{equation}
in particular 
\begin{equation}\label{reg:Phi}
\Phi\in{\rm C}^\infty(\cV_r, \kL(H^{ r+1}(\bT),H^{ r-2}(\bT))).
\end{equation}
Our goal is to apply the quasilinear parabolic theory  developed in~\cite{Am93} (see also \cite{MW20}) to the problem~\eqref{QPP} in order to establish Theorem~\ref{MT1}.
This  will be carried out in the next section, where  we also verify the remaining assumption -- namely, that $\Phi(\rho)$ generates an analytic semigroup  on~$H^{ r-2}(\bT)$ for each $\rho \in \cV_r$.

\section{Proof of Theorem~\ref{MT1}}\label{SEC:3}

This section is devoted to the proof of Theorem~\ref{MT1}, which is presented at the end of the section. 
In the first part, we prove that the problem \eqref{QPP} is of parabolic type by showing that $\Phi(\rho)$, viewed as an unbounded operator in~$H^{r-2}(\bT)$  with domain in~$H^{r+1}(\bT)$, 
 generates  a strongly continuous  analytic semigroup for each~$\rho\in\cV_r$. 
  Theorem~\ref{MT1} is then a consequence of general theory for quasilinear parabolic problems.
 
 Also  in this section $r\in(3/2,2]$ is arbitrarily fixed.

\begin{prop}\label{Prop:4} 
Given $\rho\in\cV_r$,  the operator $\Phi(\rho)$  belongs to~$\kH(H^{r+1}(\bT),H^{r-2}(\bT))$.
\end{prop}

The proof of Proposition~\ref{Prop:4} requires some preliminaries. To this end, we fix  $\rho\in\cV_r$  and choose~$r'\in(3/2,r)$.
Since $\Phi_2(\rho)\in\kL(H^{r+1}(\bT),H^{r-1}(\bT))$ is a lower order perturbation  due to~\eqref{reg:Phis},
 it suffices to show that $\Phi_1(\rho)\in\kH(H^{r+1}(\bT),H^{r-2}(\bT))$; 
see \cite[I. Theorem 1.3.1]{LQPP} and \eqref{IP}.
 Recalling that $H = 2B^{0}_{0,1}(1)$ is the periodic Hilbert transform, we infer from \eqref{krhof}, \eqref{forD}, \eqref{forB}, and~\eqref{Phi1} that 
\begin{equation}\label{phi1fs}
\Phi_1(1)=H\circ\frac{{\rm d}^3}{{\rm d}\tau^3}
\end{equation}
is a Fourier multiplier with symbol $(-|k|^3)_{k\in\Z}$.
The key argument in the proof of Proposition~\ref{Prop:4} is established in Lemma~\ref{L:loc} below.  
There, the operator $\Phi_1(\theta\rho+1-\theta)$,  for~${\tau \in [0,1]}$, is localized using Fourier multipliers, whose principal part coincides, up to a multiplicative positive constant, with that of~$\Phi_1(1)$.
 In order to  give precise statements,  we fix for each~$\e\in(0,1)$ an {\it $\e$-partition of unity;} that is, a set~$\{\pi_{j}^{\e}: 1\leq j\leq  q(\e)\}\subset {\rm C}^\infty(\bT,[0,1])$ with~$q(\e)\in\N$  sufficiently large,  such that
	\begin{equation}\label{pi_j}
	\begin{aligned}
		\bullet & \,\supp \pi_j^\e= I_j^\e+2\pi \Z\text{ with } I_j^\e:= [\tau_j^\e-\e,\tau_j^\e+\e] \text{ and } \tau_j^\e:=j\e;\\
		\bullet & \, \sum_{j=1}^{q(\e)}\pi_j^\e =1 \text{ in } \bT. 
	\end{aligned}
	\end{equation}	 
	Moreover, we  associate  with each $\e$-partition of unity a set $\{\chi_j^\e: 1\leq j \leq q(\e)\}\subset{\rm C}^\infty(\bT,[0,1])$  satisfying
	\begin{equation}\label{chi_j}
	\begin{aligned}
		\bullet & \,\supp \chi_j^\e =J_j^\e+2\pi \Z \text{ with } J_j^\e=[\tau_j^\e-2\e,\tau_j^\e+2\e] ;\\
		\bullet & \,\chi_j^\e =1 \text{ on } \supp \pi_j^\e.
	\end{aligned}
	\end{equation}
Then, for each  $\e \in (0,1)$ and  $s \in\R$, there exists  a constant $C = C(\e, s) \geq 1$  such that
\begin{equation}\label{eqnorm}
C^{-1} \|\rho\|_{H^{s}} \leq \sum_{j=1}^{q(\e)} \|\pi_j^\e \rho\|_{H^{s}} \leq C \|\rho\|_{H^{s}}, \qquad \rho \in H^{s}(\bT),
\end{equation}
meaning that the middle term defines an equivalent norm on $H^{s}(\bT)$.  
We first establish a technical result for the localization.

\begin{lemma}\label{L:localpha} Set $\rho_\theta:=\theta\rho+1-\theta$  for $\theta\in[0,1]$.
\begin{itemize}
\item[(i)] There is a constant~${C_0>0}$ and for each $\e\in(0,1)$ there is a constant~${K=K(\e)>0}$ such that for all  $0\leq j\leq q(\e)$, $\theta\in[0,1]$, and $ h\in H^{r-1}(\bT)$ it holds that
\begin{equation}\label{localph0}
\|\pi^\e_j\alpha_1(\rho_\theta)[h]\|_{H^{r-1}}\leq C_0\|\pi^\e_j h\|_{H^{r+1}}+K\|h\|_{H^{r'+1}}.
\end{equation}
\item[(ii)]  Let  $\mu>0$. For each sufficiently small $\e\in(0,1)$ there is a constant~${K=K(\e)>0}$ such that for all  $0\leq j\leq q(\e)$, $\theta\in[0,1]$, and $ h\in H^{r-1}(\bT)$  it holds that
\begin{equation}\label{localph1}
\qquad\Big\|\pi^\e_j\alpha_1(\rho_\theta)[h]+\frac{\rho_\theta}{\omega^3_{\rho_\theta}}(\tau_j^\e)(\pi_j^\e h)''\Big\|_{H^{r-1}}\leq \mu\|\pi^\e_j h\|_{H^{r+1}}+K\|h\|_{H^{r'+1}}.
\end{equation}
\end{itemize}
\end{lemma}

\begin{proof} 
In the following   we denote by $C$ constants that are independent on $\e$, whereas constants that may depend  $\e$  are denoted by $K$.

To prove (i),  let~$\e\in(0,1)$,  $0\leq j\leq q(\e)$, $\theta\in[0,1]$, and~${h\in H^{r+1}(\bT)}$. 
 We then infer from \eqref{krhof} and~\eqref{alpha1} that 
\begin{equation*}
(1+\bD(\rho_\theta))[\pi_j^\e\alpha_1(\rho_\theta)[h]]=-\frac{\rho_\theta}{\omega_{\rho_\theta}^3}\pi_j^\e h''-\llbracket\pi_j^\e, \bD(\rho_\theta) \rrbracket[\alpha_1(\rho_\theta)[h]],
\end{equation*}
 where we recall the notation~\eqref{comm}.
Combining \eqref{reg:Ds} and  Proposition~\ref{Prop:3}, we then get
\begin{align*}
\|\pi^\e_j\alpha_1(\rho_\theta)[h]\|_{H^{r-1}}&\leq C\Big\|\frac{\rho_\theta}{\omega_{\rho_\theta}^3}\pi_j^\e h''\Big\|_{H^{r-1}}+ C\|\llbracket\pi_j^\e, \bD(\rho_\theta) \rrbracket[\alpha_1(\rho_\theta)[h]]\|_{H^{r-1}}\\
&\leq C_0\|\pi^\e_j h\|_{H^{r+1}}+K\|h\|_{H^{r}}+C\|\llbracket\pi_j^\e, \bD(\rho_\theta) \rrbracket[\alpha_1(\rho_\theta)[h]]\|_{H^{r-1}}.
\end{align*}
Moreover, in view of Lemma~\ref{L:AB1}, \eqref{forD}, and \eqref{reg:alphas} (with $r=r'$ therein), we have
\[
\|\llbracket\pi_j^\e, \bD(\rho_\theta) \rrbracket[\alpha_1(\rho_\theta)[h]]\|_{H^{r-1}}\leq K\|\alpha_1(\rho_\theta)[h]\|_2\leq K\|h\|_{H^{r'+1}},
\]
and~\eqref{localph0} follows.

To establish (ii), we multiply~\eqref{alpha1} by $\pi_j^\e$ and obtain that
\begin{equation}\label{hsv1}
\pi_j^\e\alpha_1(\rho_\theta)[h]+\frac{\rho_\theta}{\omega^3_{\rho_\theta}}(\tau_j^\e)(\pi_j^\e h)''=\pi_j^\e\kappa(\rho_\theta)[h]+\frac{\rho_\theta}{\omega^3_{\rho_\theta}}(\tau_j^\e)(\pi_j^\e h)''+\pi_j^\e \bD(\rho_\theta)[\alpha_1(\rho_\theta)[h]].
\end{equation}
In view of \eqref{forD}, \eqref{reg:alphas} (with $r=r'$ therein), \eqref{localph0}, and Lemma~\ref{L:AB2}, for  sufficiently small~${\e\in(0,1)}$ and all $0\leq j\leq q(\e)$, $\theta\in[0,1]$, and~${h\in H^{r+1}(\bT)}$   we have 
\begin{equation}\label{hsv2}
\begin{aligned}
\|\pi_j^\e \bD(\rho_\theta)[\alpha_1(\rho_\theta)[h]]\|_{H^{r-1}}&\leq \frac{\mu}{2C_0}\|\pi_j^\e\alpha_1(\rho_\theta)[h]\|_{H^{r+1}}+K\|\alpha_1(\rho_\theta)[h]\|_{H^{r'+1}}\\
&\leq (\mu/2)\|\pi^\e_j h\|_{H^{r+1}}+K\|h\|_{H^{r'+1}}.
\end{aligned}
\end{equation}
 Moreover, recalling~\eqref{krhof}, we use the identity $\chi_j^\e\pi_j^\e=\pi_j^\e$ and \eqref{algebra} to estimate 
 \begin{equation}\label{hsv3}
\begin{aligned}
&\Big\|\pi_j^\e\kappa(\rho_\theta)[h]+\frac{\rho_\theta}{\omega_{\rho_\theta}}(\tau_j^\e)(\pi_j^\e h)''\Big\|_{H^{r-1}}\\
&\leq \Big\|\chi_j^\e\Big(\frac{\rho_\theta}{\omega^3_{\rho_\theta}}-\frac{\rho_\theta}{\omega^3_{\rho_\theta}}(\tau_j^\e)\Big)(\pi_j^\e h)''\Big\|_{H^{r-1}}+K\|h\|_{H^{r'+1}}\\
&\leq C\Big\|\chi_j^\e\Big(\frac{\rho_\theta}{\omega^3_{\rho_\theta}}-\frac{\rho_\theta}{\omega^3_{\rho_\theta}}(\tau_j^\e)\Big)\Big\|_\infty\| \pi_j^\e h \|_{H^{r+1}}+K\|h\|_{H^{r'+1}}\\
&\leq (\mu/2)\|\pi^\e_j h\|_{H^{r+1}}+K\|h\|_{H^{r'+1}},
\end{aligned}
\end{equation}
where the H\"older continuity of the function $\rho_\theta\omega_{\rho_\theta}^{-3}$ has been used in the last step.
The estimate~\eqref{localph1} follows immediately from \eqref{hsv1}-\eqref{hsv3}. 
\end{proof}

 Now we can state and prove the key step of the localization procedure:
\begin{lemma}\label{L:loc}
 Let  $\mu>0$. Then, for each sufficiently small $\e\in(0,1)$ there exists  a positive  constant~${K=K(\e)}$ such that for all $0\leq j\leq q(\e)$, $\theta\in[0,1]$, and $ h\in H^{r-1}(\bT)$ it holds that
\begin{equation}\label{localeq}
\|\pi^\e_j\Phi_1(\theta\rho+1-\theta)[h]-\omega^{-3}_{\theta\rho+1-\theta}(\tau_j^\e)\Phi_1(1)[\pi^\e_j h]\|_{H^{r-2}}\leq\mu\|\pi^\e_j h\|_{H^{r+1}}+K\|h\|_{H^{r'+1}}.
\end{equation}
\end{lemma}

\begin{proof} We use the same convention regarding the constants  $C$ and $K$ as in the proof of  Lemma~\ref{L:localpha} and we set again $\rho_\theta:=\theta\rho+1-\theta$  for $\theta\in[0,1]$.
To start,  for $\e\in(0,1)$, $0\leq j\leq q(\e)$, $\theta\in[0,1]$, and~$ h\in H^{r-1}(\bT)$, we note  from \eqref{Phi1} and \eqref{phi1fs} that
\begin{align*}
&\|\pi_j^\e\Phi_1(\rho_\theta)[h]-\omega^{-3}_{\rho_\theta}(\tau_j^\e)\Phi_1(1)[\pi^\e_j h]\|_{H^{r-2}}\\
&\leq \Big\|\Big(\frac{\pi_j^\e}{\rho_\theta} \bB(\rho_\theta)[\alpha_1(\rho_\theta)[h]]-\omega^{-3}_{\rho_\theta}(\tau_j^\e)H[(\pi_j^\e h)'']\Big)'\Big\|_{H^{r-2}}
+\Big\|\frac{(\pi_j^\e)'}{\rho_\theta} \bB(\rho_\theta)[\alpha_1(\rho_\theta)[h]]\Big\|_{H^{r-2}},
\end{align*}
and, by  \eqref{reg:Bs}  and \eqref{reg:alphas}    (both with $r=r'$ therein), 
\begin{align*}
\Big\|\frac{(\pi_j^\e)'}{\rho_\theta} \bB(\rho_\theta)[\alpha_1(\rho_\theta)[h]]\Big\|_{H^{r-2}}&\leq \Big\|\frac{(\pi_j^\e)'}{\rho_\theta} \bB(\rho_\theta)[\alpha_1(\rho_\theta)[h]]\Big\|_{H^{r'-1}}\leq K\|h\|_{H^{r'+1}}.
\end{align*}
Since $\|\varrho'\|_{H^{r-2}}\leq \|\varrho\|_{H^{r-1}}$ for $\varrho \in H^{r-1}(\bT)$, we thus have
\begin{equation}\label{uta1}
\begin{aligned}
&\|\pi_j^\e\Phi_1(\rho_\theta)[h]-\omega^{-3}_{\tau\rho}(\tau_j^\e)\Phi_1(1)[\pi^\e_j h]\|_{H^{r-2}}\\
&\leq \Big\|\frac{\pi_j^\e}{\rho_\theta} \bB(\rho_\theta)[\alpha_1(\rho_\theta)[h]]-\omega^{-3}_{\rho_\theta}(\tau_j^\e)H[(\pi_j^\e h)'']\Big\|_{H^{r-1}}
+K\|h\|_{H^{r'+1}},
\end{aligned}
\end{equation}
and it remains to estimate the first term on the right-hand side.

Recalling \eqref{forB}, we infer from \eqref{reg:alphas}   (with $r=r'$ therein), Lemma~\ref{L:B3p},  Lemma~\ref{L:AB2}, and Lemma~\ref{L:localpha}   that
\begin{equation}\label{uta2}
\begin{aligned}
&\Big\|\frac{\pi_j^\e}{\rho_\theta} \bB(\rho_\theta)[\alpha_1(\rho_\theta)[h]]+\frac{1}{\rho_\theta}(\tau_j^\e)H[\pi_j^\e \alpha_1(\rho_\theta)[h]]\Big\|_{H^{r-1}}\\
&\leq (\mu/(2C_0))\|\pi^\e_j \alpha_1(\rho_\theta)[h]]\|_{H^{r-1}}+K\|\alpha_1(\rho_\theta)[h]]\|_{H^{r'-1}}\\
&\leq (\mu/2)\|\pi^\e_j h\|_{H^{r+1}}+K\|h\|_{H^{r'+1}}
\end{aligned}
\end{equation}
provided that $\e\in(0,1)$ is sufficiently small,  where $C_0>0$ stems from \eqref{localph0}.

Let $C_1>0$ be chosen such that $\|u \rho_\theta \|_{H^{r-1}}\leq C_1\|u\|_{H^{r-1}}$ for all $u\in H^{r-1}(\bT)$ and~$\theta\in[0,1]$.
   Lemma~\ref{L:localpha}~(ii) ensures   for~$\e\in(0,1)$ sufficiently small that
\begin{equation}\label{uta3}
\begin{aligned}
&\Big\|\frac{1}{\rho_\theta}(\tau_j^\e)H[\pi_j^\e \alpha_1(\rho_\theta )[h]]+\omega^{-3}_{\rho_\theta}(\tau_j^\e)H[(\pi_j^\e h)'']\Big\|_{H^{r-1}}\\
&\leq C_1\Big\|\pi^\e_j \alpha_1(\rho_\theta )[h]] +\frac{\rho_\theta }{\omega_{\tau\rho}^3}(\tau_j^\e)(\pi_j^\e h)''\Big\|_{H^{r-1}} \\
&\leq (\mu/2)\|\pi^\e_j h\|_{H^{r+1}}+K\|h\|_{H^{r'+1}}.
\end{aligned}
\end{equation}
Gathering \eqref{uta1}-\eqref{uta3}, we  conclude \eqref{localeq} and the proof is complete.

\end{proof}

We next provide the proof of Proposition~\ref{Prop:4}.

\begin{proof}[Proof of Proposition~\ref{Prop:4}] 
Since $\Phi_2(\rho)\in\kL(H^{r+1}(\bT),H^{r-1}(\bT))$  due to \eqref{reg:Phis} and in view of  \cite[I. Theorem 1.3.1]{LQPP}, 
we may treat the operator $\Phi_2(\rho)$ as a lower order perturbation of~$\Phi_1(\rho)$ and 
therefore it suffices to prove that the principal part~$\Phi_1(\rho)$ generates  a strongly continuous  analytic semigroup in~$H^{r-2}(\bT)$.
To establish this property we note that, since $\rho\in\cV_r$, there is a constant~${C=C(\rho)>1}$ such  that 
$\omega_{\rho_\theta}^{-3}(x)\in[C^{-1},C]$ for all~$x\in\R$ and~$\theta\in[0,1]$, where we set  again ~$\rho_\theta:=\theta\rho+1-\theta$.
The formula~\eqref{phi1fs} and standard Fourier analysis enable us now to conclude that there exists a constant~${\kappa_0\geq 1}$ such that 
for all $a\in[C^{-1},C]$, $\lambda\in\mathbb{C}$ with $\re\lambda\geq 1$, and $h\in H^{r+1}(\bT)$ we have
\begin{equation}\label{defk0}
\kappa_0\|(\lambda- a\Phi_1(1))[h]\|_{H^{r-2}}\geq  |\lambda|\,\|h\|_{H^{r-2}}+\|h\|_{H^{r+1}}.
\end{equation}
Fix $r'\in(3/2,r)$. Lemma~\ref{L:loc} with $\mu=(2\kappa_0)^{-1}$ then ensures  there exists $\e\in(0,1)$ and a constant~${K=K(\e)>0}$
 such that, for all $0\leq j\leq q(\e)$, $\theta\in[0,1]$, and $ h\in H^{r-1}(\bT)$, 
\begin{equation*}
\kappa_0\|\pi^\e_j\Phi_1(\rho_\theta)[h]-\omega^{-3}_{\rho_\theta}(\tau_j^\e)\Phi_1(1)[\pi^\e_j h]\|_{H^{r-2}}\leq 2^{-1}\|\pi^\e_j h\|_{H^{r+1}}+\kappa_0K\|h\|_{H^{r'+1}}.
\end{equation*} 
This estimate and~\eqref{defk0} imply for $\lambda\in\mathbb{C}$ with $\re\lambda\geq 1$, $0\leq j\leq q(\e)$, $\theta\in[0,1]$, and $h\in H^{r+1}(\bT)$ that
\begin{align*}
\kappa_0\|\pi_j^\e(\lambda-\Phi_1(\rho_\theta))[h]\|_{H^{r-2}}&\geq \kappa_0\|(\lambda- \omega^{-3}_{\rho_\theta}(\tau_j^\e)\Phi_1(1))[\pi_j^\e h]\|_{H^{r-2}}\\
&\quad- \kappa_0\|\pi^\e_j\Phi_1(\rho_\theta)[ h]-\omega^{-3}_{\rho_\theta}(\tau_j^\e)\Phi_1(1)[\pi^\e_j h]\|_{H^{r-2}}\\
&\geq  |\lambda|\,\|\pi_j^\e h\|_{H^{r-2}}+2^{-1}\|\pi_j^\e h\|_{H^{r+1}}-\kappa_0K\|h\|_{H^{r'+1}}.
\end{align*} 
Summing up over $0\leq j\leq q(\e)$ and using \eqref{IP}, \eqref{eqnorm}, and Young's inequality we conclude that there exist constants $\omega>0$ and $\kappa\geq 1$ such that for all~$\lambda\in\mathbb{C}$ with 
$\re\lambda\geq \omega$,  $\theta\in [0,1]$, and $h\in H^{r+1}(\bT)$
\begin{align}\label{estres}
\kappa\|(\lambda-\Phi_1(\rho_\theta))[h]\|_{H^{r-2}}\geq  |\lambda|\,\| h\|_{H^{r-2}}+\| h\|_{H^{r+1}}.
\end{align} 
 Since $\omega-\Phi_1(1)\in\kL(H^{r+1}(\bT), H^{r-2}(\bT))$ is invertible  by \eqref{phi1fs}, the method of continuity together with~\eqref{estres} implies that  $\omega-\Phi_1(\rho)\in\kL(H^{r+1}(\bT), H^{r-2}(\bT))$  is  also invertible.
 Consequently, it follows from \cite[I.Theorem~1.2.2]{LQPP} and~\eqref{estres} (with $\theta=1$)  that $\Phi(\rho)\in\kH(H^{r+1}(\bT),H^{r-2}(\bT))$ as claimed. 
\end{proof}

We conclude this section with the proof of Theorem~\ref{MT1}.

\begin{proof}[Proof of Theorem~\ref{MT1}] Let $3/2<r<\bar r<2$.
As shown in Section~\ref{SEC:2}, the Hele-Shaw problem~\eqref{PB} is equivalent to the quasilinear problem~\eqref{QPP}, where $\Phi:\cV_r \to \kH(H^{r+1}(\bT), H^{r-2}(\bT))$ is  smooth according to \eqref{reg:Phi} and  Proposition~\ref{Prop:4}.
Recalling \eqref{IP}, we may thus apply the quasilinear parabolic theory from~\cite[Section 12]{Am93} (see also~\cite[Theorem~1.1 and Remark~1.2]{MW20})  and conclude that problem~\eqref{QPP} has for each $\rho_0\in\cV_{\bar r}$ 
a unique maximal classical solution $\rho=\rho(\cdot;\rho_0)$ with maximal existence time $T^+=T^+(\rho_0)$ that satisfies \eqref{prop1} and \eqref{prop1'}.
Moreover, the solution map $[(t,\rho_0)\mapsto\rho(t;\rho_0)]$ is a semiflow on~$\cV_{\bar r}$ which is smooth in the open set
defined in~\eqref{openset}.

It remains to prove the parabolic smoothing property \eqref{prop3},  which together with Proposition~\ref{Prop:1} immediately yields~\eqref{prop2}. 
To this end, we employ a parameter trick  that is used in other contexts as well; see \cite{An90, ES96, PSS15}. 
A direct application to solutions with initial data in $\mathcal{V}_{\bar r}$ is not possible, since the space~${\mathrm{C}^{\eta}([0,T^+), H^r(\mathbb{T}))}$, with $\eta\in(0,1)$,
 is not invariant under the scaling introduced in~\eqref{scaling}. 
 We therefore proceed in three steps:  
\begin{itemize}
    \item[(I)]  We construct more regular  maximal solutions $\wt\rho(\cdot;\rho_0)$ starting at~${\rho_0\in\mathcal{V}_{r+1}}$ (with a stronger uniqueness criterion). 
        \item[(II)] We prove that (the maximal existence time of) $\wt\rho(\cdot;\rho_0)$ and $\rho(\cdot;\rho_0)$ coincide. 
    \item[(III)] We establish the smoothing property \eqref{prop2} for  $ \rho(\cdot;\rho_0)$ with ${\rho_0\in\mathcal{V}_{r+1}}$.  
\end{itemize}\medskip

\noindent{ Regarding (I)}: Since  \eqref{reg:Phi} and Proposition~\ref{Prop:4} with $r=2$ imply that~$\Phi:\cV_2 \to \kH(H^{3}(\bT), L_2(\bT))$ is smooth,  
\cite[Theorem~1.1 and Remark~1.2]{MW20} guarantee, for each $\rho_0\in\cV_{r+1}$,  
that problem~\eqref{QPP} has a unique maximal classical solution $\wt\rho=\wt\rho(\cdot;\rho_0)$ such that 
\begin{equation}\label{prop12}
\wt \rho\in \mathrm{C}([0,\wt T^+),\cV_{r+1})\cap \mathrm{C}((0,\wt T^+), H^{3}(\bT)) \cap \mathrm{C}^1((0,\wt T^+), L_2(\bT)),
\end{equation}
and
\begin{equation}\label{prop1'2}
\wt\rho\in \mathrm{C}^{\alpha}([0,\wt T^+), H^2(\bT))\qquad\text{for some $\alpha\in(0,1)$,}
\end{equation}
with maximal existence time $\wt T^+=\wt T^+(\rho_0)\leq T^+(\rho_0)$ (and obviously $\wt\rho=\rho$ on $[0,\wt T^+)$).
We can improve the uniqueness statement and show that the solution is unique within the regularity class~\eqref{prop12}.
Indeed,  let $T<\wt T^+$ be fixed arbitrarily. 
Since $\wt \rho \in \mathrm{C}([0,T], \mathcal{V}_{r+1})$, \eqref{reg:fk} yields
\[
\kappa(\wt\rho)[\wt\rho] \in \mathrm{C}([0,T], H^{r-1}(\mathbb{T})), \qquad \lambda(\wt\rho)[\wt\rho] \in \mathrm{C}([0,T], {H}^{r-1}_0(\mathbb{T})),
\]
and, by~\eqref{reg:Ds}, we get   $\alpha_i(\wt\rho)[\wt\rho] \in \mathrm{C}([0,T],  {H}^{r-1}(\mathbb{T}))$, $i = 1,2$.
Using~\eqref{reg:Bs} and~\eqref{Phi1}-\eqref{Phi2}, we conclude
\[
\frac{{\rm d} \wt \rho}{{\rm d}t} \in \mathrm{C}([0,T], {H}^{r-2}(\mathbb{T}))
\]
and thus, by the fundamental theorem of calculus and~\eqref{IP}, we have $\wt \rho \in \mathrm{C}^\alpha([0,T], H^2(\mathbb{T}))$ with~${\alpha = (r-1)/3}$, 
which proves  in view of \cite[Remark~1.2~(ii)]{MW20} the uniqueness claim.
\medskip

\noindent{Regarding (II):} We only need to show that   $\wt T^+=T^+$. 
To this end, we assume by contradiction that~$\wt T^+<T^+$. Then, since $\rho\in {\rm C}((0,\wt T^+], \cV_{r+1})$ by \eqref{prop1}, 
  \cite[Proposition 2.1]{MW20} together with the arguments leading to the improved uniqueness claim in the previous step ensures that there exist constants~${\e>0}$ and~$\delta>0$ such that for all $\rho_*\in \cV_{r+1}$ with $\|\rho_*-\rho(\wt T^+)\|_{H^{r+1}}<\e$, the evolution problem
\[
\cfrac{{\rm d}\rho}{{\rm d}t}(t)=\Phi(\rho(t))[\rho(t)],\quad t>0,\qquad \rho(0)=\rho_*,
\]
has a unique classical solution $\bar\rho=\bar\rho(\cdot;\rho_*):[0,\delta]\to \cV_{r+1}$ enjoying~\eqref{prop12} with  $\wt T^+$ replaced by $\delta$.
Choosing $t_0<\wt T^+$ such that $\|\rho(t_0)-\rho(\wt T^+)\|_{H^{r+1}}<\e$ and $t_0+\delta>\wt T^+$, we conclude that the map
\[
t\mapsto
\left\{
\arraycolsep=1.4pt
\begin{array}{llll}
\wt \rho(t;\rho_0),&& 0\leq t<\wt T^+,\\[1ex]
\bar\rho(t-t_0;\rho(t_0)),&& t_0\leq t< t_0+\delta , 
\end{array}
\right.
\]
is a solution to \eqref{QPP} satisfying~\eqref{prop12} on $[0,t_0+\delta)$ with $t_0+\delta>\wt T^+$, which contradicts the maximality property of $\wt \rho(\cdot;\rho_0)$ and thus~$\wt T^+=T^+$.
 
  \medskip

\noindent{Regarding (III):} We  prove that   $\rho(\cdot;\rho_0)\in \mathrm{C}^\infty((0, T^+)\times\bT)$ for each $\rho_0\in\cV_{r+1}$.
To this end, we define for each~$\lambda=(\lambda_1,\lambda_2)\in(0,\infty)^2$   the function $\rho_\lambda$ by
\begin{equation}\label{scaling}
\rho_\lambda(t)(\tau):=\rho(\lambda_1 t)(\tau+\lambda_2t),\qquad\tau\in\R,\ 0\leq t<T_\lambda:=T^+/\lambda_1,
\end{equation}
which  satisfies
$$
\rho_\lambda\in \mathrm{C}([0, T_\lambda),\cV_{r+1})\cap \mathrm{C}((0, T_\lambda), H^{3}(\bT)) \cap \mathrm{C}^1((0, T_\lambda), L_2(\bT)).
$$
We next  introduce the function $u:[0,T_\lambda)\to(0,\infty)^2\times\cV_{r+1}$  by 
\[
u(t):=(\lambda_1,\lambda_2,\rho_\lambda(t)).
\]
 It is straightforward to prove that $u$ solves the quasilinear problem 
\begin{equation}\label{PDQPP}
\cfrac{{\rm d}u}{{\rm d}t}(t)=\Psi(u(t))[u(t)] ,\quad t>0,\qquad u(0)=(\lambda,\rho_0),
\end{equation}
where $\Psi:(0,\infty)^2\times\cV_2\to\kL(\R^2\times H^{3}(\bT),\R^2\times L_2(\bT))$ is defined by
\[
\Psi(u)[v]=(0,\lambda_1\Phi(\rho)[h]+\lambda_2h'),\qquad u=(\lambda,\rho),\, v=(\mu, h).
\]
Then,~\eqref{reg:Phi} (with $r=2$ therein) implies that $\Psi\in {\rm C}^\infty((0,\infty)^2\times\cV_2,\kL(\R^2\times H^{3}(\bT),\R^2\times L_2(\bT)))$. 
Since 
\[
 \Psi(u)=
 \begin{pmatrix}
 0&0\\[1ex]
 0&\lambda_1\Phi(\rho)+\lambda_2\cfrac{{\rm d}}{{\rm d}x}
 \end{pmatrix},\qquad u=(\lambda,\rho)\in (0,\infty)^2\times\cV_2,
\]
where ${\rm d}/{\rm d}x$ is a  lower order perturbation of $\lambda_1\Phi(\rho)$,  Proposition~\ref{Prop:4} (with $r=2$ therein) and \cite[I.Theorem~1.6.1]{Am93} ensure that the operator $\Psi(u)$ generates a strongly continuous and analytic semigroup
in $\R^2\times L_2(\bT)$ for each $u\in (0,\infty)\times \cV_2$. 
Arguing as in part (I) of the proof, we deduce that problem~\eqref{PDQPP} has for each~$u_0=(\lambda,\rho_0)\in (0,\infty)^2\times \cV_{r+1}$ a unique maximal solution $u=u(\cdot;u_0)$ with $u=(\lambda,\bar\rho)$ and 
\begin{equation}\label{propbr}
\bar \rho\in \mathrm{C}([0,t^+),\cV_{r+1})\cap \mathrm{C}((0,t^+), H^{3}(\bT)) \cap \mathrm{C}^1((0,t^+), L_2(\bT)),
\end{equation}
where $t^+=t^+(u_0)$  is  the maximal existence time of $u$.
Since the problems~\eqref{QPP} and \eqref{PDQPP} are equivalent we actually have  $t^+=T^+/\lambda_1$ and $\bar\rho=\rho_\lambda$.
Moreover, the set
\[
\mathcal{D}:=\{(t,u_0)\,:\, u_0\in(0,\infty)^2\times\cV_{r+1},\, 0<t<t^+(u_0)\}
\]
is an open subset of $\R^2\times \cV_{r+1}$ and the solution map $[(t,u_0)\mapsto u(t;u_0)]:\mathcal{D}\to\R^2\times\cV_{r+1}$ is smooth.

Let now $(t_0,\tau_0)\in (0,T^+)\times\R$ be fixed and choose $\e>0$ such that for all $ \lambda $  belonging to the ball~$B_\e((1,1))$  in $\R^2$  centered at $(1,1)$ of radius $\e$ we have $t_0<T^+/\lambda_1$.
This implies in particular that 
$\{t_0\}\times B_\e((1,1))\times\{\rho_0\}\subset \mathcal{D},$
the function $[\lambda\mapsto \rho_\lambda(t_0;\rho_0)]:B_\e((1,1))\to \cV_{r+1}$ being smooth.
Since $[\rho\mapsto\rho(\tau_0-t_0)]\in\kL(H^{r+1}(\bT),\R) $, we conclude that 
\begin{equation}\label{fun1}
[\lambda\mapsto \rho(\lambda_1 t_0)(\tau_0-t_0+\lambda_2t_0)]\in{\rm C}^\infty(B_\e((1,1)),\R).
\end{equation}
Let now $\delta>0$ be chosen such the (smooth) mapping $f:B_\delta((t_0,\tau_0))\to B_\e((1,1))$ with 
\begin{equation}\label{fun2}
f(t,\tau):=\Big(\frac{t}{t_0},\frac{\tau-\tau_0+t_0}{t_0}\Big).
\end{equation}
is well-defined. Composing  the functions defined in~\eqref{fun1} and \eqref{fun2}, we deduce that the mapping~$[(t,\tau)\mapsto \rho(t)(\tau)]$ belongs to $ {\rm C}^\infty(B_\delta((t_0,\tau_0)))$, which completes the proof.
\end{proof}

\section{Proof of Theorem~\ref{MT2}}\label{SEC:4}

In this section we establish the proof of Theorem~\ref{MT2}. To this end, we fix $\bar r\in(3/2,2)$ and chose some $r\in (3/2,\bar r)$.

 First note that the stationary solutions to \eqref{PB} are smooth functions~$\rho\in \cV_{\bar r}$ having the property that~$\G_\rho$ is a circle.
Indeed, by Theorem~\ref{MT1} every stationary solution to \eqref{PB} belongs to~${\rm C}^\infty(\bT)$.
Moreover, recalling~\eqref{partint} and using Lemma~\ref{L:1} and Lemma~\ref{L:2}, we deduce for $\rho\in{\rm C}^\infty(\bT)$ that 
\begin{equation}\label{PhiBeta}
\Phi(\rho)[\rho]=-\frac{1}{\rho}\bB(\rho)^*[\beta'],
\end{equation}
where $\beta'$  is the unique solution to
\begin{equation}\label{identww}
(1-\bD(\rho)^*)[\beta']=(\kappa(\rho)[\rho]+f(\rho))'.
\end{equation}
But, if $\rho$ is a stationary solution,   then $\Phi(\rho)[\rho]=0$ so that \eqref{est2222}  and \eqref{PhiBeta} ensure that $\beta'=0$. 
Recalling~\eqref{darcru}, we deduce  from  \eqref{identww} that the curvature of $\G_\rho$ is constant, meaning that $\G_\rho$ is indeed a circle.\\

Next, recall from \eqref{lol1}-\eqref{lol2} that the evolution preserves the area and the center of mass of the fluid domain. Focusing on the unit circle (corresponding to $\rho=1$), we thus investigate its stability properties for (small) perturbations with area $\pi$ and center of mass at the origin as stated in \eqref{inidatacoo}. Observe that
\begin{equation*}
\Phi(1)=H\circ\Big(\frac{{\rm d}^3}{{\rm d}\tau^3}+\frac{{\rm d}}{{\rm d}\tau}\Big)
\end{equation*}
is a Fourier multiplier with symbol $(|k|(1-|k|^2))_{k\in\Z}$  and its spectrum~$\sigma(\Phi(1))$ consists of isolated eigenvalues, being given by
\begin{equation*}
\sigma(\Phi(1))=\{|k|(1-|k|^2):k\in\Z\}.
\end{equation*}
Thus, all eigenvalues $\lambda$ of $\Phi(1)$ satisfy $\lambda\leq -6$, except for the eigenvalue $0$ which has multiplicity~$3$, and  consequently, the principle of linearized stability  \cite[Theorem 1.3]{MW20} cannot be applied in this context.
Instead, we will use Theorem~\ref{MT3}  to prove Theorem~\ref{MT2}, which requires some preparation.
Since we restrict to solutions satisfying~\eqref{inidatacoo},
 we observe that  if $\rho\in\cV_r$ is such that  $\Omega_\rho$ has area equal to $\pi$ and center of mass at $(0,0)$, then
\begin{equation}\label{eqcons}
0=\int_{-\pi}^\pi(\rho^2-1)\, {\rm d}s=\int_{-\pi}^\pi\rho^3\cos\, {\rm d}s=\int_{-\pi}^\pi\rho^3\sin\, {\rm d}s.
\end{equation}
Following \eqref{eqcons}, we deduce from \eqref{lol1}-\eqref{lol2} and~\eqref{PhiBeta} for all positive functions $ \rho\in {\rm C}^\infty(\bT)$ that
\begin{equation}\label{denti}
0=\int_{-\pi}^\pi \rho\Phi(\rho)[\rho]\, {\rm d}s =\int_{-\pi}^\pi \rho^2\Phi(\rho)[\rho]\cos\, {\rm d}s =\int_{-\pi}^\pi \rho^2\Phi(\rho)[\rho]\sin\, {\rm d}s. 
\end{equation} 
It is not possible to incorporate all the properties \eqref{eqcons} and \eqref{denti} into the  domain of definition and the target space of the operator $\Phi$, 
respectively, in order to eliminate the eigenvalue~${\lambda=0}$ from  the spectrum~$\sigma(\Phi(1))$.
This is due to the fact that  these are nonlinear properties in~$\rho$, two constraints in~\eqref{eqcons}   being cubic and one quadratic.
 Nevertheless, the cubic constraints in \eqref{eqcons} (which fix the center of mass at $(0,0)$) and the corresponding  integral identities  in \eqref{denti} { may be included to reformulate~\eqref{PB}  subject to~\eqref{eqcons}
 as a quasilinear problem to which the generalized principle of linearized stability in interpolation spaces  provided in Theorem~\ref{MT3} applies.  
More precisely,   introducing the new variable
\[
v:=\rho^3
\]
the evolution problem~\eqref{QPP} is equivalent to 
\begin{equation}\label{QPPu}
\cfrac{{\rm d}v}{{\rm d}t}(t)=\Psi(v(t))[v(t)],\quad t>0,\qquad v(0)=v_0,
\end{equation}
where the operator $\Psi$ is defined as follows.
Set 
\[
\wh{H}^s(\bT):=\{v\in H^{s}(\bT)\,:\, \langle v,\cos\rangle=\langle v,\sin\rangle=0\},\qquad s\in\R,
\]
and $\wh{\cV}_r:=\cV_r\cap \wh{H}^r(\bT)$. 
Observe that the mapping $[\rho\mapsto \rho^3]:\cV_s\to\cV_s$ is a smooth diffeomorphism for~$s>3/2$.  Hence, if $\rho\in\cV_{r+1}$  is such that the center of mass of $\Omega_\rho $ 
is the point $(0,0)$, then~$v\in\wh\cV_{r+1}$ by \eqref{eqcons} and, recalling \eqref{reg:Phi} and \eqref{denti}, it also holds that $\rho^2\Phi(\rho)[\rho] \in \wh{H}^{r-2}(\bT)$.
Thus, the maximal solution  $\rho=\rho(\cdot;\rho_0):[0,T^+)\to\cV_{\bar r}$  to problem~\eqref{QPP} determined by~${\rho_0\in\cV_{\bar r}}$   satisfying \eqref{inidatacoo} has the property that $v:=\rho^3:[0,T^+)\to\wh \cV_{\bar r}$ satisfies $v(0)=\rho_0^3=:v_0$  and
\[
\cfrac{{\rm d}v}{{\rm d}t}(t)=3v^{2/3}(t)\Phi\big(v^{1/3}(t)\big)\big[v^{1/3}(t)\big]\in \wh{H}^{r-2}(\bT)\quad  \text{for $t>0$}.
\]
To introduce the operator $\Psi:\wh{\cV}_r\to\kL(\wh H^{r+1}(\bT), \wh H^{r-2}(\bT))$  satisfying
\begin{equation}\label{traequa}
\Psi(v)[v]=3  v^{2/3}\Phi\big(v^{1/3}\big)\big[v^{1/3} \big]\quad \text{for $v\in\wh \cV_{r+1}$},
\end{equation}  
we first note from \eqref{krhof}, that 
\begin{equation}\label{kapatrans}
\kappa\big(v^{1/3}\big)\big[v^{1/3}\big]=\wh \kappa(v)[v] \qquad\text{and}\qquad (f\big(v^{1/3}\big))'=\lambda\big(v^{1/3}\big)\big[v^{1/3}\big]=\wh \lambda(v)[v],
\end{equation}
where 
\begin{equation}\label{regwhkf}
\wh\kappa\in{\rm C}^\infty(\wh \cV_r, \kL(\wh H^{r+1}(\bT),   H^{r-1}(\bT)))\qquad\text{and}\qquad \wh\lambda\in{\rm C}^\infty(\wh \cV_r, \kL(\wh H^{r+1}(\bT),     H^{r-1}_0(\bT)))
\end{equation}
  are given by 
\[
\wh \kappa(v)[w]:=- \frac{ 9w''-6v^{-1}v'w'}{v^{4/3}(9+ v^{-2}v'^2)^{3/2}} 
\]
and 
\begin{align*}
\wh \lambda(v)[w]&:=  \frac{ 27v'w''-54v^{-1}v'^2w'-81vw'-6v^{-2}v'^3w''+4v^{-3}v'^4w'}{v^{7/3}(9+ v^{-2}v'^2)^{5/2}} \\
&\qquad-\Big\langle\frac{ 27v'w''-54v^{-1}v'^2w'-81vw'-6v^{-2}v'^3w''+4v^{-3}v'^4w'}{v^{7/3}(9+ v^{-2}v'^2)^{5/2}}\Big\rangle.
\end{align*}
Recalling the formulas~\eqref{alpha1}-\eqref{alpha2}, we  obtain with \eqref{kapatrans} that 
\begin{align*}
\alpha_1\big(v^{1/3}\big)\big[v^{1/3}\big]&=(1+\bD(v^{1/3}))^{-1}\big[\kappa\big(v^{1/3}\big)\big[v^{1/3}\big]\big]=(1+\bD(v^{1/3}))^{-1}[\wh \kappa(v)[v]],\\
\alpha_2\big(v^{1/3}\big)\big[v^{1/3}\big]&=(1-\bD(v^{1/3})^*)^{-1}\big[\lambda\big(v^{1/3}\big)\big[v^{1/3}\big]\big]=(1-\bD(v^{1/3})^*)^{-1}[\wh \lambda(v)[v]], 
\end{align*}
and therefore we define
\begin{align*}
\wh\alpha_1(v)[w]&:=(1+\bD(v^{1/3}))^{-1}[\wh \kappa(v)[w]],\\
\wh \alpha_2(v)[w]&:=(1-\bD(v^{1/3})^*)^{-1}[\wh \lambda(v)[w]].
\end{align*}
 Note from  \eqref{reg:Ds}, \eqref{regwhkf}, and Proposition~\ref{Prop:3} that  
\begin{equation}\label{reg:whalphas}
\wh\alpha_1\in{\rm C}^\infty(\wh \cV_r, \kL(\wh H^{ r+1}(\bT),H^{r-1}(\bT))) \quad\text{and}\quad \wh \alpha_2\in{\rm C}^\infty(\wh \cV_r, \kL(\wh H^{ r+1}(\bT),  H^{ r-1}_0(\bT)))
\end{equation}
with 
\begin{equation}\label{idenalpha}
\alpha_i\big(v^{1/3}\big)\big[v^{1/3}\big]=\wh \alpha_i(v)[v],\qquad i=1,\, 2, \quad v\in \wh \cV_{r}.
\end{equation}
Setting
\begin{align}
\Psi(v)[w]:=3v^{1/3}\big\{\big( \bB(v^{1/3})[\wh \alpha_1(v)[w]] \big)'- \bB(v^{1/3})^*[\wh \alpha_2(v)[w]]\big\},\label{Psi2}
\end{align}
 we infer from \eqref{Phi1}-\eqref{Phi2} and \eqref{idenalpha}-\eqref{Psi2} that the identity \eqref{traequa} is satisfied and moreover 
\begin{equation}\label{regPsi}
\Psi\in{\rm C}^\infty(\wh{\cV}_r,\kL(\wh H^{r+1}(\bT), \wh H^{r-2}(\bT))), 
\end{equation}
 by \eqref{reg:Bs} and~ \eqref{reg:whalphas}. Thus, the problems~\eqref{QPP} and \eqref{QPPu} are indeed  equivalent.
Let us further note that
\begin{subequations}\label{fosymbs}
\begin{equation}
\Psi(1)=H\circ\Big(\frac{{\rm d}^3}{{\rm d}\tau^3}+\frac{{\rm d}}{{\rm d}\tau}\Big)\in\kL(\wh H^{r+1}(\bT), \wh H^{r-2}(\bT))
\end{equation}
is the Fourier multiplier with spectrum
\begin{equation}
\sigma(\Psi(1))=\{|k|(1-|k|^2):k\in\Z\setminus\{\pm 1\}\}.
\end{equation}
\end{subequations}
Furthermore, in view of \eqref{traequa}, the stationary solutions to \eqref{QPPu} are smooth functions~${v\in\wh \cV_{\bar r+1}}$ having the property that  $\Omega_{v^{1/3}}$ is a circle.
Since $\Omega_{v^{1/3}}$ has center of mass located at $(0,0)$, it follows that  $v$ is a positive constant.

In the following we identify 
 \[
\wh H^s(\bT)=  \wh  H_0^s(\bT)\times \R,\qquad s\in\R,
 \]
where
\[
\wh H_0^s(\bT):=\{v\in H^{s}(\bT)\,:\, \langle v,1\rangle=\langle v,\cos\rangle=\langle v,\sin\rangle=0\},\qquad s\in\R.
\]
Hence, we may represent any $v\in \wh H^s(\bT)$ as a pair $v=(v_1,v_2)$, where $v_2:=(2\pi)^{-1}\langle v,1\rangle\in\R$ and~${v_1:=v-v_2\in \wh H_0^s(\bT)}$.
Consequently, if  $\rho=\rho(\cdot;\rho_0):[0,T^+)\to\cV_{\bar r}$ is the maximal solution  to problem~\eqref{QPP} determined by~${\rho_0\in\cV_{\bar r}}$  satisfying~\eqref{inidatacoo} and $v=(v_1,v_2)=\rho^3$, 
 then the relation 
 \[
 \Psi(u)[u+c]= \Psi(u)[u],\qquad c\in\R,\quad u\in\wh{\cV}_{r+1},
 \] 
 implies that 
\[
\cfrac{{\rm d}(v_2-1)}{{\rm d}t}(t)=(2\pi)^{-1}\langle\Psi(v)[v]\rangle =(2\pi)^{-1}\langle\Psi(v)[(v_1,0)]\rangle
\]
and 
\[
\cfrac{{\rm d}v_1}{{\rm d}t}(t)=\Psi(v)[v]-(2\pi)^{-1}\langle\Psi(v)[v]\rangle =\Psi(v)[(v_1,0)]-(2\pi)^{-1}\langle\Psi(v)[(v_1,0)]\rangle
\]
for $t\in(0,T^+)$.
Hence, setting $u:=v-(0,1)$, problem \eqref{QPPu}  is equivalent to  the quasilinear evolution problem
\begin{equation}\label{QPPuu}
\cfrac{{\rm d}u}{{\rm d}t}(t)=A(u(t))[u(t)],\quad t>0,\qquad u(0)=u_0,
\end{equation}
where 
\[
A(u)=\begin{bmatrix}
A_1(u)&0\\
A_2(u)&0
\end{bmatrix}\in \kL( \wh H_0^{r+1}(\bT)\times\R,  \wh H_0^{r-2}(\bT)\times\R) 
\]
has entries
\begin{equation}\label{defA}
\begin{aligned}
&A_1(u) :=\Psi(u+(0,1)) - (2\pi)^{-1}\langle\Psi(u+(0,1))[\cdot]\rangle  \in\kL(\wh H_0^{r+1}(\bT), \wh  H_0^{r-2}(\bT)),\\
&  A_2(u) :=  (2\pi)^{-1}\langle\Psi(u+(0,1))[\cdot]\rangle \in\kL(\wh H_0^{r+1}(\bT), \R).
\end{aligned}
\end{equation}
In view of~\eqref{fosymbs}, the operator $A_1(0)\in \kL(\wh H_0^{ r+1}(\bT),\wh H_0^{ r-2}(\bT)))$ has spectrum
\[
\sigma(A_1(0))=\{|k|(1-|k|^2):k\in\Z\setminus\{0,\pm 1\}\}\subset\{z\in\mathbb{C}\,:\, \re z\leq -6\},
\]
while   $A_2(0)=0$.
Moreover,  since  $\Psi(1)\in\kH(\wh H^{r+1}(\bT),\wh H^{ r-2}(\bT)))$ by~\eqref{fosymbs} and standard Fourier analysis, we infer from  \eqref{regPsi} that  there are  an open neighborhood $O$ of 
$0$ in  $\wh H_0^{ r}(\bT)$ and an open interval~${\cI\subset\R}$ containing $0$ such that  
\begin{equation*}
A\in {\rm C}^\infty(O\times\cI, \kH(\wh H_0^{r+1}(\bT)\times\R,\wh H_0^{ r-2}(\bT)\times\R)).
\end{equation*}
Moreover the stationary solutions to \eqref{QPPuu} in $O\times\cI$ are the constant functions~$v=(0,x)$ with~$x\in\cI$.
Consequently, we have verified all the assumptions of Theorem~\ref{MT3} in the context of \eqref{QPPuu}.
 Applying this result, the claim of Theorem~\ref{MT2} follows  by recalling  \eqref{lol1}-\eqref{lol2}.

\appendix
\section{A generalized principle of linearized stability for quasilinear parabolic problems in interpolation spaces}\label{SEC:A-1}
We consider the quasilinear evolution problem
\begin{equation}\label{StQPP}
\cfrac{{\rm d}u}{{\rm d}t}(t)=A(u(t))[u(t)],\quad t>0,\qquad u(0)=u_0,
\end{equation}
in the following analytic setting.
Let   $E_0$ and $E_1$ be Banach spaces  over $\mathbb{K}\in \{\R,\mathbb{C}\}$ with   continuous  and dense embedding 
$E_1\hookrightarrow E_0$.
For each $\theta\in (0,1)$, we denote by $(\cdot,\cdot)_\theta$   an arbitrary admissible interpolation functor of 
 exponent~$\theta$  (see \cite[I.~Section~2.11]{LQPP}) and set~${E_\theta:= (E_0,E_1)_\theta}$  for the corresponding  interpolation space with norm~$\|\cdot\|_\theta$.
We fix  exponents
\begin{subequations}\label{ASS}
\begin{equation}\label{AS1}
 0<\beta<\alpha<1
\end{equation}
 and assume, for some open subset $O_\beta$ of $E_\beta$ containing $0$ and an open Interval $\cI\subset\R$ with~$0\in\cI$, that
 \begin{equation}\label{AS2}
A\in {\rm C}^{1-}\big(O_\beta\times\cI,\mathcal{H}(E_1\times \R,E_0\times\R) \big)\,.
\end{equation}
Moreover, we assume that $A(u)$ has a matrix structure of the form 
\begin{equation}\label{AS3}
A(u)=
\begin{pmatrix}
A_1(u)& 0\\
A_2(u)&0
\end{pmatrix},
\end{equation}
 and, for some  $\omega_0>0$,   
 \begin{equation}\label{AS4}
\sigma (A_1(0))\subset\{\lambda\in\mathbb{C}\,:\, \re\lambda\leq -\omega_0\} \qquad\text{and}\qquad A_2(0)=0.
\end{equation}
Finally, we assume that the stationary solutions to \eqref{StQPP}  are the constants; that is
 \begin{equation}\label{AS5}
A(u)[u]=0 \qquad\iff\qquad \text{$u=(0,x)$ with $x\in\cI$.}
\end{equation}
\end{subequations}
The local well-posedness of~\eqref{StQPP}  in $O_\alpha\times \cI$, where $O_\alpha:=O_\beta\cap E_\alpha$, is established in 
\cite{Am93} (see also \cite[Theorem 1.1]{MW20}) in a  more general setting.
Specifically, problem \eqref{StQPP} has for each $u_0\in O_\alpha\times\cI$ a unique maximal solution $u=u(\cdot;u_0)$ satisfying
\begin{subequations}\label{regu}
\begin{equation} \label{regu1}
 u\in {\rm C}^1((0,t^+(u_0)),E_0\times\R)\cap {\rm C}((0,t^+(u_0)),E_1\times\R)\cap  {\rm C} ([0,t^+(u_0)),O_\alpha\times\cI),
\end{equation}
where $ t^+(u_0)\in(0,\infty]$,   and
\begin{equation}\label{regu2}
  u(\cdot;u_0)\in {\rm C}^{\alpha-\eta}([0,t^+(u_0)),E_\eta\times\R)\,,\qquad \eta\in [0,\alpha].
\end{equation} 
\end{subequations}
The next results states that  if $u_0$ is close to $0\in\cO_\alpha\times\cI$, the solution $u(\cdot;u_0)$ exists for all times 
and converges at an exponential rate towards a  stationary solution $(0,x_*)$ to \eqref{StQPP}, where 
the choice of~$x_*$  generally depends   on the initial condition $u_0$.

\begin{thm}\label{MT3} Let $\omega\in(0,\omega_0)$ be fixed. Then, there exists constants $\e>0$ and $M\geq 1$ such that for each  $u_0\in \cO_\alpha\times \cI$ with $\|u_0\|_{E_\alpha\times\R}\leq \e$ the maximal solution $u=u(\cdot;u_0)$ is globally defined and there exists $x_*\in\cI$ such that
\begin{equation}
\|u(t)-(0,x_*)\|_{E_\alpha\times\R}\leq M e^{-\omega t}\|u_0\|_{E_\alpha\times\R},\qquad t\geq 0.
\end{equation}
\end{thm}
\begin{proof} We  divide the proof into three steps.\medskip

\noindent{\bf Preliminaries.}  Denote by $\sfe_{\alpha,\beta}$ the norm of the embedding $E_\alpha \hookrightarrow E_\beta$ and assume  without loss of generality that~${\sfe_{\alpha,\beta} \geq 1}$. 
 Choose  $\e_0\in(0,1]$ such that
$$\ov\bB_{E_\alpha\times\R}(0,2\e_0/\sfe_{\alpha,\beta})\subset\ov\bB_{E_\beta\times\R}(0,2\e_0)\subset \cO_\beta\times\cI.$$ 
 Further fix $\rho \in (0, \alpha - \beta)$ and $\omega \in (0, \omega_0)$, and set $4\delta := \omega_0 - \omega > 0$.
Since~$A_1(0)\in\mathcal{H}(E_1,E_0)$ by~\cite[I.Corollary 1.6.3]{LQPP},    our assumptions~\eqref{AS2} and~\eqref{AS4} together with~\cite[I.Proposition~1.4.2]{LQPP}  ensure,
after making $\e_0>0$ smaller if necessary, that  there exist constants~${\kappa\geq 1}$ and $L>0$ such that
\begin{equation}\label{qe2}
\omega_0-\delta+A_1(v)\in\mathcal{H}(E_1,E_0;\kappa,\delta),\qquad  v\in \ov\bB_{E_\beta\times\R}(0,2\e_0),
\end{equation}
and 
\begin{equation}\label{qe3}
\|A(v )-A(\bar v )\|_{\kL(E_1\times\R,E_0\times\R)}\leq L\|v-\bar v\|_{E_\beta\times\R},\qquad  v,\,\bar v \in \ov\bB_{E_\beta\times\R}(0,2\e_0).
\end{equation} 
Fix $T\in(0,\infty)$ and define
\[
\mathcal{M}(T):=\left\{ v\in {\rm C}\big([0, T],\ov\bB_{ E_\alpha\times\R}(0,2\e_0/{\sf e}_{\alpha,\beta})\big)\,:\, 
\|v(t)-v(s)\|_{E_\beta\times\R}\le \frac{N}{L }\vert t-s\vert^{\rho}\,,\, 0\le s\leq t\le T\right\},
\]
where $N>0$ is chosen as follows. 
Given $v\in\mathcal{M}(T)$, the estimates~\eqref{qe2}-\eqref{qe3} yield
\begin{subequations}\label{b}
\begin{equation}\label{b1}
\omega_0-\delta + A_1(v(t))\in \mathcal{H}(E_1,E_0,\kappa,\delta)\,,\qquad t\in [0,T],
\end{equation}
and
\begin{equation}\label{b2}
 A_1(v)\in {\rm C}^\rho\big([0,T],\mathcal{L}(E_1,E_0)\big)
\quad\text{with}\quad \sup_{0\le s<t\le T} \frac{\| A_1(v(t))-A_1(v(s))\|_{\mathcal{L}(E_1,E_0)}}{(t-s)^\rho}\le N.
\end{equation}
\end{subequations}
In view of \eqref{b} we may apply  results from \cite[II.Section~5]{LQPP} to the family~${\mathcal{A}:=\{A_1(v)\,:\,v\in\mathcal{M}(T)\}}$.
Letting~$c_0(\rho)>0$ be the constant from \cite[II.Theorem 5.1.1]{LQPP} (which is independent of $N$), we choose the constant $N>0$ such that  $c_0(\rho)N^{1/\rho}=\delta$.
 Then, by \cite[II.Theorem 5.1.1, II.Lemma 5.1.3]{LQPP} there exists for each  $v\in\mathcal{M}(T)$ a  unique  evolution operator $U_{A_1(v)}$ for $A_1(v)$ satisfying 
\begin{subequations}\label{evol}
\begin{equation}\label{ev0} 
\begin{split}
\|U_{ A_1(v)}(t,s)\|_{\mathcal{L}(E_\theta)}+ (t-s)^{\theta-\vartheta_0}\|U_{A_1(v)}(t,s)\|_{\mathcal{L}(E_\vartheta,E_\theta)}  \leq  \frac{M_1}{2}e^{-\nu (t-s)}\,,\qquad 0\le s\le t\le T,
\end{split}
\end{equation}
 where the  constant $M_1>0$ is independent of $T>0$  (but depends on $N,\,\kappa,\,\delta$, and $\rho$) and
\[
-\nu:= c_0(\rho)N^{1/\rho}-\omega_0+\delta+\delta=-\omega-\delta <-\omega<0.
\]
  The estimate \eqref{ev1} holds for  $0\le \vartheta_0\le \vartheta\le \theta\le 1$ with $\vartheta_0<\vartheta$ if $0<\vartheta< \theta< 1$.
  Moreover, we infer from \cite[II.Theorem 5.3.1]{LQPP} (with~${f=0}$  therein)  that there  exists a further constant~$M_2>0$, which is also independent of $T>0$,  such that
\begin{align}\label{ev1}
\|U_{A_1(v)}(t,0)-U_{A_1(v)}(s,0)\|_{\mathcal{L}(E_\alpha,E_\beta)}\le M_2 (t-s)^{\alpha-\beta},\qquad 0\le s\le t\le T.
\end{align}
  \end{subequations}
  
 Let~${\e\in (0,\e_0)}$ be chosen such that
\begin{subequations}\label{g5}
\begin{align}
&\frac{\e M_1}{2} \leq\frac{\e_0}{\sfe_{\alpha,\beta}},\qquad \e M_2\Big(\frac{4L}{N}\Big)^{(\alpha-\beta-\rho)/\rho}\leq    \frac{N}{2L},
\qquad  \e\Big(1+LM_1\int_0^\infty  r^{\alpha-1}e^{-\nu r}\,{\rm d}r\Big)\leq \frac{\e_0}{\sfe_{\alpha,\beta}},\label{MTp1}\\
&  \e  \alpha^{-1} LM_1 |\tau-s|^{\rho}\Big(\frac{4L }{N}\Big)^{(\alpha-\rho)/\rho}\leq  \frac{N}{2L}.\label{MTp2}
\end{align}
\end{subequations}

\noindent{\bf Global existence and uniform estimates.}
Let $ u_0=(u_{0,1}, u_{0,2})\in\ov\bB_{E_\alpha\times \R}(0,\e\big)$ be fixed 
 and let ~$u=u(\cdot;u_0)$ be the corresponding maximal solution to \eqref{StQPP}  satisfying \eqref{regu}.
Moreover, as shown in (the proof of) \cite[Proposition 2.1]{MW20}, after making $\e$ smaller if necessary, 
there are constants~$k_0\geq 1$  and~${t_0\in(0,1)}$ such that for all~$u_0\in\ov\bB_{E_\alpha\times \R}(0,\e\big)$  it holds that $t^+(u_0)\geq t_0$ and
\begin{equation}\label{qe4} 
\|u(t)\|_{E_\alpha\times\R}\leq k_0\|u_0\|_{E_\alpha\times\R}\leq k_0\e\leq 2\e_0/{\sfe_{\alpha,\beta}},\qquad 0\leq t\leq t_0,
\end{equation}  
as well as, recalling that $\rho\in(0,\alpha-\beta)$,
\begin{equation}\label{qe5}
\|u(t)-u(s)\|_{E_\beta\times\R}\leq k_0|t-s|^{\alpha-\beta}\leq \frac{N}{L}(t-s)^\rho,\qquad 0\leq s\leq t\leq t_0.
\end{equation}

We now define
\[
t_1:=\sup\big\{t<t^+(u_0)\,:\, u\vert_{[0,t]}\in\mathcal{M}(t) \big\}
\]
and infer from  \eqref{qe4}-\eqref{qe5} that $t_1\geq t_0$ for all $u_0\in\ov\bB_{E_\alpha\times \R}(0,\e\big)$.
Let $t\in(0,t_1)$ be arbitrary.
Noticing that $u_1$ solves the evolution problem
\begin{equation*}
\cfrac{{\rm d}u_1}{{\rm d}t}(t)=A_1(u(t))[u_1(t)],\quad t>0,\qquad u_1(0)=u_{0,1},
\end{equation*}
we deduce  from \eqref{regu} that 
\begin{equation}\label{foru1}
u_1(\tau)=U_{A_1(u)}(\tau,0)u_{0,1}, \qquad 0\leq \tau\leq t.
\end{equation}
Using \eqref{ev0} and recalling \eqref{MTp1}, we therefore have
\begin{equation}\label{proppp1}
\|u_1(\tau)\|_\alpha\leq \frac{M_1}{2}e^{-\nu \tau}\|u_{0,1}\|_\alpha\leq \frac{\e M_1}{2} \leq\frac{\e_0}{\sfe_{\alpha,\beta}},\qquad \tau\in[0,t],
\end{equation} 
and
\begin{equation}\label{proppp1'}
\|u_1(\tau)\|_1\leq \frac{M_1}{2}\tau^{\alpha-1}e^{-\nu \tau}\|u_{0,1}\|_\alpha \leq  \frac{M_1\e}{2}\tau^{\alpha-1}e^{-\nu \tau},\qquad \tau\in[0,t],
\end{equation} 
The  estimate~\eqref{proppp1} shows, for  $0\leq s\leq \tau \leq t$  with $|\tau-s|^\rho\geq 4\e_0 L/N$, that
 \begin{equation*}
\|u_1(\tau)-u_1(s)\|_\beta\leq \sfe_{\alpha,\beta} \|u_1(\tau)-u_1(s)\|_\alpha\leq 2\e_0=\frac{2\e_0}{|\tau-s|^\rho}|\tau-s|^\rho  \leq  \frac{N}{2L} |\tau-s|^\rho,
\end{equation*} 
while, for $|\tau-s|^\rho\leq 4\e_0 L/N$,  the inequalities \eqref{ev1} and  \eqref{MTp1} entail that
 \begin{equation*}
\|u_1(\tau)-u_1(s)\|_{\beta}\leq M_2|\tau-s|^{\alpha-\beta}\|u_{0,1}\|_\alpha\leq  \e M_2\Big(\frac{4  L }{N}\Big)^{(\alpha-\beta-\rho)/\rho}|\tau-s|^{\rho}\leq    \frac{N}{2L} |\tau-s|^\rho.
\end{equation*} 
Consequently, we have
 \begin{equation}\label{proppp2}
\|u_1(\tau)-u_1(s)\|_{\beta} \leq   \frac{N}{2L} |\tau-s|^\rho \qquad\text{for all $0\leq s\leq\tau\leq t$.}
\end{equation}

Concerning the second component $u_2$ of $u$, we infer from \eqref{StQPP}, \eqref{AS3}, \eqref{AS4},  \eqref{qe3}, \eqref{MTp1},  and~\eqref{proppp1'},  since
\[
\frac{{\rm d} u_2}{{\rm d}t}(t)=A_2(u(t))[u_1(t)]\,,\quad t\in(0,t^+(u_0)),\qquad u_2(0)=u_{0,2},
\]
that for~${\tau\in[0,t]}$ we have 
\begin{equation}\label{proppp12}
\begin{aligned}
|u_2(\tau)|&=|u_{0,2}|+\int_0^\tau\big| A_2(u(r))[u_1(r)]\big|\,{\rm d}r\leq \e+ L\int_0^\tau \|u(r)\|_{E_\beta\times\R}\|u_1(r)\|_1\,{\rm d}r\\
&\leq \e+ \e LM_1\int_0^\tau  r^{\alpha-1}e^{-\nu r}\,{\rm d}r \leq \e\Big(1+LM_1\int_0^\infty  r^{\alpha-1}e^{-\nu r}\,{\rm d}r\Big)\leq \frac{\e_0}{\sfe_{\alpha,\beta}}.
\end{aligned}
\end{equation}
 In particular, for   $0\leq s\leq \tau \leq t$  with $|\tau-s|^\rho\geq 4\e_0 L/N$, it holds that
 \begin{equation*}
|u_2(\tau)-u_2(s)|\leq 2\e_0=\frac{2\e_0}{|\tau-s|^\rho}|\tau-s|^\rho  \leq  \frac{N}{2L} |\tau-s|^\rho.
\end{equation*} 
Using \eqref{MTp2} and (the arguments used to derive) estimate \eqref{proppp12} we deduce for all $0\leq s\leq \tau\leq t$ with~${|\tau-s|^\rho\leq 4\e_0 L/N}$  that
 \begin{equation*}
\begin{aligned}
|u_2(\tau)-u_2(s)|&\leq \int_s^\tau\big| A_2(u(r))[u_1(r)]\big|\,{\rm d}r \leq  L\int_s^\tau  \|u(r)\|_{E_\beta\times\R}\|u_1(r)\|_1\,{\rm d}r\\
&\leq \e  LM_1 \int_s^\tau  r^{\alpha-1}e^{-\nu r}\,{\rm d}r\leq \e  LM_1 \int_s^\tau  r^{\alpha-1}\,{\rm d}r =\e \alpha^{-1} LM_1 (\tau^\alpha-s^\alpha)  \\
&\leq \e  \alpha^{-1} LM_1 |\tau-s|^\alpha \leq \e  \alpha^{-1} LM_1 \Big(\frac{4L }{N}\Big)^{(\alpha-\rho)/\rho} |\tau-s|^\rho\\
&\leq \frac{N}{2L} |\tau-s|^\rho ,
\end{aligned}
\end{equation*}
and therefore
 \begin{equation}\label{proppp22}
|u_2(\tau)-u_2(s)| \leq   \frac{N}{2L} |\tau-s|^\rho, \quad 0\leq s\leq\tau\leq t.
\end{equation} 
Gathering~\eqref{proppp1} and \eqref{proppp2}-\eqref{proppp22}, we conclude that $t_1=t^+(u_0)$ and  for $0\leq s\leq t<t^+(u_0)$ 
\begin{equation}\label{globest}
\|u(t)\|_{E_\alpha\times\R}\leq  \frac{2\e_0}{\sfe_{\alpha,\beta}} \qquad\text{and}\qquad \|u(t)-u(s)\|_{E_\beta\times\R}\leq  \frac{N}{L} |t-s|^\rho.
\end{equation}
These estimates directly imply that $t^+(u_0) = \infty$. Indeed, assume for contradiction, that~${t^+(u_0) < \infty}$. Then the solution $u$ can be extended as a Hölder continuous function $u : [0, t^+(u_0)] \to E_\beta \times \mathbb{R}$. In particular,~\eqref{b} remain valid with $T = t^+(u_0)$, and consequently, the evolution operator~$U_{A_1(u)}$ satisfies~\eqref{evol} for $T = t^+(u_0)$.
Formula~\eqref{foru1} now implies  that the mapping $u : [0, t^+(u_0)) \to E_\alpha \times \mathbb{R}$ is uniformly continuous. Therefore,  recalling~\eqref{globest}, 
the orbit $u([0, t^+(u_0)))$ is relatively compact in $O_\alpha \times \mathcal{I}$. Applying \cite[Theorem 1.1~(iii)]{MW20}, it follows that $t^+(u_0) = \infty$, which contradicts our assumption. 
Hence, $t^+(u_0) = \infty$ and  the estimates \eqref{globest} are valid for 
all~$0\leq s\leq t<\infty$.\medskip

\noindent{\bf Exponential stability.} Since \eqref{globest} ensure that  \eqref{b} and~\eqref{evol} hold also for $T =\infty$,
 we deduce for $0\leq t\leq T$, using~\eqref{MTp1} and~\eqref{foru1}, that
\begin{equation}\label{expu1}
e^{\omega t}\|u_1(t)\|_\alpha\leq e^{\omega t}\|U_{A_1(u)}(t,0)\|_{\kL(E_\alpha)}\|u_{1,0}\|_\alpha\leq M_1e^{(\omega-\nu)t}\|u_{1,0}\|_\alpha\leq M_1\|u_{1,0}\|_\alpha
\end{equation}
and
\begin{align*}
\|u_1(t)\|_1\leq \|U_{A_1(u)}(t,0)\|_{\kL(E_\alpha,E_1)}\|u_{1,0}\|_\alpha\leq\frac{ M_1\e}{2} t^{\alpha-1}e^{-\nu t}\leq   t^{\alpha-1}e^{-\nu t}.
\end{align*}
The latter estimate  together with \eqref{AS4}, \eqref{qe3}, \eqref{proppp1'} and \eqref{globest} yields
\begin{align*}
 \int_0^\infty |A_2(u(r))[u_1(r)]|\,{\rm d}r &\leq  L\int_0^\infty \|u(r)\|_{E_\beta\times\R}\|u_1(r)\|_1\,{\rm d}r\leq  \e LM_1\int_0^\infty  r^{\alpha-1}e^{-\nu r} \,{\rm d}r.
\end{align*}
Hence,  for $t\to\infty$,  
\[
u_2(t)\to x_*:= u_{0,2}+\int_0^\infty A_2(u(r))[u_1(r)]\,{\rm d}r,
\]
with $x_*\in\cI$  by~\eqref{MTp1}.
Moreover, arguing as in the proof of \eqref{proppp22}, we  get
 \begin{equation}\label{expu2}
\begin{aligned}
|u_2(t)-  x_*|&\leq \int_t^\infty \big|A_2(u(r))[u_1(r)]\big|\,{\rm d}r\leq L\int_t^\infty \|u(r)\|_{E_\beta\times\R}\|u_1(r)\|_1\,{\rm d}r\\
&\leq  LM_1 |u_{0,1}|  \int_t^\infty  r^{\alpha-1}e^{-\nu r}\,{\rm d}r\leq LM_1 |u_{0,1}|  \int_t^\infty  r^{\alpha-1}e^{-\delta r}e^{-\omega r}\,{\rm d}r\\
  &\leq LM_1 |u_{0,1}|e^{-\omega t} \int_0^\infty  r^{\alpha-1}e^{-\delta r}\,{\rm d}r.
\end{aligned}
\end{equation}
Hence, in view of \eqref{expu1} and \eqref{expu2}, there exists a constant $M>0$ such that
\[
\|u(t)-(0,x_*)\|_{E_\alpha\times\R}\leq Me^{-\omega t}\|u_0\|_{E_\alpha\times\R},\qquad t\in[0,\infty),
\]
which completes the proof. 
\end{proof}

\section{Mapping properties for the (singular) integral operators $B_{n,m}^p$}\label{SEC:A}

Let $r\in(3/2,2]$, recall the definition~\eqref{cvr} of~$\cV_r$, fix (an arbitrary) $M>1$,  and set
\begin{equation}\label{defVM}
\cV_{r,M}:=\{ \rho\in \cV_r\,:\,\text{$\rho> M^{-1}$ and $\|\rho\|_{H^r}<  M$}\}.
\end{equation}
Some of our arguments below rely on the observation that  for $x\in(-\pi/2,\pi/2)$ we have
 \begin{equation}\label{inequalities}
 |\tan(x)-x|\leq x^2|\tan(x)| \qquad\text{and}\qquad |x|\leq |\tan(x)|.
 \end{equation}

\subsection*{Mapping properties for $B_{n,m}^{p}$} We recall the definition \eqref{Bnmp} of the integral operators $B_{n,m}^{p}$ with~${m,\,n,\,p\in\N}$  and $0\leq p\leq n+1$.
 We begin by showing that, for $p=0$, the operator  $B_{n,m}^{p}(\varrho)[h,\cdot]$  belongs to~${\kL(L_2(\bT))}$ for each $\varrho\in\cV_{r,M}^m$ and~$h \in W^1_\infty(\bT)^n$.
To this end we introduce a second family of multilinear singular integral operators 
\[
G_{n} [h,\beta](\tau)
:=\displaystyle\frac{1}{\pi}\PV\int_{-\pi}^\pi\Big(\prod\limits_{i=1}^n
\frac{\delta_{[\tau,s]}h_i}{t_{[s]}}\Big)\frac{\beta(\tau-s)}{t_{[s]}}\,{\rm d}s,
\]
where $n\in\N$, $h=(h_1,\ldots, h_n)\in W^1_\infty(\bT)^n$, $\beta\in L_2(\bT)$, and $\tau\in\R$.
In the arguments that follow we will use the algebraic property
\begin{equation}\label{difference}
\begin{aligned}
&B_{n,m}^{p}(\varrho)[h,\beta]-B_{n,m}^{p}(\bar\varrho)[h,\beta]\\
&=\sum_{j=1}^m\Big\{
B_{n+2,m+1}^{p}(\bar\varrho_1,\ldots,\bar\varrho_j,\varrho_j,\ldots,\varrho_m)
[h,\bar \varrho_j+\varrho_j,\bar \varrho_j-\varrho_j,\beta]\\[-2ex]
&\hspace{1.275cm}+(\bar \varrho_j^2-\varrho_j^2)B_{n,m+1}^{p}(\bar\varrho_1,\ldots,\bar\varrho_j,\varrho_j,\ldots,\varrho_m)[h,\beta]\\[1ex]
&\hspace{1.275cm}+(\bar \varrho_j+\varrho_j)B_{n,m+1}^{p}(\bar\varrho_1,\ldots,\bar\varrho_j,\varrho_j,\ldots,\varrho_m)[h,(\bar \varrho_j-\varrho_j)\beta]\\[1ex]
&\hspace{1.275cm}+ (\bar \varrho_j-\varrho_j)B_{n,m+1}^{p}(\bar\varrho_1,\ldots,\bar\varrho_j,\varrho_j,\ldots,\varrho_m)[h,(\bar \varrho_j+\varrho_j)\beta]\\ 
&\hspace{1.275cm}+ B_{n,m+1}^{p}(\bar\varrho_1,\ldots,\bar\varrho_j,\varrho_j,\ldots,\varrho_m)[h,(\bar \varrho_j^2-\varrho_j^2)\beta]
\Big\},
\end{aligned}
\end{equation} 
 which holds for all $0\leq p\leq n+1$,  $\varrho=(\varrho_1,\ldots,\varrho_m),\,\bar \varrho=(\bar\varrho_1,\ldots,\bar\varrho_m)\in\cV_{r,M}^m$, $h \in W^1_\infty(\bT)^n$, and~$\beta\in L_2(\bT)$.
Obviously, the operators $B_{n,m}^{p}(\varrho)[h,\cdot]$ are only singular when $p=0$, in which case we have:
 
\begin{lemma}\label{L:B1} Let $n,\,m\in\N$   be given.
\begin{itemize}
\item[(i)] There is a constant $C=C(M)>0$ such that
for all $h \in W^1_\infty(\bT)^n $  and~${\varrho \in\cV_{r,M}^m}$  it holds that
\begin{equation}\label{eq:B2}
\|B_{n,m}^{0}(\varrho)[h,\cdot]\|_{\kL(L_2(\bT))}\leq C\prod_{i=1}^n\|h_i'\|_{\infty}.
\end{equation} 
Moreover, the mapping $[\varrho\mapsto B_{n,m}^0(\varrho)]:\cV_{r,M}^m\to\kL_{\rm sym}^n(W^1_\infty(\bT),\kL(L_2(\bT)))$ is locally Lipschitz continuous.
\item[(ii)] There exists a constant $C=C(M)>0$ such that
for all $\beta \in H^{r-1}(\bT)$, $h \in H^r(\bT)^n $,   and~$\varrho \in\cV_{r,M}^m$  it holds that
\begin{equation}\label{eq:B2'}
\|B_{n,m}^{0}(\varrho)[h,\beta]\|_{\infty}\leq C\|\beta\|_{H^{r-1}}\prod_{i=1}^n\|h_i\|_{H^r}.
\end{equation} 
Moreover, the mapping $[\varrho\mapsto B_{n,m}^0(\varrho)]:\cV_{r,M}^m\to\kL_{\rm sym}^n(H^{r}(\bT),\kL(H^{r-1}(\bT), L_\infty(\bT)))$ is locally Lipschitz continuous.
\end{itemize}
\end{lemma}
\begin{proof}
In order to establish (i), we note that there exists a constant $C=C(M)>0$ such that for all $\varrho\in\cV_{r,M}^m$ and 
$\tau,\, s\in (-\pi,\pi)$ we have
\begin{equation*} 
\Bigg|\frac{1}{\prod\limits_{i=1}^m\Big[(\varrho_i(\tau)+\varrho_i(\tau-s))^2+\Big(\frac{\delta_{[\tau,s]}\varrho_i}{t_{[s]}}\Big)^2\Big]}
-\frac{1}{\prod\limits_{i=1}^m 4\omega_{\varrho_i}^2(\tau)}\Bigg|\leq C|s|^{r-3/2}.
\end{equation*}
 Using the latter estimate together with Minkowski's inequality, we derive for
\begin{equation}\label{desc}
\begin{aligned}
A_{n,m}(\varrho)[h,\beta]:=B_{n,m}^0(\varrho)[h,\beta]
-\Big(\prod_{i=1}^m \frac{1}{4\omega_{\varrho_i}^2}\Big)G_{n}[h,\beta]
\end{aligned}
\end{equation}
the   inequality
\begin{align*}
\|A_{n,m}(\varrho)[h,\beta]\|_{2}
&\leq C \bigg(\prod_{i=1}^n \|h_i'\|_\infty\bigg)\int_{-\pi}^{\pi}\frac{|s|^{r-3/2}}{t_{[s]}}
\Big(\int_{-\pi}^{\pi}|\beta(\tau)|^2\, {\rm d}\tau\Big)^{1/2}\, {\rm d}s\\
&\leq C \|\beta\|_2 \prod_{i=1}^n\|h_i'\|_\infty.
\end{align*}
The estimate \eqref{eq:B2} is now a direct consequence of \cite[Lemma A.3]{BM25x}.
 Finally, the  local Lipschitz continuity assertion is a straightforward consequence of~\eqref{difference} and~\eqref{eq:B2}.  \medskip
 
In order to establish \eqref{eq:B2'},  we denote by $K_{n,m}=K_{n,m}(\tau,s)$ the bounded part of the  kernel of the singular integral operator $B_{n,m}^0(\varrho)[h,\cdot]$; that is, 
\[
B_{n,m}^0(\varrho)[h,\beta](\tau)= \PV\int_{-\pi}^\pi K_{n,m}(\tau,s)\frac{\beta(\tau-s)}{t_{[s]}}\,{\rm d}s,\qquad \tau\in\R.
\]
Then
\begin{align*}
B_{n,m}^0(\varrho)[h,\beta](\tau)&=   \int_{\{|s|<\pi\}} K_{n,m}(\tau,s)\frac{\beta(\tau-s)-\beta(\tau)}{t_{[s]}}\,{\rm d}s
 + \beta(\tau) \PV\int_{\{|s|<1\}} \frac{K_{n,m}(\tau,s)}{t_{[s]}}\,{\rm d}s,
\end{align*}
 where the first  term on the right-hand side may be estimated by the right-hand side of \eqref{eq:B2'} due to~$\beta\in{\rm C}^{r-3/2}(\bT)$. 
 Concerning the remaining term we get, in view of
 \[
 |K_{n,m}(\tau,s)-K_{n,m}(\tau,-s)|\leq  C|s|^{r-3/2} \prod_{i=1}^n\|h_i\|_{H^r},
 \]
 that 
 \begin{align*}
 \bigg|\PV\int_{\{|s|<1\}} \frac{K_{n,m}(\tau,s)}{t_{[s]}}\,{\rm d}s \bigg|= \bigg| \int_0^1 \frac{K_{n,m}(\tau,s)-K_{n,m}(\tau,-s)}{t_{[s]}}\,{\rm d}s \bigg|\leq C \prod_{i=1}^n\|h_i\|_{H^r},
 \end{align*}
 and  thus \eqref{eq:B2'}. 
 Since the local Lipschitz continuity assertion follows as before from~\eqref{difference} and \eqref{eq:B2'}, the proof is complete. 
\end{proof}

We next consider the complementary case of Lemma~\ref{L:B1}, where $1\leq p\leq n+1$,  in which the operators $B_{n,m}^p$ are more regular.
\begin{lemma}\label{L:B1p} Let $n,\,m,\, p\in\N$ with $1\leq p\leq n+1$   be given.
\begin{itemize}
\item[(i)] There is  a  constant~${C=C(M)>0}$ such that
for all $h \in {\rm C^1}(\bT)^n $   and~$\varrho\in\cV_{r,M}^m$  it holds that
\begin{equation}\label{eq:B1p1}
\|B_{n,m}^p(\varrho)[h,\cdot]\|_{\kL(L_1(\bT),{\rm C}(\bT))}\leq C\prod_{i=1}^n\|h_i'\|_{\infty}.
\end{equation} 
Moreover, the mapping $[\varrho\mapsto B_{n,m}^p(\varrho)]:\cV_{r,M}^m\to\kL_{\rm sym}^n(W^1_\infty(\bT),\kL(L_1(\bT), {\rm C}(\bT)))$
 is locally Lipschitz continuous.
 \item[(ii)]  For $n\geq 1$ there is  a  constant~${C=C( M)>0}$ such that
for all~${h \in H^1(\bT)\times{\rm C^1}(\bT)^{n-1}}$  and~$\varrho\in\cV_{r,M}^m$  it holds that
\begin{equation}\label{eq:B1p2}
\|B_{n,m}^p(\varrho)[h,\cdot]\|_{\kL(L_2(\bT))}\leq C\|h_1'\|_{2}\prod_{i=2}^n\|h_i'\|_{\infty}.
\end{equation} 
\end{itemize}
\end{lemma}
\begin{proof}
 Regarding (i), we infer from \eqref{inequalities}  and $1\leq p\leq n+1$ that for  $\tau\in\R$, we have 
\begin{align*}
|B_{n,m}^p(\varrho)[h,\beta](\tau)|
&\leq C \bigg(\prod_{i=1}^n \|h_i'\|_\infty\bigg)\int_{-\pi}^{\pi}\Big|\frac{s}{t_{[s]}}\Big|^n|t_{[s]}|^{p-1}|\beta(\tau-s)|\,{\rm d}s
\leq C \|\beta\|_1 \prod_{i=1}^n\|h_i'\|_\infty,
\end{align*}
 which shows in particular that $B_{n,m}^p(\varrho)[h,\cdot]\in\kL(L_1(\bT),L_\infty(\bT))$.
 Moreover, if~$\beta\in{\rm C}(\bT)$, a simple application of the dominated convergence theorem  yields~$B_{n,m}^p(\varrho)[h,\beta]\in{\rm C}(\bT)$, which proves~\eqref{eq:B1p1} via a standard density argument.
The local Lipschitz continuity follows now from  \eqref{difference} and \eqref{eq:B1p1}.

Concerning~(ii), we note  from $|\delta_{[\tau,s]}h_1|\leq \|h_1'\|_2|s|^{1/2}$ and   Minkowski's inequality that
\begin{align*}
\|B_{n,m}^p(\varrho)[h,\beta]\|_2\leq C \|h_1'\|_2\bigg(\prod_{i=1}^n \|h_i'\|_\infty\bigg)\int_{-\pi}^{\pi}|s|^{-1/2}\Big(\int_{-\pi}^{\pi}|\beta(\tau-s)|^2\,{\rm d}\tau\Big)^{1/2}\,{\rm d }s,
\end{align*}  
which proves~\eqref{eq:B1p2}.
\end{proof}

The next regularity result is one of the main ingredients in the analysis of \eqref{QPP}.
\begin{lemma}\label{L:B3} Let    $n,\,m\in\N$  be given. Then,
 there exists a constant $C=C(M)>0$ such that
for all $h \in H^r(\bT)^n $   and $\varrho\in\cV_{r,M}^m$ it holds that
\begin{equation}\label{eq:B3'}
\|B_{n,m}^0(\varrho)[h,\cdot]\|_{\kL(H^{r-1}(\bT))}\leq C\prod_{i=1}^n\|h_i\|_{H^r}.
\end{equation} 
Moreover, the mapping $[\varrho\mapsto B_{n,m}^0(\varrho)]:\cV_{r,M}^m\to\kL_{\rm sym}^n(H^{r}(\bT),\kL(H^{r-1}(\bT)))$ is locally Lipschitz continuous.
\end{lemma}

In the  non-singular case of  $B_{n,m}^p(\varrho)$ with  $1\leq p\leq n+1$ we establish a stronger result. 
\begin{lemma}\label{L:B3p} Given  $m,\,n,\, p\in\N$ with $1\leq p\leq n+1$, there exists  a  constant~${C=C(M)>0}$ such that
for all $h \in {\rm C^1}(\bT)^n $   and $\varrho\in\cV_{r,M}^m$  it holds that
\begin{equation}\label{eq:B3p'}
\|B_{n,m}^p(\varrho)[h,\cdot]\|_{\kL(L_2(\bT),H^{1}(\bT))}\leq C\prod_{i=1}^n\|h_i\|_{{\rm C}^1}.
\end{equation} 
Moreover, the mapping $[\varrho\mapsto B_{n,m}^p(\varrho)]:\cV_{r,M}^m\to\kL_{\rm sym}^n({\rm C}^1(\bT),\kL(L_2(\bT), H^1(\bT)))$
 is locally Lipschitz continuous.
\end{lemma}

The proofs of Lemma~\ref{L:B3} and Lemma~\ref{L:B3p} require some preparation and are therefore postponed until after the proof of the following property:

\begin{lemma}\label{L:B4} Given   $n,\,m\in\N$  with $n\geq1$,
 there exists a positive  constant $C=C(M)$ such that
for all~$\beta \in H^{r-1}(\bT)^n $, $h=(h_1,\ldots,h_n) \in H^r(\bT)^n $,   and $\varrho\in\cV_{r,M}^m$ it holds that
\begin{equation}\label{eq:B4}
\|B_{n,m}^0(\varrho)[h_1,\ldots,h_n,\beta]\|_{2}\leq C\|h_1\|_{H^1}\|\beta\|_{H^{r-1}}\prod_{i=2}^n\|h_i\|_{H^r}.
\end{equation} 
Moreover, the mapping $[\varrho\mapsto B_{n,m}^0(\varrho)]:\cV_{r,M}^m\to\kL^n(H^1(\bT)\times H^{r}(\bT)^{n-1},\kL(H^{r-1}(\bT), L_2(\bT)))$
is locally Lipschitz continuous.
\end{lemma}
\begin{proof}
We may assume that $\beta\in{\rm C}^\infty(\bT)$. To start, we set for $s\in(-\pi,\pi)$ and $\tau\in\R$
\[
\Theta(\tau,s):= \frac{1}{\pi}  \frac{s^2}{t_{[s]}^2}\bigg(\prod\limits_{i=2}^n\frac{\delta_{[\tau,s]}h_i}{t_{[s]}}\bigg)
 \prod\limits_{i=1}^m\Big[(\varrho_i(\tau)+\varrho_i(\tau-s))^2+\Big(\frac{\delta_{[\tau,s]}\varrho_i}{t_{[s]}}\Big)^2\Big]^{-1},
\]
so that 
\[
B_{n,m}^0(\varrho)[h,\beta](\tau)=\PV\int_{-\pi}^\pi \frac{\delta_{[\tau,s]}h_1}{s^2}\Theta(\tau,s)\beta(\tau-s)\,{\rm d}s,\qquad \tau\in\R.
\]
In view of the identities
\[
\frac{\delta_{[\tau,s]}h_1}{s^2}=\frac{h_1'(\tau-s)}{s }-\p_s\Big(\frac{\delta_{[\tau,s]}h_1}{s}\Big)\qquad\text{and}\qquad
\p_s(\delta_{[\tau,s]}\beta)=\beta'(\tau-s),
\]
 we have $B_{n,m}^0(\varrho)[h,\beta]=E_1+E_2+E_3,$
 where
 \begin{align*}
 E_1(\tau)&:=\beta(\tau)\PV\int_{-\pi}^\pi\frac{h_1'(\tau-s)}{s}\Theta(\tau,s)\,{\rm d}s,\\
  E_2(\tau)&:=-\PV\int_{-\pi}^\pi\Big(\frac{\delta_{[\tau,s]}h_1}{s}\Big)
  \Big(\frac{\delta_{[\tau,s]}\beta}{s}\Big)\Theta(\tau,s)\,{\rm d}s,\\
   E_3(\tau)&:=\beta(\tau)\PV\int_{-\pi}^\pi\frac{\delta_{[\tau,s]}h_1}{s}\p_s\Theta(\tau,s)\,{\rm d}s.
 \end{align*}
 
 \noindent{\em Estimate for $E_1$.}
Given $\tau\in\R$, we have
 \begin{align*}
  E_1(\tau)&=\beta(\tau)B_{n-1,m}^0(\rho)[h_2,\ldots,h_n,h_1'](\tau)
  +\beta(\tau)\int_{-\pi}^\pi\frac{h_1'(\tau-s)}{s}\frac{s-2t_{[s]}}{s}\Theta(\tau,s)\,{\rm d}s,
 \end{align*}
 and, using Minkowski's inequality, Lemma~\ref{L:B1}, and the inequalities \eqref{inequalities}, we get
 \begin{equation}\label{EstT1}
 \begin{aligned}
 \|E_1\|_2&\leq C\|\beta\|_\infty\Big(\prod_{i=2}^n\|h_i'\|_{\infty}\Big)\Big[\|h_1'\|_2+\int_{-\pi}^\pi\Big(\int_{-\pi}^\pi |h_1'(\tau-s)|^{2}\,{\rm d}\tau\Big)^{1/2}\,{\rm d}s\Big]\\
& \leq C\|\beta\|_\infty\|h_1'\|_2\prod_{i=2}^n\|h_i'\|_{\infty}.
\end{aligned}
 \end{equation}
 
  \noindent{\em Estimate for $E_2$.} Using Minkowski's inequality, Hölder's inequality, the inequalities~\eqref{inequalities}, 
  and the estimate~$|\delta_{[\tau,s]}h_1|\leq [h_1]_{{\rm C}^{1/2}}|s|^{1/2}$,  where $[\cdot]_{{\rm C}^{1/2}}$ denotes the standard H\"older seminorm,  we get
   \begin{equation}\label{EstT2}
 \begin{aligned}
 \|E_2\|_2&\leq C [h_1]_{{\rm C}^{1/2}}\Big(\prod_{i=2}^n\|h_i'\|_{\infty}\Big)
 \int_{-\pi}^\pi\frac{\|\sfT_s\beta-\beta\|_2}{|s|^{3/2}}  \,{\rm d}s\\
& \leq C\|\beta\|_{H^{r-1}}\|h_1\|_{H^1}\Big(\prod_{i=2}^n\|h_i'\|_{\infty}\Big) 
\Big(\int_{-\pi}^\pi |s|^{2r-4}   \,{\rm d}s\Big)^{1/2}\\
& \leq C\|\beta\|_{H^{r-1}}\|h_1\|_{H^1}\prod_{i=2}^n\|h_i'\|_{\infty}.
\end{aligned}
 \end{equation} 
 
     \noindent{\em Estimate for $E_3$.} The integral $E_3$ can be represented as a sum
     \[
E_3=S_1+\sum_{j=2}^nS_{2,j} +   \sum_{k=1}^mS_{3,k},
     \]
   where, for $\tau\in\R$,
   \begin{align*}
   |S_1(\tau)|&\leq C\|\beta\|_\infty\Big(\prod_{i=2}^n\|h_i'\|_{\infty}\Big)
    \int_{-\pi}^\pi \Big|\frac{\delta_{[\tau,s]}h_1}{s}\p_s\Big(\frac{s^2}{t_{[s]}^2}\Big)\Big|\,{\rm d}s,\\
    |S_{2,j}(\tau)|&\leq C\|\beta\|_\infty\Big(\prod_{i=2,i\neq j}^n\|h_i'\|_{\infty}\Big)
    \int_{-\pi}^\pi \Big|\frac{\delta_{[\tau,s]}h_1}{s}\p_s\Big(\frac{\delta_{[\tau,s]}h_j}{t_{[s]}}\Big)\Big|\,{\rm d}s,\\
        |S_{3,k}(\tau)|&\leq C\|\beta\|_\infty\Big(\prod_{i=2}^n\|h_i'\|_{\infty}\Big)
    \int_{-\pi}^\pi \Big|\frac{\delta_{[\tau,s]}h_1}{s}\Big|
    \bigg(|   \varrho_j'(\tau-s) \big|+  \Big| \p_s \Big(\frac{\delta_{[\tau,s]}\varrho_j}{t_{[s]}}\Big|\bigg)\,{\rm d}s.
\end{align*}     
 Since   $|\p_s\big( s^2/t_{[s]}^2\big)|\leq C|s|$ for $s\in(-\pi,\pi)$ by \eqref{inequalities},   we have
   \begin{equation}\label{EstS1}
 \|S_1\|_2\leq C \|\beta\|_{\infty}\|h_1\|_{\infty}\prod_{i=2}^n\|h_i'\|_{\infty}.
 \end{equation}
 To estimate the terms $S_{2,j}$ and $S_{3,k}$  we use the same strategy that we thus detail only for the first term.
 Since for $\tau\in\R$ and $s\in(-\pi,\pi)$ with $s\neq 0$ we have by \eqref{inequalities}
 \[
 \Big|\p_s\Big(\frac{\delta_{[\tau,s]}h_j}{t_{[s]}}\Big)\Big|
 \leq C\|h_j\|_{H^r}+C\frac{\big|\delta_{[\tau,s]}h_j-h_j'(\tau-s)\big|}{s^2},
 \]
Minkowski's inequality together with the embedding $H^1(\bT)\hookrightarrow {\rm C}^{1/2}(\bT)$ leads to
\begin{align*}
\|S_{2,j}\|_2&\leq C\|\beta\|_\infty\|h_1\|_{H^1}\Big(\prod_{i=2,i\neq j }^n\|h_i'\|_{H^r}\Big)
  \bigg(\|h_j\|_{H^r}\int_{-\pi}^\pi \frac{1}{|s|^{1/2}}\,{\rm d} s  
+ J_j\bigg),
\end{align*}
where, using also the fundamental theorem of calculus and H\"olders inequality,
\begin{align*}
J_j&= \int_{-\pi}^\pi\frac{1}{|s|^{5/2}}\Big(\int_{-\pi}^\pi\big|\delta_{[\tau,s]}h_j- sh_j'(\tau-s)\big|^2\,{\rm d}\tau\Big)^{1/2}\,{\rm d}s \\
&\leq\int_0^1\int_{-\pi}^\pi\frac{1}{|s|^{3/2}}\Big(\int_{-\pi}^\pi\big| h_j'(\tau+as)-h_j'(\tau)\big|^2\,{\rm d}\tau\Big)^{1/2}\,{\rm d}s\,{\rm d}a\\
&\leq \int_{-\pi}^\pi\frac{\|h_j'-\sfT_s h_j'\|_ 2}{|s|^{3/2}} \,{\rm d}s \leq C\|h_j\|_{H^r}\Big(\int_{-\pi}^\pi |s|^{2r-4}  \,{\rm d}s \Big)^{1/2}.
\end{align*}
Consequently, for $2\leq j\leq n$ and $1\leq k\leq m$ we have
  \begin{equation}\label{EstS23}
 \|S_{2,j}\|_2+\|S_{3,k}\|_2\leq C \|\beta\|_{\infty}\|h_1\|_{H^1}\prod_{i=2}^n\|h_i\|_{H^r}.
 \end{equation}
 The desired estimate~\eqref{eq:B4} follows now from \eqref{EstT1}-\eqref{EstS23}.
\end{proof}

We next establish the proof of Lemma~\ref{L:B3}.
\begin{proof}[Proof of Lemma~\ref{L:B3}] To start, we fix $\varrho=(\varrho_1,\ldots,\varrho_m)\in\cV_{r,M}^m$, $h=(h_1,\ldots,h_n) \in H^r(\bT)^n $, and~${\beta \in H^{r-1}(\bT)}$.
We first establish the claim for $r\in(3/2,2)$, in which case  it remains  to estimate the $[\cdot]_{H^{r-1}}$-seminorm of $B_{n,m}^0(\varrho)[h,\beta]$ in view of  \eqref{equiv} and \eqref{eq:B2}.
To this end, using \eqref{difference} we write 
\begin{equation}\label{ffor}
\sfT_\tau \big(B_{n,m}^0(\varrho)[h,\beta]\big)-B_{n,m}^0(\varrho)[h,\beta]=\sum_{j=1}^3E_j(\tau),\qquad\tau\in\R,
\end{equation} 
where
\begin{align*}
E_1(\tau)&:=B_{n,m}^0(\sfT_\tau\varrho)[\sfT_\tau h,\sfT_\tau\beta-\beta],\\
E_2(\tau)&:=\sum_{j=1}^nB_{n,m}^0(\sfT_\tau\varrho)[h_1,\ldots,h_{j-1}, \sfT_\tau h_j-h_j,\sfT_\tau h_{j+1} ,\ldots, \sfT_\tau h_n,  \beta],\\
E_3(\tau)&:=\sum_{j=1}^m\Big\{
B_{n+2,m+1}^{0}(\varrho_1,\ldots,\varrho_j,\sfT_\tau\varrho_j,\ldots,\sfT_\tau\varrho_m)
[h, \varrho_j+\sfT_\tau\varrho_j, \varrho_j-\sfT_\tau\varrho_j,\beta]\\[-2ex]
&\hspace{1.275cm}+( \varrho_j^2-\sfT_\tau\varrho_j^2)B_{n,m+1}^{0}(\varrho_1,\ldots,\varrho_j,\sfT_\tau\varrho_j,\ldots,\sfT_\tau\varrho_m)[h,\beta]\\
&\hspace{1.275cm}+(\varrho_j+\sfT_\tau\varrho_j)B_{n,m+1}^{0}(\varrho_1,\ldots,\varrho_j,\sfT_\tau\varrho_j,\ldots,\sfT_\tau\varrho_m)[h,(\varrho_j-\sfT_\tau\varrho_j)\beta]\\
&\hspace{1.275cm}+( \varrho_j-\sfT_\tau\varrho_j)B_{n,m+1}^{0}(\varrho_1,\ldots,\varrho_j,\sfT_\tau\varrho_j,\ldots,\sfT_\tau\varrho_m)[h,(\varrho_j+\sfT_\tau\varrho_j)\beta]\\ 
&\hspace{1.275cm}+ B_{n,m+1}^{0}(\varrho_1,\ldots,\varrho_j,\sfT_\tau\varrho_j,\ldots,\sfT_\tau\varrho_m)[h,( \varrho_j^2-\sfT_\tau\varrho_j^2)\beta]
\Big\}.
\end{align*}
Applying Lemma~\ref{L:B1} and Lemma~\ref{L:B4}, we deduce for some constant $C>1$ that for all $\tau\in\R$ we have
\begin{align*}
\sum_{j=1}^3\|E_j(\tau)\|_2&\leq C\bigg(\big(\|\sfT_\tau\beta-\beta\|_2+\|\beta\|_{H^{r-1}}\|\sfT_\tau\varrho-\varrho\|_{H^1}\big)\prod_{i=1}^n\|h_i\|_{H^r}\\
&\hspace{1cm}+\|\beta\|_{H^{r-1}}\sum_{j=1}^n \|\sfT_\tau h_j-h_j\|_{H^1}\prod_{i=1,i\neq j}^n\|h_i\|_{H^r}\bigg),
\end{align*}
and the assertion \eqref{eq:B3'} with $r\in(3/2,2)$ follows now directly in view of \eqref{equiv}.

We  next address the case $r=2$. Dividing \eqref{ffor} by $\tau\neq0$, we infer from Lemma~\ref{L:B1} and Lemma~\ref{L:B4} that $\big(\sfT_\tau \big(B_{n,m}^0(\varrho)[h,\beta]\big)-B_{n,m}^0(\varrho)[h,\beta]\big)/\tau$ converges in $L_2(\bT)$
as $\tau\to0$ towards the weak derivative
\begin{equation}\label{forder}
\begin{aligned}
(B_{n,m}^0(\varrho)[h,\beta])'&=B_{n,m}^0(\varrho)[ h,\beta']+\sum_{j=1}^nB_{n,m}^0(\varrho)[h_1,\ldots,h_{j-1}, h_j', h_{j+1} ,\ldots,  h_n,  \beta]\\[-1ex]
&\quad-2\sum_{j=1}^m\Big\{B_{n+2,m+1}^{0}(\varrho)[h, \varrho_j, \varrho_j',\beta]+\varrho_j\varrho_j'B_{n,m+1}^{0}(\varrho,\varrho_j)[h,\beta]\\[-1ex]
&\hspace{1.875cm}+\varrho_jB_{n,m+1}^{0}(\varrho,\varrho_j)[h,\varrho_j'\beta]+ \varrho_j'B_{n,m+1}^{0}(\varrho , \varrho_j)[h, \varrho_j \beta]\\
&\hspace{1.875cm}+ B_{n,m+1}^{0}(\varrho ,\varrho_j)[h,\varrho_j\varrho_j'\beta]
\Big\}.
\end{aligned}
\end{equation}
Thus, $B_{n,m}^0(\varrho)[h,\beta]\in H^1(\bT)$ and the estimate \eqref{eq:B3'} with $r=2$ follows now from Lemma~\ref{L:B1} and Lemma~\ref{L:B4}.

Since the Lipschitz continuity property for $r\in(3/2,2]$ is a straightforward consequence of \eqref{difference} and \eqref{eq:B3'}, this completes the proof.
\end{proof}

We are now in a position to prove Lemma~\ref{L:B3p}.
\begin{proof}[Proof of Lemma~\ref{L:B3p}]
We first  assume that  $(h,\beta)\in{\rm C}^\infty(\bT)^{n+1}$. 
 Denoting by $K_{n,m}^p=K_{n,m}^p(\tau,s)$  the integral kernel of $B_{n,m}^p(\rho)[h,\cdot]$; that is,
\begin{equation*}
B_{n,m}^p(\varrho )[h,\beta](\tau)=\int_{-\pi}^\pi K_{n,m}^p(\tau,s)\beta(\tau-s)\,{\rm d}s,\qquad \tau\in\R,
\end{equation*}
the theorem on differentiation of parameter integrals ensures that~$B_{n,m}^p(\varrho )[h,\beta]\in{\rm C}^1(\bT)$ with
\begin{equation}\label{derH1}
\begin{aligned}
(B_{n,m}^p(\varrho )[h,\beta]\big)'(\tau)&=\int_{-\pi}^\pi \p_\tau K_{n,m}^p(\tau,s)\beta(\tau-s)-K_{n,m}^p(\tau,s)\p_s(\beta(\tau-s))\,{\rm d}s\\
&=\PV\int_{-\pi}^\pi \big[(\p_\tau+\p_s)K_{n,m}^p(\tau,s)\big]\beta(\tau-s)\,{\rm d}s
\end{aligned}
\end{equation}
for $\tau\in\R$, where integration by parts is used in the last step.
Hence,
\begin{align*}
(B_{n,m}^p(\varrho )[h,\beta]\big)'&=\frac{p-n-1}{2}\big(B_{n,m}^{p-1}(\varrho )[h,\beta]+B_{n,m}^{p+1}(\varrho )[h,\beta]\big)\\
&\quad+\sum_{j=1}^nh_j'B_{n-1,m}^{p-1}(\varrho )[h_1,\ldots,h_{j-1},h_{j+1},\ldots, h_n,\beta]\\
&\quad-\sum_{j=1}^m\Big[2\varrho_j\varrho_j'B_{n,m+1}^{p}(\varrho,\varrho_j )[h,\beta]+2\varrho_j'B_{n,m+1}^{p}(\varrho,\varrho_j )[h,\varrho_j\beta]\\
&\hspace{1.5cm}+2\varrho_j'B_{n+1,m+1}^{p-1}(\varrho,\varrho_j )[h,\varrho_j,\beta]-B_{n+2,m+1}^{p-1}(\varrho,\varrho_j )[h,\varrho_j,\varrho_j,\beta]\\
&\hspace{1.5
cm}-B_{n+2,m+1}^{p+1}(\varrho,\varrho_j )[h,\varrho_j,\varrho_j,\beta]\Big],
\end{align*} 
the right-hand side of the latter identity belonging to~$L_2(\R)$  in view of Lemma~\ref{L:B1} and Lemma~\ref{L:B1p},  even when merely assuming~$h\in{\rm C}^1(\bT)^{n}$ and $\beta\in L_2(\bT)$. 
The desired claims follow now by a standard density argument from Lemma~\ref{L:B1}, Lemma~\ref{L:B1p},  and \eqref{difference}.
\end{proof}

We finally address the Fr\'echet  differentiability of the mapping $[\rho\mapsto {\sf B}_{n,m}^p(\rho)]$ defined in~\eqref{sfB} and prove that this mapping is actually smooth.  
\begin{lemma}\label{L:B5} 
 Given $n,\,m,\, p\in\N$ with $0\leq p\leq n+1$, the  following properties hold:
 \begin{itemize}
 \item[(i)] The mapping $[\rho\mapsto \sfB_{n,m}^0(\rho)]:\cV_r\to\kL(H^{r-1}(\bT))$ is smooth.
\item[(ii)]  If $1\leq p\leq n+1$, then $[\rho\mapsto \sfB_{n,m}^p(\rho)]:\cV_r\to\kL(L_2(\bT),H^1(\bT))$ is smooth.
  \end{itemize}
\end{lemma}
\begin{proof}
To start, we introduce for $q\in\N$ the multilinear operator
\[
\sfB_{n,m}^{p,q}(\rho)[f][\beta]:=B_{n+q,m}^{p}(\rho,\ldots,\rho)
[\underset{n}{\underbrace{\rho,\ldots,\rho}},f,\beta],
\]
where $f:=(f_1,\ldots,f_q)$, and infer from Lemma~\ref{L:B3} and Lemma~\ref{L:B3p} that  
\[ 
\sfB_{n,m}^{0,q}\in {\rm C}^{1-}(\cV_r, \kL^q_{\rm sym}(H^r(\bT),\kL(H^{r-1}(\bT)))),
\]
 respectively 
\[
\sfB_{n,m}^{p,k}\in {\rm C}^{1-}(\cV_r, \kL^q_{\rm sym}({\rm C}^1(\bT),\kL(L_2(\bT), H^1(\bT))))\qquad\text{for $1\leq p\leq n+1$.}
\]
Given  $\rho\in\cV_r$, we next prove that both  mappings are Fr\'echet differentiable with 
\begin{equation}\label{forfreder}
\begin{aligned}
\p\sfB_{n,m}^{p,q}(\rho)[h][f]&=\big(n\sfB_{n-1,m}^{p,q+1}(\rho)-2m\sfB_{n+1,m+1}^{p,q+1}(\rho)\big)[h,f]\\
&\quad-2m\rho\big( h\sfB_{n,m+1}^{p,q}(\rho)[f]+\sfB_{n,m+1}^{p,q}(\rho)[f][h\,\cdot]\big)\\
&\quad-2m\big( h\sfB_{n,m+1}^{p,q}(\rho)[f][\rho\cdot]+\sfB_{n,m+1}^{p,q}(\rho)[f][\rho h\,\cdot]\big)
\end{aligned}
\end{equation}
for   $h\in H^r(\bT)$ and $f\in H^{r}(\bT)^q$ if $p=0$, respectively with 
$f\in {\rm C}^1(\bT)^q$ if $1\leq p\leq n+1$. The assertion of the lemma is then a straightforward consequence of \eqref{forfreder}.

In order to prove~\eqref{forfreder},  we define for~$\rho\in\cV_r$, $h \in H^r(\bT)$ with~$\|h\|_{H^r}\ll 1$ (to ensure in particular that $\rho, \rho+h\in\cV_{r,M}$ for some $M>1$) and
 $(f,\beta)\in {\rm C^\infty}(\bT)^{q+1}$ 
the  rest
\begin{align*}
R^p(h)&:=\sfB_{n,m}^{p,q}(\rho+h)[f,\beta]-\sfB_{n,m}^{p,q}(\rho)[f,\beta]\\
&\quad\,\,-\big(n\sfB_{n-1,m}^{p,q+1}(\rho)-2m\sfB_{n+1,m+1}^{p,q+1}(\rho)\big)[h,f][\beta]\\
&\quad\,\,+2m\rho\big(h\sfB_{n,m+1}^{p,q}(\rho)[f][\beta]+\sfB_{n,m+1}^{p,q}(\rho)[f][h\beta]\big)\\
&\quad\,\,+2m\big(h\sfB_{n,m+1}^{p,q}(\rho)[f][\rho\beta]+\sfB_{n,m+1}^{p,q}(\rho)[f][\rho h\beta]\big).
\end{align*}
It remains to show  that there exists a constant $C>0$ such that  for all $h\in H^r(\bT)$ with~$\|h\|_{H^r}\ll 1$  and all $(f,\beta)\in {\rm C^\infty}(\bT)^{q+1}$   we have 
\begin{equation}\label{esfin1}
\|R^0(h)\|_{H^{r-1}}\leq C\|h\|_{H^r}^2\|\beta\|_{H^{r-1}}\prod_{i=1}^{q}\|f_i\|_{H^r},  
\end{equation}
respectively, for $1\leq p\leq n+1$,
\begin{equation}\label{esfin2}
\|R^p(h)\|_{H^1}\leq C\|h\|_{H^r}^2\|\beta\|_{2}\prod_{i=1}^{q}\|f_i\|_{{\rm C}^1}.
\end{equation}
Using elementary algebraic manipulations, we may write $R^p(\rho)$  as a linear combination of terms of the form
\[
P(\rho)h^{k_1}B_{\wt n,\wt m}^{p}(\underset{\text{$\ell$}}{\underbrace{\rho,\ldots,\rho}}, h+\rho,\ldots,h+\rho)
[f,\rho,\ldots,\rho,\underset{\text{$k_2$}}{\underbrace{h,\ldots,h}},Q(\rho)h^{k_3}\beta],
\]
where $P,Q$ are polynomials, $\ell, \wt n,\wt m\in\N$ satisfy  $n+q\leq \wt n $ (in particular $p\leq \wt n+1$) and~$\ell\leq \wt m$, and~${k_1,k_2,k_3\in\N}$ fulfill   
\[
k_1+k_2+k_3\geq 2.
\]   
The desired estimates \eqref{esfin1}-\eqref{esfin2} are now straightforward consequences of Lemma~\ref{L:B3} (if~${p=0}$) and Lemma~\ref{L:B3p} (if $1\leq p\leq n+1$).
\end{proof}

\subsection*{A second family of (singular) integral operators} 
We introduce a  further family of singular integral operators used in the proof of Lemma~\ref{L:AB3} below by defining, 
for given integers~$m,\,n\in\N$    and ${\varrho:=(\varrho_1,\ldots,\varrho_m)\in\cV^m}$,  
\begin{equation}\label{Hnmp}
H_{n,m}(\varrho)[h,\beta](\tau):=\displaystyle\frac{1}{\pi}\PV\int_{-\pi}^\pi\frac{\prod\limits_{i=1}^n\frac{\delta_{[\tau,s]}h_i}{s}}
{\prod\limits_{i=1}^m\Big[(\varrho_i(\tau)+\varrho_i(\tau-s))^2+\Big(\frac{2\delta_{[\tau,s]}\varrho_i}{s}\Big)^2\Big]}
\frac{\beta(\tau-s)}{s}\,{\rm d}s
\end{equation}
where $h=(h_1,\ldots, h_n):\R\to\R^n$ is Lipschitz continuous, $\beta\in L_2(\bT)$, and~${\tau\in\R}$.
When the components of $\varrho$ and $h$ are equal to $\rho\in\cV_r$, we set
\begin{equation}\label{sfH}
 \sfH_{n,m}(\rho):=H_{n,m}(\rho,\ldots,\rho )[\rho,\ldots,\rho,\cdot].
\end{equation} 

\begin{lemma}\label{L:C1} Let $n,\,m\in\N$ and $M>0$  be given.
\begin{itemize}
\item[(i)] There exists a constant $C=C(M)>0$ such  that for all ${\varrho \in\cV_{r,M}^m}$,   $\theta\in\R$, and  
 all  Lipschitz functions~$h:\R\to\R^n $ we have 
\begin{equation}\label{eq:C1}
\|H_{n,m}(\varrho)[h,\cdot]\|_{\kL(L_2(\bT), L_2((-\pi+\theta,\pi+\theta)))}\leq C\prod_{i=1}^n\|h_i'\|_{\infty}.
\end{equation} 
\item[(ii)] Assume that $n\geq 1$. Then, there exists a constant $C=C(M)>0$  with the property that
for all~$\beta \in H^{r-1}(\bT)$, $h \in H^r(\bT)^n$,   and~$\varrho \in\cV_{r,M}^m$ it holds that
\begin{equation}\label{eq:C2}
\|H_{n,m}(\varrho)[h,\beta]\|_{2}\leq C\|h\|_{H^1}\|\beta\|_{H^{r-1}}\prod_{i=2}^n\|h_i\|_{H^r}.
\end{equation} 
\item[(iii)] Given $\rho\in\cV_2$ and $\beta \in H^{1}(\bT)$, it holds that $\sfH_{n,m}(\rho)[\beta]\in H^1(\bT)$ with 
\begin{equation}\label{eq:C3}
\begin{aligned}
(\sfH_{n,m}(\rho)[\beta])'&=\sfH_{n,m}(\rho)[\beta']+nH_{n,m}(\rho,\ldots,\rho)[\rho,\ldots,\rho,\rho',\beta]\\
&\quad-2m\big\{\rho\sfH_{n,m+1}(\rho)[\rho'\beta]+\rho\rho'\sfH_{n,m+1}(\rho)[\beta]\\
&\hspace{1.55cm}+\rho'\sfH_{n,m+1}(\rho)[\rho\beta]+ \sfH_{n,m+1}(\rho)[\rho\rho'\beta]\\
&\hspace{1.55cm}+4H_{n+2,m+1}(\rho,\ldots,\rho)[\rho,\ldots,\rho,\rho',\beta]\big\}.
\end{aligned}
\end{equation} 
\end{itemize}
\end{lemma}
\begin{proof} 
 Claim (i) can be established by following the arguments used in the proof of Lemma~\ref{L:B1}~(i), while assertion~(ii) is obtained by a similar approach to that in Lemma~\ref{L:B4}.
Finally, \eqref{eq:C3} follows from~(i) and (ii) by arguing as in the derivation of \eqref{forder}.  
\end{proof}

The operator $\sfH_{n,m}(\rho)$ is related to the operator $\sfB^0_{n,m}(\rho)$ in the sense that the difference
\begin{equation}\label{identbah}
\sfA_{n,m}(\rho):=\sfB^0_{n,m}(\rho)-2^{n+1}\sfH_{n,m}(\rho)
\end{equation}
is regularizing, as the next result shows.
\begin{lemma}\label{L:C2}
Given $n,\,m\in\N$ and $M>0$, there exists a constant $C=C(M)>0$ such that for all~${\rho\in\cV_{r,M}}$  it holds  that
\[
\|\sfA_{n,m}(\rho)\|_{\kL({\rm C}(\bT), {\rm C}^1(\bT))}\leq C.
\]
\end{lemma}
\begin{proof}
 For $\ell\in\{0,\,1\}$, $\tau\in\R$, and $0\neq s\in(-\pi,\pi)$ we define
\begin{equation*}
K_{n,m}^\ell(\tau,s):= \frac{\frac{1}{\pi t_{[s]}^{1+\ell}}\Big(\frac{\delta_{[\tau,s]}\rho}{t_{[s]}}\Big)^n}
{\Big((\rho(\tau)+\rho(\tau-s))^2+\Big(\frac{\delta_{[\tau,s]}\rho}{t_{[s]}}\Big)^2\Big)^m}-
\frac{\frac{2^{1+\ell}}{\pi s^{1+\ell}}\Big(\frac{2\delta_{[\tau,s]}\rho}{s}\Big)^n}
{\Big((\rho(\tau)+\rho(\tau-s))^2+\Big(\frac{2\delta_{[\tau,s]}\rho}{s}\Big)^2\Big)^m}
\end{equation*}
and we denote by $\sfA_{n,m}^\ell(\rho)$ the integral operator with kernel $K_{n,m}^\ell$; that is
\[
\sfA_{n,m}^\ell(\rho)[\beta](\tau):=\int_{-\pi}^\pi K_{n,m}^\ell(\tau,s)\beta(\tau-s)\, {\rm d}s,\qquad \tau\in\R.
\]
We have $\sfA_{n,m}^0(\rho)=\sfA_{n,m}(\rho)$.
Using \eqref{inequalities}, it is not difficult to find a constant $C=C(M)>0$ such that
\begin{equation}\label{est:KA}
|K_{n,m}^\ell(\tau,s)|\leq C|s|^{1-\ell},\qquad 0\neq s\in(-\pi,\pi),\, \tau\in\R,\, \ell\in\{0,1\}.
\end{equation}
Moreover, if $0\neq s\in(-\pi,\pi)$, then $K_{n,m}^\ell(\cdot,s)\beta(\cdot-s):\bT\to\R$ is continuous and the
 theorem on the continuity of parameter integrals ensures that $\|\sfA_{n,m}^\ell(\rho)\|_{\kL({\rm C}(\bT))}\leq C$ for~$\ell\in\{0,1\}$.

To prove that actually $\sfA_{n,m}(\rho)\in\kL({\rm C}(\bT), {\rm C}^1(\bT))$, we now assume that $\beta\in{\rm C}^1(\bT)$.
Then, 
\begin{align*}
&K_{n,m}^0(\cdot,s)\beta(\cdot-s)\in {\rm C}^1(\bT)\qquad\text{for all $0\neq s\in(-\pi,\pi)$,}\\
&K_{n,m}^0(\tau,\cdot)\beta(\tau-\cdot)\in {\rm C}^1([-\pi,\pi])\qquad\text{for all $\tau\in\R$.}
\end{align*}
Using \eqref{est:KA}, Fubini's theorem, and integration by parts, we conclude that $\sfA_{n,m}(\rho)[\beta]$ is weakly differentiable with
\[
(\sfA_{n,m}(\rho)[\beta])'(\tau)=(K_{n,m}^0(\tau,-\pi)-K_{n,m}^0(\tau,\pi)\beta(\tau-\pi)+\int_{-\pi}^\pi\big[(\p_\tau+\p_s) K_{n,m}^0(\tau,s)\big]\beta(\tau-s)\,{\rm d}s
\]
for $\tau\in\R$, or equivalently
\begin{align*}
(\sfA_{n,m}(\rho)[\beta])'&=\beta(\cdot-\pi)\frac{2^{n+1}\frac{1+(-1)^{n}}{\pi^2}\Big(\frac{\delta_{[\cdot,\pi]}\rho}{\pi}\Big)^n}
{\prod\limits_{i=1}^m\Big[(\rho +\rho(\cdot-\pi))^2+\Big(\frac{2\delta_{[\cdot,\pi]}\rho}{\pi}\Big)^2\Big]}+n\rho'\sfA^1_{n-1,m}(\rho)[\beta]\\
&\quad-\frac{n+1}{2}\big(\sfA^1_{n,m}(\rho)+\sfB_{n,m}^1(\rho)\big)[\beta]-2m\rho'\sfA^1_{n+1,m+1}(\rho)[\beta]\\
&\quad+m\big(\sfA^1_{n+2,m+1}(\rho)+\sfB^1_{n+2,m+1(\rho)}\big)[\beta]\\
&\quad -2m\rho\rho'\sfA_{n,m+1}(\rho)[\beta]-2m \rho'\sfA_{n,m+1}(\rho)[\rho\beta].
\end{align*}
Lemma~\ref{L:B1p} together with $\|\sfA_{n,m}^\ell(\rho)\|_{\kL({\rm C}(\bT))}\leq C$, $\ell\in\{0,1\}$, yield the desired claim via a density argument.
\end{proof}

 \section{Localization of the operators $\sfB_{n,m}^p(\rho)$}\label{SEC:B}

This section presents commutator and localization results for the operators~$\sfB_{n,m}^p(\rho)$, 
which play a crucial role in the proofs of Lemma~\ref{L:loc} and Lemma~\ref{L:localpha}.
Throughout this section we fix again~${3/2<r'<r\leq 2}$ and~$M>1$. We   refer to~\eqref{pi_j}-\eqref{eqnorm} and recall the definition~\eqref{defVM} of~$\cV_{r,M}$.
To start, we establish the following commutator property.

\begin{lemma}\label{L:AB1}
Let $n,\,m,\,p\in\N$  satisfy $p\leq n+1$ and fix $a\in {\rm C}^1(\bT)$. 
Then, there exists a positive constant~$C=C(M,\|a\|_{{\rm C}^1})$ such that for all $\rho\in\cV_{r,M}$  it holds that
\begin{equation}\label{comest}
\|\llbracket a,\sfB_{n,m}^p(\rho)\rrbracket\|_{\kL(L_2(\bT), H^1(\bT))}\leq C.
\end{equation} 
\end{lemma}
\begin{proof}
If $1\leq p\leq n+1$, the  claim follows directly from Lemma~\ref{L:B3p}.
The claim for  $p=0$ follows by arguing as in the proof of \cite[Lemma~12]{AM22} and therefore we omit the details. 
\end{proof}

The following localization result is  repeatedly used in the proof of Lemma~\ref{L:localpha} and Lemma~\ref{L:loc}, 
as it provides in particular a localization of the operator $\sfB_{n,m}^p(\rho) $ with $\rho\in\cV_r$ by Fourier multipliers which differ 
from the Hilbert transform $H=2\sfB_{0,1}^0(1)$ only by a multiplicative constant. 
\begin{lemma}\label{L:AB2}
		Let $n,\,m\in\N$, $a,\,b\in H^{r-1}(\bT)$, and~${\mu>0}$. 
		Then,  for each sufficiently small~${\e\in(0,1)}$, there is a constant~$K=K(\e,M)>0$ such that for all~$\rho\in\cV_{r,M}$, ${1\leq j\leq q(\e)}$, and~$\beta\in H^{r-1}(\bT)$ it holds that
		\begin{equation}\label{locbnm}
			\bigg\|\pi_j^\e a \sfB_{n,m}^0 (\rho)[b\beta]
			-\frac{2^{n-1}ab\rho'^n}{(2\omega_\rho)^{2m}}(\tau_j^\e) H[\pi_j^\e \beta]\bigg\|_{H^{r-1}}
			\leq \mu \|\pi_j^\e \beta\|_{H^{r-1}}+K\|\beta\|_{H^{r'-1}}.
		\end{equation}
	\end{lemma}
	\begin{proof}
We compute
		\begin{equation*}
			\pi_j^\e a \sfB_{n,m}^0 (\rho)[b\beta]-\frac{2^{n-1}ab\rho'^n}{(2\omega_\rho)^{2m}}(\tau_j^\e)  H[\pi_j^\e \beta]
			=a(E_1+E_2)+b(\tau_j^\e)\big(E_3+ a(\tau_j^\e)E_4\big),
		\end{equation*}
where
		\begin{equation*}
		\begin{aligned}
			&E_1 := \llbracket\pi_j^\e, \sfB_{n,m}^0 (\rho)\rrbracket[(b-b(\tau_j^\e))\beta],
			& & E_2 := \sfB_{n,m}^0(\rho)[\pi_j^\e(b-b(\tau_j^\e))\beta],\\
			&E_3 := \pi_j^\e a \sfB_{n,m}^0(\rho)[\beta]-a(\tau_j^\e)\sfB_{n,m}^0 (\rho)[\pi_j^\e \beta],
			& &E_4 := \sfB_{n,m}^0(\rho)[\pi_j^\e\beta]-\frac{(2\rho')^n}{(2\omega_\rho)^{2m}}(\tau_j^\e) \sfB_{0,1}^0(1)[\pi_j^\e \beta].
		\end{aligned}
		\end{equation*}
		Using Lemma~\ref{L:AB1}, we have for $1\leq j\leq q(\e)$ and $\beta\in H^{r-1}(\bT)$
		\begin{equation}\label{TES1}
		\|aE_1\|_{H^{r-1}}\leq C\|E_1\|_{H^1}\leq K\|(b-b(\tau_j^\e))\beta\|_2\leq K\|\beta\|_{H^{r'-1}}.
		\end{equation}
		Moreover, since $\chi_j^\e\pi_j^\e=\pi_j^\e$ by \eqref{chi_j},  Lemma~\ref{L:B3} together with \eqref{algebra} lead to
		\begin{equation}\label{TES2}
		\begin{aligned}
			\|aE_2\|_{H^{r-1}}&\leq C\|\chi_j^\e(b-b(\tau_j^\e))\pi_j^\e\beta\|_{H^{r-1}}\leq C\|\chi_j^\e(b-b(\tau_j^\e))\|_\infty\|\pi_j^\e\beta\|_{H^{r-1}}+K\|\pi_j^\e\beta\|_\infty\\
		&\leq (\mu/3)\|\pi_j^\e\beta\|_{H^{r-1}}+K\|\beta\|_{H^{r'-1}}
			\end{aligned}
		\end{equation}	
		if $\e\in(0,1)$ is sufficiently small, due to the fact that $b\in {\rm C}^{r-3/2}(\bT)$ and \eqref{chi_j}.
		
Using \eqref{chi_j}, the term $E_3$ can be decomposed as  
\[
E_3= \chi_j^\e(a-a(\tau_j^\e)) \sfB_{n,m}^0(\rho)[\pi_j^\e\beta] - a\chi_j^\e\llbracket\pi_j^\e, \sfB_{n,m}^0 (\rho)\rrbracket[\beta] + a(\tau_j^\e)\llbracket\chi_j^\e, \sfB_{n,m}^0 (\rho)\rrbracket[\pi_j^\e\beta].
\]  
 We may then proceed as in \eqref{TES1}-\eqref{TES2} to deduce from Lemma~\ref{L:B3}, Lemma~\ref{L:AB1},~\eqref{algebra},  and the fact~\(a\in {\rm C}^{r-3/2}(\bT)\), for sufficiently small \(\e\in(0,1)\), that  
\begin{equation}\label{TES3}
\begin{aligned}
\|b(\tau_j^\e)E_3\|_{H^{r-1}} &\leq C\|\chi_j^\e(a-a(\tau_j^\e))\|_\infty \|\sfB_{n,m}^0(\rho)[\pi_j^\e\beta]\|_{H^{r-1}} + K\|\beta\|_{H^{r'-1}} \\
&\leq (\mu/3)\|\pi_j^\e\beta\|_{H^{r-1}} + K\|\beta\|_{H^{r'-1}}.
\end{aligned}
\end{equation}  

Concerning $E_4$, we use again  \eqref{chi_j} to write
\[
E_4=\chi_j^\e E_4-\llbracket\chi_j^\e, \sfB_{n,m}^0 (\rho)\rrbracket[\pi_j^\e\beta] +
\frac{(2\rho')^n}{(2\omega_\rho)^{2m}}(\tau_j^\e) \llbracket\chi_j^\e, \sfB_{0,1}^0(1)\rrbracket[\pi_j^\e\beta],
\]
and, using Lemma~\ref{L:AB1}, we find a constant $C_1>0$ such that for all $\e\in(0,1)$, $1\leq j\leq q(\e)$, and~${\beta\in H^{r-1}(\bT)}$ we have
\begin{equation}\label{TES4a}
\|(ab)(\tau_j^\e)E_4\|_{H^{r-1}}\leq C_1\|\chi_j^\e E_4\|_{H^{r-1}}+K\|\beta\|_{H^{r'-1}}.
\end{equation}
To estimate the term $\|\chi_j^\e E_4\|_{H^{r-1}}$,  we infer from  Lemma~\ref{L:B1} that 
\begin{equation}\label{TES4b1}
\|\chi_j^\e E_4\|_2\leq K\|\beta\|_2\leq K\|\beta\|_{H^{r'-1}}.
\end{equation}
It remains to estimate the term $[\chi_j^\e E_4]_{H^{r-1}}$ if $r\in(3/2,2)$, respectively the norm~$\|(\chi_j^\e E_4)'\|_2$ if $r=2$.

Let first $r\in(3/2,2)$.
Then, since $\rho\in\cV_{r,M}$,   elementary algebraic manipulations imply that the seminorm $[\chi_j^\e E_4]_{H^{r-1}}$ can be estimated, 
up to a multiplicative constant $C>0$, by a finite sum of terms of the form
\begin{equation*}
F_1:=\|\chi_j^\e(\rho-\rho(\tau_j^\e))\sfB_{0,\wt m }^0(\rho)[\rho^k\pi_j^\e\beta]\|_{H^{r-1}}+\|\chi_j^\e\sfB_{0,\wt m }^0(\rho)[(\rho-\rho(\tau_j^\e))\rho^k\pi_j^\e\beta]\|_{H^{r-1}}
\end{equation*}
and 
\begin{equation*}
F_2:=\big[\chi_j^\e\big(\sfB_{\wt n+1,\wt m}^0(\rho)-2\rho'(\tau_j^\e)\sfB_{\wt n,\wt m}^0(\rho)\big)[\pi_j^\e\beta]\big]_{H^{r-1}}
\end{equation*}
with $k\in\{0,1\}$ and $\wt n,\,\wt m\in\N$. 
Thus, recalling \eqref{TES4a} and \eqref{TES4b1}, in order to show that
\begin{equation}\label{TES4}
\|(ab)(\tau_j^\e)E_4\|_{H^{r-1}}\leq (\mu/3)\|\chi_j^\e E_4\|_{H^{r-1}}+K\|\beta\|_{H^{r'-1}}
\end{equation}
for each  sufficiently small $\e\in(0,1)$, uniformly in $\rho\in\cV_{r,M}$, $1\leq j\leq q(\e)$, and~${\beta\in H^{r-1}(\bT)}$, 
we need to prove that the terms  $F_1$ and $F_2$ 
can be estimated, for any given arbitrary $\theta>0$, by~$\theta\|\pi_j^\e\beta\|_{H^{r-1}}+K\|\beta\|_{H^{r'-1}}$, uniformly in $\rho\in\cV_{r,M}$, $1\leq j\leq q(\e)$, and~$\beta\in H^{r-1}(\bT)$, 
provided that~${\e\in(0,1)}$ is  sufficiently small.

Since $F_1$ can be estimated similarly as the terms $E_2$ and $E_3$ above and the estimate for~$F_2$ is provided in Lemma~\ref{L:AB3} below, 
the desired claim~\eqref{locbnm} with $r\in(3/2,2)$ follows from~\eqref{TES1}-\eqref{TES3} and~\eqref{TES4}.

Let now $r=2$. Using the identity \eqref{identbah}, Lemma~\ref{L:C1} (with $r=r'$  therein) and   Lemma~\ref{L:C2}  yield
\begin{align*}
 \|(\chi_j^\e E_4)'\|_2&\leq C\Big\{\|\chi_j^\e\sfA_{n,m}(\rho)[\pi_j^\e\beta]\|_{{\rm C}^1}+\|\chi_j^\e\sfA_{0,1}(1)[\pi_j^\e\beta]\|_{{\rm C}^1}\\
&\hspace{1cm}+\Big\|\Big[\chi_j^\e\Big(\sfH_{n,m}(\rho)-\frac{\rho'^n}{(2\omega_\rho)^{2m}}(\tau_j^\e) \sfH_{0,1}(1)\Big)[\pi_j^\e\beta]\Big]'\Big\|_{2}\Big\}\\
&\leq C\Big\| \chi_j^\e\Big(\sfH_{n,m}(\rho)-\frac{\rho'^n}{(2\omega_\rho)^{2m}}(\tau_j^\e) \sfH_{0,1}(1)\Big)[(\pi_j^\e\beta)']\Big\|_{2}+K\|\beta\|_{H^{r'-1}}, 
\end{align*}
and the first term on the right-hand side can be estimated, after some algebraic manipulations, up to a  multiplicative constant $C>0$,  by  a finite sum of terms of the form
\begin{align*} 
\wt F_1&:=\|\chi_j^\e  (\rho-\rho(\tau_j^\e))\sfH_{0,\wt m}(\rho)[\rho^{k}(\pi_j^\e\beta)']\|_2+\|\chi_j^\e \sfH_{0,\wt m}(\rho)[(\rho-\rho(\tau_j^\e))\rho^{k}(\pi_j^\e\beta)']\|_2\big),\\
 \wt F_2&:=\|\chi_j^\e \big(\sfH_{ \wt n+1, \wt m}(\rho)-\rho'(\tau_j^\e)\sfH_{\wt n,\wt m}(\rho)\big)[(\pi_j^\e\beta)']\|_2,
\end{align*}
with  $k\in\{0,1\}$ and $\wt n,\,\wt m\in\N$. Hence, in order to establish~\eqref{TES4} in the case $r=2$, it remains to prove that the terms  $\wt F_1$ and $\wt F_2$ 
can be estimated, for any given arbitrary~$\theta>0$, by~$\theta\|\pi_j^\e\beta\|_{H^{1}}+K\|\beta\|_{H^{r'-1}}$, uniformly in $\rho\in\cV_{r,M}$, $1\leq j\leq q(\e)$, and~$\beta\in H^{1}(\bT)$, 
for each sufficiently small~${\e\in(0,1)}$. Since the estimate for $\wt F_1$ is an immediate consequence of Lemma~\ref{L:C1}~(i)   and the estimate for~$\wt F_2$ is established in Lemma~\ref{L:AB3} below, 
the  claim~\eqref{locbnm} with $r=2$ follows from~\eqref{TES1}-\eqref{TES3} and~\eqref{TES4}.
 		\end{proof}
		
		We conclude this section by establishing  some estimates used in the proof of Lemma~\ref{L:AB2}.
		\begin{lemma}\label{L:AB3}
		Let $n,\,m\in\N $ and~${\theta>0}$ be fixed. 
\begin{itemize}
\item[(i)] Let $r\in(3/2,2)$. Then, for each sufficiently small~${\e\in(0,1)}$, there exists a positive constant~$K=K(\e,M)$
		 such that for all~$\rho\in\cV_{r,M}$,  ${\beta\in H^{r-1}(\bT)}$, and~$1\leq j\leq q(\e)$ it holds that
		\begin{equation}\label{lochold}
		\qquad \big[\chi_j^\e\big(\sfB_{ n+1, m}^0(\rho)-2\rho'(\tau_j^\e)\sfB_{ n, m}^0(\rho)\big)[\pi_j^\e\beta]\big]_{H^{r-1}}
			\leq \theta \|\pi_j^\e \beta\|_{H^{r-1}}+K\|\beta\|_{H^{r'-1}}.
		\end{equation}
\item[(ii)] Let $r=2$. Then, for each  sufficiently small~${\e\in(0,1)}$, there is  a  constant~${K=K(\e,M)>0}$
		 such that for all~$\rho\in\cV_{r,M}$,  ${\beta\in H^{1}(\bT)}$, and~$1\leq j\leq q(\e)$ it holds that
		\begin{equation}\label{locholdr2}
		 \|\chi_j^\e \big(\sfH_{ n+1, m}(\rho)-\rho'(\tau_j^\e)\sfH_{n,m}(\rho)\big)[(\pi_j^\e\beta)']\|_2
			\leq \theta \|\pi_j^\e \beta\|_{H^{1}}.
		\end{equation}
 \end{itemize}		  
	\end{lemma}
	\begin{proof}
	To prove (i), we infer from Lemma~\ref{L:C2} and the identity \eqref{identbah}  that
\begin{equation}\label{fcba}
	\begin{aligned}
		&\big[\chi_j^\e\big(\sfB_{ n+1, m}^0(\rho)-2\rho'(\tau_j^\e)\sfB_{ n, m}^0(\rho)\big)[\pi_j^\e\beta]\big]_{H^{r-1}}\\
		&\leq C\big(\|\chi_j^\e\sfA_{ n+1, m}(\rho)[\pi_j^\e\beta]\big\|_{{\rm C}^1}+\|\chi_j^\e\sfA_{ n, m}(\rho)\big)[\pi_j^\e\beta]\|_{{\rm C}^1}\big)\\
		&\quad+C\big[\chi_j^\e\big(\sfH_{ n+1, m}(\rho)-\rho'(\tau_j^\e)\sfH_{ n, m}(\rho)\big)[\pi_j^\e\beta]\big]_{H^{r-1}}\\
			&\leq C\big[\chi_j^\e\big(\sfH_{ n+1, m}(\rho)-\rho'(\tau_j^\e)\sfH_{ n, m}(\rho)\big)[\pi_j^\e\beta]\big]_{H^{r-1}}+K\|\beta\|_{H^{r'-1}},
		\end{aligned}
		\end{equation}
		and it remains to estimate the $H^{r-1}$-seminorm  of  $E:=\chi_j^\e\big(\sfH_{ n+1, m}(\rho)-\rho'(\tau_j^\e)\sfH_{ n, m}(\rho)\big)[\pi_j^\e\beta]$ by the right-hand side of \eqref{lochold}.
		To this end we compute for $\xi\in\R$, that
	\[
	\sfT_\xi E-E:=E_1+E_2+E_3-E_4,
\]
	where, with $\ov u:=\sfT _\xi u$ for  $u:\R\to\R$, we set
	\begin{align*}
	E_1&:=(\ov{\chi_j^\e}-\chi_j^\e)\big(\sfH_{n+1,m}(\ov\rho)-\rho'(\tau_j^\e)\sfH_{n,m}(\ov\rho)\big)[\ov{\pi_j^\e\beta}],\\
	E_2&:=\chi_j^\e H_{n+1,m}(\rho,\ldots,\rho)[\rho,\ldots,\rho,\rho-\rho'(\tau_j^\e){\rm id}_\R, \ov{\pi_j^\e\beta}-\pi_j^\e\beta],\\
	E_3&:=\chi_j^\e H_{n+1,m}(\ov \rho,\ldots,\ov\rho)[\ov\rho,\ldots,\ov\rho,\ov \rho-\rho, \ov{\pi_j^\e\beta}]\\
&\quad +\chi_j^\e\sum_{i=0}^{n-1}H_{n+1,m}(\rho,\ldots,\rho)[\underset{i}{\underbrace{\rho,\ldots,\rho}}, \ov\rho,\ldots,\ov\rho,\ov \rho-\rho,\rho-\rho'(\tau_j^\e){\rm id}_\R,\ov{\pi_j^\e\beta}]\\
&\quad-4\chi_j^\e\sum_{i=0}^{m-1}H_{n+3,m+1}(\underset{i+1}{\underbrace{\rho,\ldots,\rho}},\ov\rho,\ldots,\ov \rho)[ \ov\rho,\ldots,\ov\rho,\rho-\rho'(\tau_j^\e){\rm id}_\R,\ov\rho+\rho,\ov\rho-\rho, \ov{\pi_j^\e\beta}],\\
	E_4&:=\chi_j^\e\sum_{i=0}^{m-1}
	(\ov\rho^2-\rho^2)H_{n+1,m+1}(\underset{i+1}{\underbrace{\rho,\ldots,\rho}},\ov\rho,\ldots,\ov \rho)[ \ov\rho,\ldots,\ov\rho,\rho-\rho'(\tau_j^\e){\rm id}_\R,\ov{\pi_j^\e\beta}]\\
	&\hspace{1.65cm}+(\ov\rho+\rho)H_{n+1,m+1}(\underset{i+1}{\underbrace{\rho,\ldots,\rho}},\ov\rho,\ldots,\ov \rho)[ \ov\rho,\ldots,\ov\rho,\rho-\rho'(\tau_j^\e){\rm id}_\R,(\ov\rho-\rho)\ov{\pi_j^\e\beta}]\\
	&\hspace{1.65cm}+(\ov\rho-\rho)H_{n+1,m+1}(\underset{i+1}{\underbrace{\rho,\ldots,\rho}},\ov\rho,\ldots,\ov \rho)[ \ov\rho,\ldots,\ov\rho,\rho-\rho'(\tau_j^\e){\rm id}_\R,(\ov\rho+\rho)\ov{\pi_j^\e\beta}]\\
		&\hspace{1.65cm}+H_{n+1,m+1}(\underset{i+1}{\underbrace{\rho,\ldots,\rho}},\ov\rho,\ldots,\ov \rho)[ \ov\rho,\ldots,\ov\rho,\rho-\rho'(\tau_j^\e){\rm id}_\R,(\ov\rho^2-\rho^2)\ov{\pi_j^\e\beta}].
	\end{align*}
	Recalling Lemma~\ref{L:C1}, we have
	\begin{equation}\label{fcb1}
	\|E_1\|_2+\|E_4\|_2\leq C(\|\ov\rho-\rho\|_{H^1}+\|\ov\chi_j^\e-\chi_j^\e\|_{H^1})\|\beta\|_2.
	\end{equation}
	Concerning $E_2$, for $|\xi|\geq\e$ we infer from Lemma~\ref{L:C1}~(i) that
	\begin{equation}\label{fcb2}
	\|E_2\|_2\leq C \|\beta\|_2.
	\end{equation}
 For $|\xi|<\e$ we have, recalling \eqref{pi_j}-\eqref{chi_j} and that $E_2$ is  $2\pi$-periodic,
	\[
	\|E_2\|_{2}=\|E_2\|_{L_2((\tau_j^\e-\pi,\tau_j^\e+\pi))}
	\]
	Let $F_j:\R\to\R$ be the Lipschitz continuous function satisfying~$F_j=\rho$ on $J_j^\e=[\tau_j^\e-2\e,\tau_j^\e+2\e]$ and~${F_j'=\rho'(\tau_j^\e)}$ in~$\R\setminus J_j^\e$.
	 Then, for all  $|\xi|<\e$ and $\tau\in (\tau_j^\e-\pi,\tau_j^\e+\pi)$, it holds that
	 \begin{equation}\label{eq:vcv}
	 E_2(\tau)=\chi_j^\e H_{n+1,m}(\rho,\ldots,\rho)[\rho,\ldots,\rho,F_j-\rho'(\tau_j^\e){\rm id}_\R, \ov{\pi_j^\e\beta}-\pi_j^\e\beta](\tau).
	 \end{equation}
	 Indeed, we only need to establish~\eqref{eq:vcv} for \(\tau \in J_j^\e\), in which case~\(\rho(\tau) = F_j(\tau)\).
Moreover, if~\(\tau - s \in J_j^\e\), then \(\rho(\tau - s) = F_j(\tau - s)\), and thus \eqref{eq:vcv} holds.
Conversely, if \(\tau - s \not\in J_j^\e\), then for all~\({|\xi| < \e}\) we get \(\xi + \tau - s \not\in I_j^\e\).
Additionally, since \(\tau \in J_j^\e\), \(|s| < \pi\), and \(|\xi| < \e\) (with \(\e\) sufficiently small), it follows that  
\[
\xi + \tau - s \in \left(\tau_j^\e - \frac{3\pi}{2}, \tau_j^\e + \frac{3\pi}{2} \right)\qquad\text{and}\qquad
\operatorname{supp} \pi_j^\e \cap \left(\tau_j^\e - \frac{3\pi}{2}, \tau_j^\e + \frac{3\pi}{2} \right) = I_j^\e.
\]  
Thus, if \(\tau - s \not\in J_j^\e\), it follows that \(\pi_j^\e(\tau - s) = \pi_j^\e(\xi + \tau - s) = 0\), ensuring that \eqref{eq:vcv} remains valid in this case as well.
	
	Using Lemma~\ref{L:C1}~(i), for $|\xi|<\e$ we thus have, since $\rho'\in{\rm C}^{r-3/2}(\bT)$,
	 \begin{equation}\label{fcb3}
	\|E_2\|_2\leq C \|\rho'-\rho'(\tau_j^\e)\|_{L_\infty( (\tau_j^\e-2\e,\tau_j^\e+2\e))}\|\ov{\pi_j^\e\beta}-\pi_j^\e\beta\|_2\leq C\e^{r-3/2}\|\ov{\pi_j^\e\beta}-\pi_j^\e\beta\|_2.
	\end{equation}
	Finally, applying Lemma~\ref{L:C1}~(ii), we get 
	 \begin{equation}\label{fcb4}
	\|E_3\|_2\leq C \|\ov\rho -\rho\|_{H^1}\|\beta\|_2.
	\end{equation}
	Combining~\eqref{fcb1}, \eqref{fcb2}, \eqref{fcb3}, and \eqref{fcb4} we conclude that
	\begin{align*}
	[E]_{H_{r-1}}^2&=\int_{-\pi}^\pi\frac{\|\sfT_\xi E-E\|_2^2}{|\xi|^{1+2 (r-1)}}\,{\rm d}\xi\leq K\|\beta\|_2^2
	+C\e^{2(r-3/2)}\int_{-\e}^\e\frac{\|\sfT_\xi(\pi_j^\e\beta)-\pi_j^\e\beta\|_2^2}{|\xi|^{1+2 (r-1)}}\,{\rm d}\xi\\
	&\leq \theta^2[\pi_j^\e\beta]_{H^{r-1}}^2+K\|\beta\|_2^2
	\end{align*}
	for sufficiently small $\e$, and the desired claim follows now in virtue of~\eqref{fcba}.

	To prove (ii) we note, with $F_j:\R\to\R$ denoting the Lipschitz continuous function defined above, that
\begin{align*}
&\chi_j^\e \big(\sfH_{ n+1, m}(\rho)-\rho'(\tau_j^\e)\sfH_{n,m}(\rho)\big)[(\pi_j^\e\beta)']\\
&=\chi_j^\e   H_{ n+1, m}(\rho,\ldots,\rho)[\rho,\ldots,\rho,F_j-\rho'(\tau_j^\e){\rm id}_\R,(\pi_j^\e\beta)'],
\end{align*}   
and the desired claim \eqref{locholdr2} follows  from Lemma~\ref{L:C1}~(i) by arguing as in~\eqref{fcb3}.
	\end{proof}

\medskip



\bibliographystyle{siam}
\bibliography{Literature}
\end{document}